\documentclass{amsart}
\usepackage{amssymb,amsmath,amsfonts,stmaryrd}
\usepackage[T1]{fontenc}
\usepackage{lmodern}
\usepackage[utf8]{inputenc}
\usepackage{microtype}
\usepackage{xspace}
\usepackage[usenames, dvipsnames]{xcolor}

\numberwithin{equation}{section}

\usepackage[margin=1.3in]{geometry}
\usepackage{amsthm}
\usepackage{mathtools}
\DeclarePairedDelimiter{\set}{\{}{\}}
\DeclarePairedDelimiter{\ang}{\langle}{\rangle}
\usepackage{adamsmacros}
\usepackage{tikz-cd}
\usepackage{quiver}
\usepackage{subcaption}
\usepackage{setspace}
\usetikzlibrary{calc}

\usepackage[mathscr]{eucal}
\makeatletter
\def\instring#1#2{TT\fi\begingroup
  \edef\x{\endgroup\noexpand\in@{#1}{#2}}\x\ifin@}
\def\isuppercase#1{%
  \instring{#1}{ABCDEFGHIJKLMNOPQRSTUVWXYZ}%
}%
\newcommand{\C@lIfUpper}[1]{
 \if\isuppercase{#1}\mathscr{#1}%
 \else #1%
 \fi
}
\newcommand{\cat}[1]{\mathit{\@tfor\next:=#1\do{\C@lIfUpper{\next}}}}
\makeatother

\usepackage{xparse}
\newcommand{\Z}{\mathbb Z}

\newcommand{\U}{\mathrm U}
\newcommand{\Spin}{\mathrm{Spin}}
\newcommand{\String}{\mathrm{String}}
\newcommand{\KO}{\mathit{KO}}
\newcommand{\KU}{\mathit{KU}}
\newcommand{\TMF}{\mathit{TMF}}
\newcommand{\Tmf}{\mathit{Tmf}}
\newcommand{\Sph}{\mathbb S}

\newcommand{\MTSpin}{\mathit{MTSpin}}
\newcommand{\MTPin}{\mathit{MTPin}}
\newcommand{\MTString}{\mathit{MTString}}
\newcommand{\ko}{\mathit{ko}}
\newcommand{\ku}{\mathit{ku}}
\newcommand{\tmf}{\mathit{tmf}}

\newcommand{\Sq}{\mathrm{Sq}}
\newcommand{\abs}[1]{\lvert #1 \rvert}

\newcommand{\spinc}{spin\textsuperscript{$c$}\xspace}

\newcommand{\SU}{\mathrm{SU}}
\newcommand{\Sp}{\mathrm{Sp}}

\newcommand{\SO}{\mathrm{SO}}
\newcommand{\PSO}{\mathrm{PSO}}
\newcommand{\SL}{\mathrm{SL}}
\newcommand{\GL}{\mathrm{GL}}

\AtBeginDocument{
	\renewcommand{\O}{\mathrm O}
}
\newcommand{\id}{\mathrm{id}}
\newcommand{\RP}{\mathbb{RP}}
\newcommand{\bl}{\text{--}}
\newcommand{\R}{\mathbb R}
\newcommand{\Det}{\mathrm{Det}}
\newcommand{\Pin}{\mathrm{Pin}}

\newcommand{\cN}{\mathcal N}

\newcommand{\SH}{\mathit{SH}}
\newcommand{\SK}{\mathit{SK}}

\newcommand{\term}{\emph}

\newcommand{\cA}{\mathcal A}
\newcommand{\cB}{\mathcal B}
\newcommand{\cE}{\mathcal E}
\newcommand{\cP}{\mathcal P}
\newcommand{\Ext}{\mathrm{Ext}}
\newcommand{\Tor}{\mathrm{Tor}}
\newcommand{\Hom}{\mathrm{Hom}}
\newcommand{\End}{\mathrm{End}}

\newcommand{\MTSO}{\mathit{MTSO}}
\newcommand{\MTO}{\mathit{MTO}}

\newcommand{\Map}{\mathrm{Map}}

\newcommand{\wG}{\widetilde G}

\newcommand\MAILTO[1]{\href{mailto:#1}{\nolinkurl{#1}}}

\usepackage[backref=page, colorlinks=true]{hyperref}
\hypersetup{
 linkcolor={red!50!black},
 citecolor={green!50!black},
 urlcolor={blue!80!black}
}
\usepackage[capitalize, noabbrev]{cleveref} 

\newtheorem{thm}[equation]{Theorem}
\newtheorem*{thm*}{Theorem}
\newtheorem{lem}[equation]{Lemma}
\newtheorem{prop}[equation]{Proposition}
\newtheorem{cor}[equation]{Corollary}
\theoremstyle{definition}
\newtheorem{defn}[equation]{Definition}
\newtheorem{exm}[equation]{Example}
\theoremstyle{remark}
\newtheorem{rem}[equation]{Remark}
\newtheorem{question}[equation]{Question}

	\crefname{thm}{Theorem}{Theorems}
	\crefname{lem}{Lemma}{Lemmas}
	\crefname{cor}{Corollary}{Corollaries}
	\crefname{exm}{Example}{Examples}
	\Crefname{thm}{Theorem}{Theorems}
	\Crefname{lem}{Lemma}{Lemmas}
	\Crefname{cor}{Corollary}{Corollaries}
	\Crefname{exm}{Example}{Examples}

\DeclareDocumentCommand{\shortexact}{s O{} O{} mmmm}{
\IfBooleanTF{#1}{ 
\begin{tikzcd}[ampersand replacement=\&]
	{1} \& {#4} \& {#5} \& {#6} \& {1#7}
	\arrow[from=1-1, to=1-2]
	\arrow["#2", from=1-2, to=1-3]
	\arrow["#3", from=1-3, to=1-4]
	\arrow[from=1-4, to=1-5]
\end{tikzcd}
}{ 
\begin{tikzcd}[ampersand replacement=\&]
	{0} \& {#4} \& {#5} \& {#6} \& {0#7}
	\arrow[from=1-1, to=1-2]
	\arrow["#2", from=1-2, to=1-3]
	\arrow["#3", from=1-3, to=1-4]
	\arrow[from=1-4, to=1-5]
\end{tikzcd}
}
}

\newcommand{\Mzero}[3]{
	\tikzpt{#1}{#2}{#3}{};
	\foreach \y in {2, 3, 5} {
		\tikzpt{#1}{#2 + \y}{}{};
	}
	\sqone(#1, #2 + 2);
	\sqtwoL(#1, #2);
	\sqtwoL(#1, #2 + 3);
}

\title{Adams spectral sequences for non-vector-bundle Thom spectra}
\date{\today}

\author{Arun Debray}
\address{Department of Mathematics, University of Kentucky,
719 Patterson Office Tower,
Lexington, KY 40506-0027}
\email{\href{mailto:a.debray@uky.edu}{a.debray@uky.edu}}

\author{Matthew Yu}
\address{Mathematical Institute, University of Oxford,  Woodstock Road, Oxford, UK}
\email{\href{mailto:yumatthew70@gmail.com}{yumatthew70@gmail.com}}

\thanks{It is a pleasure to thank
Andrew Baker,
Ivano Basile,
Jonathan Beardsley,
Bob Bruner,
Matilda Delgado,
Sanath Devalapurkar,
Dan Isaksen,
Theo Johnson-Freyd,
Cameron Krulewski,
Miguel Montero,
Natalia Pacheco-Tallaj,
Luuk Stehouwer,
Mayuko Yamashita, and the winter 2023 electronic Computational Homotopy Theory reading seminar
for helpful discussions related to this work. Part of this project was completed while AD visited the Perimeter
Institute for Theoretical Physics; research at Perimeter is supported by the Government of
Canada through Industry Canada and by the Province of Ontario through the Ministry
of Research \& Innovation.}

\begin{document}
\maketitle

\begin{abstract}
When $R$ is one of the spectra $\ku$, $\ko$, $\tmf$, $\MTSpin^c$, $\MTSpin$, or $\MTString$, there is a standard approach to computing twisted $R$-homology groups of a space $X$ with the Adams spectral sequence, by using a change-of-rings isomorphism to simplify the $E_2$-page. This approach requires the assumption that the twist comes from a vector bundle, i.e.\ the twist map $X\to B\GL_1(R)$ factors through $B\O$. We show this assumption is unnecessary by working with Baker-Lazarev's Adams spectral sequence of $R$-modules and computing its $E_2$-page for a large class of twists of these spectra. We then work through two example computations motivated by anomaly cancellation for supergravity theories.
\end{abstract}

\tableofcontents


\setcounter{section}{-1}
\addtocontents{toc}{\protect\setcounter{tocdepth}{1}}
\section{Introduction}
There is a standard formula for computing Steenrod squares in the cohomology of the Thom space or spectrum of a vector bundle $V\to X$: if $U$ is the Thom class,
\begin{equation}
\label{vb_steen}
    \Sq^n(Ux) = \sum_{i+j=n}Uw_i(V)\Sq^j(x).
\end{equation}
The ubiquity of the Steenrod algebra in computational questions in algebraic topology means this formula has been
applied to questions in topology and geometry, and recently even in physics, where it is used to run the Atiyah-Hirzebruch and Adams spectral sequences computing groups of invertible field theories.

It is possible to build Thom spectra using more general data than vector bundles, and recently these Thom spectra
have appeared in questions motivated by anomaly cancellation in supergravity theories~\cite{DY22, Deb23}. Motivated
by these applications (which we discuss more in \S\ref{section:applications}), our goal in this paper is to
understand the analogue of~\eqref{vb_steen} for non-vector-bundle twists of commonly studied generalized cohomology
theories. We found that the most direct generalization of~\eqref{vb_steen} is true; in a sense, for the theories we study, these more general Thom spectra behave just like vector bundle Thom spectra for the purpose of computing their homotopy groups with the Adams spectral sequence.


\subsection*{Statement of results}
Now for a little more detail: our main theorem and the language needed to define it. We use
Ando-Blumberg-Gepner-Hopkins-Rezk's approach to twisted generalized cohomology theories~\cite{ABGHR14a, ABGHR14b},
which generalizes the notion of a local system. Twists of $\Z$-valued cohomology on a pointed, connected space $X$
are specified by \term{local systems} with fiber $\Z$, which are equivalent data to homomorphisms $\pi_1(X)\to
\mathrm{Aut}(\Z)$, or, since $\Z$ is discrete, to maps $X\to B\mathrm{Aut}(\Z)$.

Ando-Blumberg-Gepner-Hopkins-Rezk generalize this to $E_\infty$-ring spectra.\footnote{An \term{$E_\infty$-ring
spectrum} is the avatar in stable homotopy theory of a generalized cohomology theory with a commutative ring
structure.  Examples include ordinary cohomology, real and complex $K$-theory, and many cobordism
theories.}\textsuperscript{,}\footnote{In fact, much of Ando-Blumberg-Gepner-Hopkins-Rezk's theory works in greater
generality, but we only need $E_\infty$-ring spectra in this article.} If $R$ is an $E_\infty$-ring spectrum,
Ando-Blumberg-Gepner-Hopkins-Rezk define a notion of local system of free rank-$1$
$R$-module spectra that is classified by maps to an object called $B\GL_1(R)$, making $B\GL_1(R)$ the classifying
space for twists of $R$-homology. Given a twist $f\colon X\to B\GL_1(R)$, they then define a \term{Thom spectrum}
$Mf$, and the homotopy groups of $Mf$ are the $f$-twisted $R$-homology groups of $X$. This construction
simultaneously generalizes twisted ordinary homology, twisted $K$-theory, and the vector bundle twists mentioned
above.

We are interested in twisted $R$-homology for several $E_\infty$-ring spectra, so our first step is to give
examples of twists. Most of these examples are known, but by using a result of May-Quinn-Ray~\cite[Lemma IV.2.6]{MQRT77}, one can produce them in a unified way.

\begin{thm*}\hfill
\begin{enumerate}
	\item (\cref{splitting_spinc,KU_comparison}) There is a map $K(\Z/2, 1)\times K(\Z, 3)\to B\GL_1(\MTSpin^c)$, meaning
	\spinc bordism can be twisted on a space $X$ by $H^1(X; \Z/2)\times H^3(X;\Z)$. The induced maps to
	$B\GL_1(\ku)$ and $B\GL_1(\KU)$ recover the usual notion of $K$-theory twisted by $H^1(X; \Z/2)\times
	H^3(X;\Z)$.
	\item~(\cref{splitting_spin,KO_comparison}) There is a map $K(\Z/2, 1)\times K(\Z/2, 2)\to B\GL_1(\MTSpin)$, meaning
	spin bordism can be twisted on a space $X$ by $H^1(X; \Z/2)\times H^2(X;\Z/2)$. The induced maps to
	$B\GL_1(\ko)$ and $B\GL_1(\KO)$ recover the usual notion of $\KO$-theory twisted by $H^1(X; \Z/2)\times
	H^2(X;\Z/2)$.
	\item (\cref{splitting_string,tmf_comparison}) Let $\SK(4)$ be a classifying space for degree-$4$ supercohomology classes. Then
	there is a map $K(\Z/2, 1)\times \SK(4)\to B\GL_1(\MTString)$, meaning
	string bordism can be twisted on a space $X$ by $H^1(X; \Z/2)\times \SH^4(X)$. The induced map
	\begin{equation}
		K(\Z, 4)\longrightarrow \SK(4)\longrightarrow B\GL_1(\MTString)\longrightarrow B\GL_1(\tmf)
	\end{equation}
	recovers the Ando-Blumberg-Gepner twist of $\tmf$ (and $\Tmf$ and $\TMF$) by degree-$4$ cohomology classes.
\end{enumerate}
\end{thm*}
\term{Supercohomology} refers to a generalized cohomology theory $\SH$ introduced by Freed~\cite[\S 1]{Fre08} and Gu-Wen~\cite{GW14}: $\pi_{-2}\SH = \Z/2$ and $\pi_0\SH = \Z$, with the unique nontrivial $k$-invariant, and no other nonzero homotopy groups. We explicitly define $\SK(4)$ in~\eqref{sk_defn}.

Though twists of $\tmf$ by degree-$4$ cohomology classes are relatively well-studied, this supercohomology
generalization appears to only be suggested at in the literature by various authors including~\cite{FHT10, JF20, BLM23, TY23ASD, TY21, TY25},
and sees more of the homotopy type of $B\GL_1(\tmf)$. It would be interesting to study instances of this twist.

We call the twists in the above theorem \term{fake vector bundle twists}: when the twist is given by a vector bundle $V$, these cohomology classes appear as characteristic classes of $V$, but these twists exist whether or not there is a vector bundle with the prescribed characteristic classes.

If $R$ is one of the spectra mentioned in the above theorem, the Thom spectrum $Mf$ of a fake vector bundle twist $f\colon X\to B\GL_1(R)$ is an $R$-module spectrum. This grants us access to Baker-Lazarev's variant of the Adams spectral sequence~\cite{baker2001adams}.
\begin{thm*}[Baker-Lazarev~\cite{baker2001adams}]
Let $p$ be a prime number and $R$ be an $E_\infty$-ring spectrum such that $\pi_0(R)$ surjects onto $\Z/p$, so that
$H \coloneqq H\Z/p$ acquires the structure of an $R$-algebra. For $R$-module spectra $M$ and $N$, let $N_R^*M\coloneqq \pi_{-*}\mathrm{Map}_R(M, N)$. Then there is an Adams-type spectral sequence with signature
\begin{equation}
    E_2^{s,t} = \Ext_{H_R^*H}^{s,t}(H_R^*(M), \Z/p) \Longrightarrow \pi_{t-s}(M)_p^\wedge,
\end{equation}
which converges for all $M$ and all $E_\infty$-ring spectra $R$ we consider in this paper.
\end{thm*}
What Baker-Lazarev prove is more general than what we state here: we stated only the generality we need.

For $H\Z$, $\ko$, and $\ku$ ($p = 2$), and $\tmf$ ($p = 2$ and $p = 3$), $H_R^*H$ is known due to work of various
authors: let $\cA(n)$ be the subalgebra of the mod $2$ Steenrod algebra generated by $\Sq^1,\dotsc,\Sq^{2^n}$. Then, at $p = 2$,
\begin{enumerate}
    \item $H_{H\Z}^*H\cong\cA(0)$,
    \item $H_{\ku}^*H\cong\cE(1)\coloneqq\ang{\Sq^1, \Sq^2\Sq^1 + \Sq^1\Sq^2}$,
    \item $H_{\ko}^*H\cong\cA(1)$, and
    \item $H_{\tmf}^*H\cong\cA(2)$.
\end{enumerate}
See~\eqref{eq:commonCOR} and the surrounding text.
For $\tmf$ at $p = 3$, see \cref{intro_Atmf}. These algebras are small enough for computations to be tractable, so
if we can compute the $H_R^*H$-module structure on $H_R^*(Mf)$ for $f$ a fake vector bundle twist, we can run the
Adams spectral sequence and hope to compute $\pi_*(Mf)$. This is the content of our main theorem,
\cref{thom_module_calc}.

The first step is to understand $H_R^*(Mf)$ as a vector space. In \cref{Rmod_Thom_iso}, we establish a Thom isomorphism
\begin{equation}
    H_R^*(Mf)\overset\cong\longrightarrow H^*(X;\Z/2)\cdot U,
\end{equation}
where $U\in H^0_R(Mf)$ is the Thom class. Using this, we can state our main theorem:
\begin{thm*}[\Cref{thom_module_calc}]
Let $X$ be a topological space. 
\begin{enumerate}
    \item Given $a \in H^1(X;\Z/2)$ and $c \in H^3(X; \Z)$, let $f_{a,c}: X \rightarrow B\GL_1(\ku)$ be the corresponding fake vector bundle twist.   $H^*_{\ku}(M^{\ku}f_{a,c})$ is a $\mathcal{E}(1)$-module with $Q_0$-and $Q_1$-actions by
    \begin{equation*}
    \begin{aligned}
		Q_0(Ux) &\coloneqq Uax + U Q_0(x)\\
		Q_1(Ux) &\coloneqq U(c\bmod 2+a^3)x + UQ_1(x).
	\end{aligned}
    \end{equation*}
    \item\label{pt:thm_ko} Given $a\in H^1(X; \Z/2)$ and $b \in H^2(X; \Z/2)$, let $f_{a,b}\colon X\to B\GL_1(\ko)$ be the
	corresponding fake vector bundle twist. $H^*_{\ko}(M^{\ko}f_{a,b})$ is an $\mathcal{A}(1)$-module with $\Sq^1$-and $\Sq^2$-actions
	\begin{equation*}
	\begin{aligned}
		\Sq^1(Ux) &\coloneqq U(ax + \Sq^1(x))\\
		\Sq^2(Ux) &\coloneqq U(bx + a\Sq^1(x) + \Sq^2(x)).
	\end{aligned}
	\end{equation*}
	\item Given $a\in H^1(X;\Z/2)$, and $d\in SH^4(X)$, let $f_{a,d}\colon X\to
	B\GL_1(\tmf)$ be the corresponding fake vector bundle twist. $H_\tmf^*(M^\tmf f_{a,d})$ is an $\cA(2)$-module with $\Sq^1$-and $\Sq^2$-action the same as (\ref{pt:thm_ko}) above, and $\Sq^4$-action
	\begin{equation*}
			\Sq^4(Ux) = U(\delta x + (t(d) a+\Sq^1(t(d))) \Sq^1(x) + t(d)\Sq^2(x) + a\Sq^3(x) + \Sq^4(x)).
	\end{equation*}
	Furthermore, $H^*_{\tmf}(M^{\tmf}f_{0,d};\Z/3)$ is an $\cA^\tmf$-module with $\beta$ and $\mathcal{P}^1$ actions
	\begin{equation*}
		\begin{aligned}
		\beta(Ux) &\coloneqq U\beta(x)\\
		\cP^1(Ux) &\coloneqq U((d\bmod 3)x + \cP^1(x)).
	\end{aligned}
	\end{equation*} 
\end{enumerate}
\end{thm*}
This theorem computes the inputs to the Baker-Lazarev Adams spectral sequences for $H\Z$, $\ku$, $\ko$, and $\tmf$ for the Thom spectra we study. We find three avatars of this fact:
\begin{enumerate}
    \item In \cref{little_sibling_Adams_corollary}, we describe in all degrees the $E_2$-page of the Baker-Lazarev Adams spectral sequence for a fake vector bundle twist of $H\Z$, $\ku$, $\ko$, or $\tmf$.
    \item In \cref{big_sibling_Adams_estimate}, we describe in low degrees the $E_2$-page of the Baker-Lazarev Adams spectral sequence for a fake vector bundle twist of $\MTSO$, $\MTSpin^c$, $\MTSpin$, or $\MTString$.
    \item In \cref{big_sibling_Adams_corollary}, we describe variants of the Baker-Lazarev Adams spectral sequence for fake vector bundle twists of $\MTSO$, $\MTSpin^c$, and $\MTSpin$, and compute the $E_2$-pages in all degrees.
\end{enumerate}
We then give three examples of applications of our techniques.
\begin{enumerate}
    \item In \S\ref{sub:uduality}, we use \cref{big_sibling_Adams_estimate} to compute low-dimensional $G$-bordism groups for $G = \Spin\times_{\set{\pm 1}}\SU_8$. These are the twisted spin bordism groups for a twist over $B(\SU_8/\set{\pm 1})$ which is not a vector bundle twist. In~\cite{DY22}, we discussed an application of $\Omega_5^G$ to an anomaly cancellation question in $4$-dimensional $\cN = 8$ supergravity; using \cref{big_sibling_Adams_estimate}, we can give a much simpler calculation of $\Omega_5^G$ than appears in~\cite[Theorem 4.26]{DY22}. See Kuroda~\cite{Kur25} for more computations of twisted spin bordism groups using \cref{thom_module_calc}.
    \item In \S\ref{sub:stringbord}, we study twisted string bordism groups for a non-vector bundle twist over $B((E_8\times E_8)\rtimes\Z/2)$, where $\Z/2$ acts on $E_8\times E_8$ by swapping the factors. These bordism groups have applications in the study of the $E_8\times E_8$ heterotic string; see~\cite{Deb23} for more information. Here, we work through the $3$-primary calculation, simplifying a computation in~\cite{Deb23}.
    \item In \S\ref{sub:kumod}, we reprove a result of Devalapurkar~\cite[Remark 2.3.16]{Dev} describing $H\Z/2$ as a $\ku$-module Thom spectrum; Devalapurkar's proof uses different methods.
\end{enumerate}
Our theorems proceed similarly for several different families of spectra. One naturally wonders if there are more families out there. Specifically, there is a spectrum for which many but not all of the ingredients of our proofs were present at the time we wrote the first version of this paper.
\begin{question}[\cref{level_str}]
\label{level_qn}
Let $\tmf_1(3)$ denote the connective spectrum of topological modular forms with a level structure for the congruence subgroup $\Gamma_1(3)\subset\SL_2(\Z)$~\cite{HL16}. Is there a tangential structure $\xi\colon B\to B\O$ such that $\mathit{MT\xi}$ is an $E_\infty$-ring spectrum with an $E_\infty$-ring map $\mathit{MT\xi}\to \tmf_1(3)$ which is an isomorphism on low-degree homotopy groups after $2$-completion?

If such a spectrum exists, then one could use our approach to run the Baker-Lazarev Adams spectral sequence to compute twisted $\tmf_1(3)$-homology; the needed change-of-rings formula for $\tmf_1(3)$ is due to Mathew~\cite[Theorem 1.2]{Mat16}. 
\end{question}
Devalapurkar~\cite{Dev22} constructed a tangential structure called a \term{string$^h$ structure} and an $E_\infty$-ring map $\sigma_1(3)\colon \MTString^h\to\tmf_1(3)_{(2)}$, answering most of \cref{level_qn}; that any such orientation is an isomorphism on low-degree homotopy groups after $2$-completion was shown in~\cite[Corollary 3.53]{DY24}. We plan to study the Baker-Lazarev Adams spectral sequence for $\tmf_1(3)$-module Thom spectra in future work.



%
%
\subsection*{Outline}
\S\ref{section:reviewthom} is about twists and Thom spectra. First, in \S\ref{sub:ABGHR}, we review
Ando-Blumberg-Gepner-Hopkins-Rezk's theory of Thom spectra~\cite{ABGHR14a, ABGHR14b} and discuss some constructions
and lemmas we need later in the paper. Then, in \S\ref{sub:faketwist}, we construct fake vector bundle twists for
the four families of ring spectra that we study in this paper: $\MTSO$ and $H\Z$ in \S\ref{so_twists}; $\MTSpin^c$,
$\ku$, and $\KU$ in \S\ref{spinc_twists}; $\MTSpin$, $\ko$, and $\KO$ in \S\ref{spin_twists}; and $\MTString$,
$\tmf$, $\Tmf$, and $\TMF$ in \S\ref{string_twists}.

In \S\ref{s:BL} we study the Adams spectral sequence for the Thom spectra of these twists.  We begin in
\S\ref{VB_change_of_rings} by reviewing how the change-of-rings story simplifies Adams computations for vector
bundle Thom spectra. Then, in \S\ref{sub:BLASS}, we introduce Baker-Lazarev's $R$-module Adams spectral
sequence~\cite{baker2001adams}. In \S\ref{sub:mainthm} we prove \cref{thom_module_calc} computing the input to the
Baker-Lazarev Adams spectral sequence for the Thom spectra of our fake vector bundle twists.

We conclude in \S\ref{section:applications} with some applications and examples of computations using the main
theorem: a twisted spin bordism example in \S\ref{s:u_duality} and an application to U-duality anomaly
cancellation; a twisted string bordism example in \S\ref{sub:stringbord} motivated by anomaly cancellation in
heterotic string theory; and a twisted $\ku$-homology example in \S\ref{sub:kumod} exhibiting $H\Z/2$ as the
$2$-completion of a $\ku$-module Thom spectrum.

\addtocontents{toc}{\protect\setcounter{tocdepth}{2}}
\section{Thom spectra and twists à la Ando-Blumberg-Gepner-Hopkins-Rezk}\label{section:reviewthom}

\subsection{The Ando-Blumberg-Gepner-Hopkins-Rezk approach to Thom spectra}\label{sub:ABGHR}
In this subsection we introduce Ando-Blumberg-Gepner-Hopkins-Rezk's theory of Thom spectra~\cite{ABGHR14a,
ABGHR14b} and recall the key facts we need for our theorems.\footnote{Here and throughout the paper, we work with
the symmetric monoidal $\infty$-category of spectra constructed by Lurie~\cite[\S 1.4]{Lurie:HA}, where by
``$\infty$-category'' we always mean quasicategory. In \S\ref{s:BL}, we use work of
Baker-Lazarev~\cite{baker2001adams}, who work with a different model of spectra, the $\Sph$-modules of
Elmendorf-Kriz-Mandell-May~\cite[S 2.1]{EKMM97}. The equivalence between the $\infty$-category presented by the
model category of $\Sph$-modules and Lurie's $\infty$-category of spectra follows from work of
Mandell-May-Schwede-Shipley~\cite{MMSS:2001}, Schwede~\cite{Sch01}, and Mandell-May~\cite{MM02}. Likewise, these
papers show that commutative algebras in the category of $\Sph$-modules correspond to $E_\infty$-rings in the
$\infty$-category of spectra.
\label{footnote:spectra}} In this paper, we only need to work
with $E_\infty$-ring spectra, and we will state some theorems in only the generality we need, which is less general
than what Ando-Blumberg-Gepner-Hopkins-Rezk prove.

By an \term{$\infty$-group} we mean a grouplike $E_1$-space, which is a homotopically invariant version of topological group.
By an
\term{abelian $\infty$-group} we mean a grouplike $E_\infty$-space.
\begin{defn}[{May~\cite[\S III.2]{MQRT77}}]
Let $R$ be an $E_\infty$-ring spectrum. The \term{group of units} of $R$ is the abelian $\infty$-group $\GL_1(R)$
defined to be the following pullback:
\begin{equation}
\label{GL1_defn}
\begin{tikzcd}
	{\GL_1(R)} & {\Omega^\infty R} \\
	{\pi_0(R)^\times} & {\pi_0(R)}\,.
	\arrow[from=1-1, to=1-2]
	\arrow[from=1-1, to=2-1]
	\arrow[from=2-1, to=2-2]
	\arrow[from=1-2, to=2-2]
	\arrow["\lrcorner"{anchor=center, pos=0.125}, draw=none, from=1-1, to=2-2]
\end{tikzcd}
\end{equation}
\end{defn}
The pullback~\eqref{GL1_defn} takes place in the $\infty$-category of abelian $\infty$-groups.
As the three legs of the
pullback diagram~\eqref{GL1_defn} are functorial in $R$, $\GL_1(R)$ is also functorial in $R$.

Since $\GL_1(R)$ is an $\infty$-group, it has a classifying space $B\GL_1(R)$; we refer to a map $X\to B\GL_1(R)$
as a \term{twist} of $R$ over $X$. There is a sense in which $B\GL_1(R)$ carries the universal local system of
\term{$R$-lines}, or free $R$-module spectra of rank $1$: see~\cite[Corollary 2.14]{ABGHR14a}.
\begin{exm}
\label{discrete_example}
If $A$ is a commutative ring and $R = HA$, then the equivalence of abelian $\infty$-groups $\pi_0\colon
\Omega^\infty HA\overset\simeq\to A$ induces an equivalence of abelian $\infty$-groups $\GL_1(R)\simeq
A^\times$.
\end{exm}

Let $\cat{Mod}_R$ denote the $\infty$-category of $R$-module spectra and $\cat{Line}_R$ denote the
$\infty$-category of $R$-lines, and let $\pi_{\le\infty}(X)$ denote the
fundamental $\infty$-groupoid of a space $X$. The
identification $\abs{\cat{Line}_R}\overset\simeq\to B\GL_1(R)$~\cite[Corollary 2.14]{ABGHR14a} allows us to
reformulate the inclusion $\cat{Line}_R\hookrightarrow\cat{Mod}_R$ as a functor $M\colon
\pi_{\le\infty}(B\GL_1(R))\to \cat{Mod}_R$, which one can think of as sending a point in $B\GL_1(R)$ to the
$R$-line which is the fiber of the universal local system of $R$-lines on $B\GL_1(R)$. In the rest of this paper,
we will simply write $X$ for $\pi_{\le\infty}(X)$, as we will never be in a situation where this causes ambiguity.
\begin{defn}[{\cite[Definition 2.20]{ABGHR14a}}]
\label{thom_defn_not_eqn}
Let $R$ be an $E_\infty$-ring spectrum and $f\colon X\to B\GL_1(R)$ be a twist of $R$. The \term{Thom spectrum}
$M^Rf$ of the map $f$ is the colimit of the $X$-shaped diagram
\begin{equation}
\label{Thom_defn}
	X\overset{f}{\longrightarrow} B\GL_1(R)\longrightarrow \cat{Mod}_R\,.
\end{equation}
When $R$ is clear from context, we will write $Mf$ for $M^Rf$.
\end{defn}
By construction, $Mf$ is an $R$-module spectrum. If the reader is familiar with the definition of a Thom spectrum
associated to a virtual vector bundle, this definition is related but more general.

\begin{exm}[Thom spectra from vector bundles]
\label{VB_Thom}
Let $V\to X$ be a virtual stable vector bundle of rank zero; $V$ is classified by a map $f_V\colon X\to B\O$. There
is a map of abelian $\infty$-groups $J\colon B\O\to B\GL_1(\Sph)$ called the \term{$J$-homomorphism}, where $B\O$
has the abelian $\infty$-group structure induced by direct sum of (rank-zero virtual) vector bundles \cite{Whitehead42}. Theorems of Lewis~\cite[Chapter IX]{LMSM86} and
Ando-Blumberg-Gepner-Hopkins-Rezk~\cite[Corollary 3.24]{ABGHR14a} together imply that the Thom spectrum $X^V$ in
the usual sense is naturally equivalent to the Thom spectrum $M(J\circ f_V)$ in the
Ando-Blumberg-Gepner-Hopkins-Rezk sense.
\end{exm}
\begin{exm}[Trivial twists]
\label{trivial_twists}
Suppose that the map $f\colon X\to B\GL_1(R)$ is null-homotopic. Then by definition, the colimit
of~\eqref{Thom_defn} is $R\wedge X_+$; more precisely, a null-homotopy of $f$ induces an equivalence of $R$-module
spectra $Mf\simeq R\wedge X_+$.
\end{exm}
We will need the following fact a few times. We believe it is known, but were unable to locate a proof in this generality in the literature.
\begin{lem}
\label{Thom_change_of_rings}
Let $g\colon R_1\to R_2$ be a map of $E_\infty$-ring spectra and $f\colon X\to B\GL_1(R_1)$ be a twist. Then there
is an equivalence of $R_2$-module spectra
\begin{equation}
	M^{R_2}(g\circ f)\overset\simeq\longrightarrow M^{R_1}f\wedge_{R_1} R_2.
\end{equation}
\end{lem}
When $R_1 = \Sph$, Ando-Blumberg-Gepner-Hopkins-Rezk~\cite[\S 1.2]{ABGHR14b} mention that this lemma is a
straightforward consequence of a different, equivalent definition of the Thom spectrum~\cite[Definition
3.13]{ABGHR14b}.
\begin{proof}
We will show that the diagram
\begin{equation}
\label{Thom_mod}
\begin{tikzcd}
	{B\GL_1(R_1)} & {B\GL_1(R_2)} \\
	{\cat{Mod}_{R_1}} & {\cat{Mod}_{R_2}}
	\arrow["g", from=1-1, to=1-2]
	\arrow["M^{R_1}"', from=1-1, to=2-1]
	\arrow["{\text{--}\wedge_{R_1} R_2}"', from=2-1, to=2-2]
	\arrow["M^{R_2}", from=1-2, to=2-2]
\end{tikzcd}
\end{equation}
is (homotopy) commutative, where just as above we identify the spaces $B\GL_1(R_i)$ with their
fundamental $\infty$-groupoids. Once we know this, the lemma is immediate from the colimit definition of $M^{R_2}(g\circ
f)$ in \cref{thom_defn_not_eqn}: replace $M^{R_2}\circ g\circ f$ with $(\bl\wedge_{R_1}R_2)\circ M^{R_1}\circ f$.

The key obstacle in establishing commutativity of~\eqref{Thom_mod} is that $g\colon B\GL_1(R_1)\to B\GL_1(R_2)$
comes from maps of spectra via~\eqref{GL1_defn}, but $\bl\wedge_{R_1}R_2$ has a more module-theoretic flavor. The
resolution, which is the same as in the proof of~\cite[Proposition 2.9]{ABGHR14a}, is that the three other pieces
of the pullback~\eqref{GL1_defn} defining $\GL_1$, namely $\Omega^\infty$, $\pi_0$, and $\pi_0(\bl)^\times$, have
module-theoretic interpretations: there are homotopy equivalences of abelian $\infty$-groups $\Omega^\infty
R\overset\simeq\to \End_R(R)$, and likewise $\pi_0(R)\overset\simeq\to\pi_0(\End_R(R))$ and
$\pi_0(R)^\times\overset\simeq\to \pi_0(\End_R(R))^\times$. And all of these identifications are compatible with
the tensor product functor $\cat{Mod}_{R_1}\to\cat{Mod}_{R_2}$, thus also likewise for their classifying spaces,
establishing commutativity of~\eqref{Thom_mod}.
\end{proof}

The usual Thom diagonal for a Thom space $X^V$ gives $H^*(X^V;\Z/2)$ the structure of a module over
$H^*(X;\Z/2)$. One can generalize this for $R$-module Thom spectra as follows.
\begin{defn}[{Thom diagonal~\cite[\S 3.3]{ABGHR14b}}]
\label{def:ThomDiag}
Let $R$ be an $E_\infty$-ring spectrum and $f\colon X\to B\GL_1(R)$ be a twist. The \term{Thom diagonal} for $Mf$
is an $R$-module map
\begin{equation}
\label{TD_equation}
	Mf\overset{\Delta^t}{\longrightarrow} Mf\wedge R\wedge X_+
\end{equation}
defined by applying the Thom spectrum functor to the maps $f\colon X\to B\GL_1(R)$ and $(f, 0)\colon X\times X\to
B\GL_1(R)$: if $\Delta\colon X\to X\times X$ is the diagonal map, then $f = \Delta^*(f, 0)$, so $\Delta$ induces
the desired map $\Delta^t$ of $R$-module Thom spectra in~\eqref{TD_equation}.
\end{defn}
See Beardsley~\cite[\S 4.3]{Bea18} for a nice coalgebraic interpretation of the Thom diagonal.
\subsection{Constructing non-vector-bundle twists}\label{sub:faketwist}
Let $X$ and $Y$ be $E_\infty$-spaces and $f_1\colon X\to Y$ and $f_2\colon Y\to B\GL_1(\Sph)$ be $E_\infty$-maps. Ando-Blumberg-Gepner~\cite[Theorem 1.7]{ABG11} show that the $E_\infty$-structure on $f_2\circ f_1$ induces an $E_\infty$-ring structure on $M(f_2\circ f_1)$.
\begin{thm}[{May-Quinn-Ray~\cite[Lemma IV.2.6]{MQRT77}}]
\label{cofiber_twists}
Let $R$ be an $E_\infty$-ring spectrum. The data of an $E_\infty$-ring map $\rho\colon M(f_2\circ f_1)\to R$ induces a map
$T_{f_1,f_2}\colon Y/X\to B\GL_1(R)$ of abelian $\infty$-groups.\footnote{May-Quinn-Ray proves that $T_{f_1,f_2}$ is a map of \term{$\mathscr I_*$-functors}. This implies it is a map of $E_\infty$-spaces, as discussed in~\cite[\S I.1]{MQRT77}.}
\end{thm}
An $E_\infty$-ring map $\rho$ of this kind is often called an \term{$M(f_2\circ f_1)$-orientation} of $R$.
\begin{rem}
May-Quinn-Ray state this result only for $R = M(f_2\circ f_1)$ and $\rho = \id$; Beardsley~\cite[Theorem 1]{Bea17} provides another, quite different, proof of this case and uses it to obtain many commonly-studied twists of various cohomology theories. We will usually apply it for maps to $B\O$ and implicitly compose with the $E_\infty$-map $J\colon B\O\to B\GL_1(\Sph)$, like in \cref{VB_Thom}.

The more general version of May-Quinn-Ray's theorem appearing in~\cref{cofiber_twists} follows immediately from the version in~\cite{MQRT77}: the abelian $\infty$-group $B\GL_1(R)$ is natural in the $E_\infty$-ring spectrum $R$, so given $\rho\colon M(f_2\circ f_1)\to R$ as in the statement of \cref{cofiber_twists}, we may compose May-Quinn-Ray's map $Y/X\to B\GL_1(M(f_2\circ f_1))$ with the base change map $B\GL_1(M(f_2\circ f_1))\to B\GL_1(R)$ to finish.
\end{rem}
\begin{rem}
Just as the map $J\colon B\O\to B\GL_1(\Sph)$ is related to the classical $J$-homomorphism $\pi_*(\O)\to\pi_*(\Sph)$, the map $T_{f_1,f_2}\colon Y/X\to B\GL_1(M(f_2\circ f_1))$ from \cref{cofiber_twists} is related to Harris' generalized $J$-homomorphisms~\cite{Har69}: see May-Quinn-Ray~\cite[\S IV.2]{MQRT77}. These generalized $J$-homomorphisms also appear in work of Ray~\cite{Ray71, Ray74}, Gozman~\cite{Goz77}, and Bier-Ray~\cite{BR78}.
\end{rem}
\begin{thm}[Beardsley~{\cite[Theorem 1]{Bea17}}]
\label{universal_twist}
For $R = M(f_2\circ f_1)$, there is a natural equivalence 
$M^RT_{f_1,f_2}\overset\simeq\to M^{\Sph}f_2$.
\end{thm}
%
%
%
In this paper we consider twisted $R$-(co)homology for several different ring spectra $R$. These spectra are organized into several families: in each family there is a Thom spectrum $Mf$, another ring spectrum $R$, and a map of ring spectra $Mf\to R$ which is an isomorphism on homotopy groups in low degrees. In the context of a specific family, we will refer to $Mf$ as the \term{big sibling} and $R$ as the \term{little sibling}. The four families we consider in this paper are $(\MTSO, H\Z)$, $(\MTSpin, \ko)$, $(\MTSpin^c, \ku)$, and
$(\MTString, \tmf)$:\footnote{In the homotopy theory literature, it is common to refer to bordism spectra $\mathit{MSO}$, $\mathit{MSpin}$, etc., corresponding to the bordism groups of manifolds with orientations, resp.\ spin structures, on the stable normal bundle. In the mathematical physics literature, one sees $\MTSO$, $\MTSpin$, etc., corresponding to the same structures on the stable \emph{tangent} bundle. If $\xi\colon B\to B\O$ is a tangential structure such that the map $\xi$ is a map of abelian $\infty$-groups, as is the case for $\O$, $\SO$, $\Spin^c$, $\Spin$, and $\String$, there is a canonical equivalence $\mathit{M\xi}\overset{\simeq}{\to}\mathit{MT\xi}$. For other tangential structures, this is not necessarily true: in particular, $\mathit{MPin}^\pm\simeq\MTPin^\mp$.}
\begin{itemize}
    \item The map $\Omega_0^\SO\overset\cong\to\Z$ counting the number of points refines to a map of $E_\infty$-ring spectra $\MTSO\to H\Z$. Work of Thom~\cite[Théorème IV.13]{ThomThesis} shows this map is an isomorphism on homotopy groups in degrees $3$ and below.
    \item The Atiyah-Bott-Shapiro map $\MTSpin^c\to\ku$~\cite{ABS64} was shown to be a map of $E_\infty$-ring spectra by Joachim~\cite{Joa04}, and Anderson-Brown-Peterson~\cite{ABP67} showed this map is an isomorphism on homotopy groups in degrees $3$ and below.
    \item Joachim~\cite{Joa04} also showed the real Atiyah-Bott-Shapiro map $\MTSpin\to\ko$~\cite{ABS64} is a map of $E_\infty$-ring spectra, and Milnor~\cite{Mil63} showed this map is an isomorphism on homotopy groups in degrees $7$ and below.
    \item Ando-Hopkins-Rezk~\cite{AHR10} produced a map of $E_\infty$-ring spectra $\sigma\colon \MTString\to\tmf$, which Hill~\cite[Theorem 2.1]{Hil09} showed is an isomorphism on homotopy groups in degrees $15$ and below.
\end{itemize}
For all of these cases but $\MTString$, one can $2$-locally decompose the big sibling into a sum of modules over the little sibling: Wall~\cite{Wal60} produced a $2$-local equivalence
\begin{subequations}\label{eq:bigandlittle}
\begin{equation}
    \MTSO_{(2)}\overset\simeq\longrightarrow H\Z_{(2)} \vee \Sigma^4 H\Z_{(2)} \vee \Sigma^5 H\Z/2\vee\dotsb,
\end{equation}
and Anderson-Brown-Peterson~\cite{ABP67} produced $2$-local equivalences
\begin{align}
    \MTSpin_{(2)} &\overset\simeq\longrightarrow \ko_{(2)}\vee \Sigma^8 \ko_{(2)}\vee \Sigma^{10}(\ko\wedge J)_{(2)}\vee \dots\\
    \MTSpin^c_{(2)} &\overset\simeq\longrightarrow \ku_{(2)}\vee \Sigma^4 \ku_{(2)} \vee\Sigma^8 \ku_{(2)} \vee \Sigma^8 \ku_{(2)}\vee\dotsb,
\end{align}
\end{subequations}
where $J$ is a certain spectrum such that $\Sigma^2\ko\wedge J$ is the Postnikov $2$-connected cover of $\ko$.\footnote{The \spinc decomposition is implicit in~\cite{ABP67}; see Bahri-Gilkey~\cite{BG87b} for an explicit reference.}
It is not known whether $\tmf$ is a summand of
$\MTString$ (see, e.g., \cite{Lau04, Dev19, LS19, Pet19, Dev20}) so we do not know if there is a splitting like in
the three other cases.
\begin{rem}[String bordism with level structures?]
\label{level_str}
Associated to congruence subgroups $\Gamma\subset\SL_2(\Z)$ there are ``topological modular forms with level
structure:'' Hill-Lawson~\cite{HL16} and Meier~\cite{Mei21} construct
$E_\infty$-ring spectra $\TMF(\Gamma)$, $\Tmf(\Gamma)$, and $\tmf(\Gamma)$ with maps between them like
for vanilla $\tmf$.\footnote{Before the general constructions of Hill-Laswon and Meier, various examples of $\TMF(\Gamma)$, $\Tmf(\Gamma)$, and $\tmf(\Gamma)$ were constructed by Behrens~\cite{behrens2006modular, Beh07}, Mahowald-Rezk~\cite{MR09}, and Stojanoska~\cite{Sto12}.} The case $\Gamma = \Gamma_1(3)$ is especially interesting,
as many of the several ingredients we need for the proof of our main theorem are known to be true for
$\tmf(\Gamma_1(3))$ (usually written $\tmf_1(3)$): by work of Mathew~\cite[Theorem 1.2]{Mat16}, there is a change-of-rings theorem allowing one to simplify $2$-primary Adams spectral sequence computations to an easier subalgebra (see \S\ref{VB_change_of_rings}),
but at the time we originally wrote this paper, it was not yet known how to construct an $E_\infty$-ring Thom spectrum $M$ with an orientation $M\to\tmf_1(3)$
that is an isomorphism on low-degree homotopy groups.\footnote{See Wilson~\cite{Wil15} for results on closely
related questions.}
The existence of such a spectrum $M$ would lead to
generalizations of our main theorems to twists of $\tmf_1(3)$-homology.\footnote{A theorem of Meier~\cite[Theorem
1.4]{Mei21} suggests this may also apply to twists of $\tmf_1(n)$-homology for other values of $n$.}

Since we finished the first version of this paper, we learned of a tangential structure, called a ``string$^h$ structure,'' whose Thom spectrum has an $E_\infty$-ring map to $\tmf_1(3)$, as shown by Devalapurkar~\cite{Dev22} (see also~\cite{DY24} for another construction). In future work, we (the authors of this paper) will use this orientation to study $\tmf_1(3)$-module Thom spectra.
\end{rem}

\subsubsection{Twists of $\MTSO$ and $H\Z$}\label{so_twists}
We walk through the implications of \cref{cofiber_twists,universal_twist} in a relatively simple setting, addressing
\begin{itemize}
    \item what cohomology classes define twists of $\MTSO$ and $H\Z$ by way of \cref{cofiber_twists},
    \item what the corresponding twisted bordism and cohomology groups are, and
    \item what \cref{universal_twist} implies the Thom spectrum of the universal twist is.
\end{itemize}
Letting $f_2\colon X\to B\O$ be the identity and $f_1\colon X\to Y$ be $B\SO\to B\O$, we obtain twists of
$\MTSO$-oriented ring spectra, notably $\MTSO$ and $H\Z$, by maps to $B\SO/B\O\simeq K(\Z/2, 1)$, recovering a perspective of Hebestreit-Sagave~\cite[\S 1]{HS20}. The map $B\O\to
B\O/B\SO$ admits a section defined by regarding a map to $K(\Z/2, 1)$ as a real line bundle, so these twists are
given by real line bundles in the sense of \cref{VB_Thom}. Specifically, a class $a\in H^1(X;\Z/2)$ defines
a twist $f_a\colon X\to B\GL_1(\MTSO)$ by interpreting $a$ as a map $X\to B\O/B\SO$ and invoking
\cref{cofiber_twists}, and $a$ defines a second twist $g_a$ by choosing a real line bundle $L_a$ with $w_1(L_a) =
a$ (a contractible choice) and making the vector bundle twist as in \cref{VB_Thom}, but $f_a\simeq g_a$ and so
$M^\MTSO f_a\simeq \MTSO\wedge X^{L_a-1}$. Thus in a sense this example is redundant, as the main theorems of this
paper are long known for vector bundle twists, but we include this example because we found it a useful parallel to
to other families we study.

Let $\Omega_*^\SO(X, a)\coloneqq\pi_*(M^\MTSO f_a)$. Using the vector bundle interpretation of this twist,
$\Omega_*^\SO(X, a)$ has an interpretation as twisted oriented bordism groups, specifically the bordism groups
of manifolds $M$ with a map $h\colon M\to X$ and an orientation on $TM\oplus h^*L_a$. Alternatively, one could
think of this as the bordism groups of manifolds $M$ with a map $h\colon M\to X$ and a trivialization of the class
$w_1(M) - h^*a$; this perspective will be useful in later examples of non-vector-bundle twists.

%

\Cref{universal_twist} then implies the Thom spectrum of
\begin{equation}
\label{MSO_O1}
	K(\Z/2,1)\overset{\simeq}{
	\rightarrow}B\O_1\overset\sigma\longrightarrow B\O\longrightarrow B\GL_1(\mathbb S)\longrightarrow B\GL_1(\MTSO),
\end{equation}
is equivalent to $\MTO$. \Cref{Thom_change_of_rings} implies the Thom spectrum of~\eqref{MSO_O1} is $\MTSO\wedge
(B\O_1)^{\sigma-1}$, so we have reproved a theorem of Atiyah: $\MTSO\wedge
(B\O_1)^{\sigma-1}\simeq\MTO$~\cite[Proposition 4.1]{Ati61}.

The twist of $H\Z$ defined by $a$ recovers the usual notion of integral cohomology twisted by a class in $H^1(X; \Z/2)$.

\subsubsection{Twists of $\MTSpin^c$, $\ku$, and $\KU$}\label{spinc_twists}
Our next family of examples includes \spinc bordism and complex $K$-theory following the perspective of Hebestreit-Sagave~\cite{HS20}.\footnote{Hebestreit-Sagave use a different model for parametrized homotopy theory than Ando-Blumberg-Gepner-Hopkins-Rezk's; these two perspectives are shown to be equivalent by Hebestreit-Sagave-Schlichtkrull~\cite[Theorems 1.6 and 1.8]{HSS20}.} In \cref{splitting_spinc} we use \cref{cofiber_twists} to construct a map $K(\Z/2, 1)\times K(\Z, 3)\to B\GL_1(\MTSpin^c)$, defining twists of $\MTSpin^c$, $\ku$, and $\KU$ by classes in $H^1(\bl;\Z/2)$ and $H^3(\bl;\Z)$. These recover the usual twists of $K$-theory by these cohomology classes studied by~\cite{DK70, Ros89, AS04, ABG10} (\cref{KU_comparison}), and in \cref{twisted_spinc_bordism_comparison} we use work of Hebestreit-Joachim~\cite[Proposition 3.3.6]{HJ20} to describe the homotopy groups of the corresponding $\MTSpin^c$-module Thom spectra as bordism groups of manifolds with certain kinds of twisted \spinc structures.

The Atiyah-Bott-Shapiro orientation~\cite{ABS64, Joa04} defines ring homomorphisms $\mathit{Td} \colon
\MTSpin^c\to\ku\to\KU$, so by \cref{cofiber_twists} there are maps
\begin{equation}
\label{ku_cofib_twist}
	B\O/B\Spin^c\longrightarrow B\GL_1(\MTSpin^c)\overset{\mathit{Td}}{\longrightarrow} B\GL_1(\ku)\longrightarrow
	B\GL_1(\KU),
\end{equation}
i.e.\ twists of $\MTSpin^c$, $\ku$, and $\KU$ by maps to $B\O/B\Spin^c$.\footnote{The map~\eqref{ku_cofib_twist} is
nowhere near a homotopy equivalence; for example, it misses the ``higher twists'' of $\KU$ studied in, e.g.,~\cite{DP15, Pen16}. However, Gómez~\cite{Gom10} has proven that if $G$ is a compact Lie group, the map $[BG, BO/B\Spin^c]\to [BG, B\GL_1(\KU)]$ induced by~\eqref{ku_cofib_twist} is an equivalence: there are no higher twists of $\KU$ over $BG$.}
\begin{prop}
\label{splitting_spinc}
The map $K(\Z/2, 1)\to B\O$ defined by the tautological line bundle induces a homotopy equivalence of spaces
\begin{equation}
\label{ospinc_equiv}
	B\O/B\Spin^c\overset\simeq\longrightarrow K(\Z/2, 1)\times K(\Z, 3),
\end{equation}
implying that $\MTSpin^c$, $\ku$, and $\KU$ can be twisted over a space $X$ by classes $a\in H^1(X;\Z/2)$ and $c\in
H^3(X;\Z)$.
\end{prop}
See Beardsley-Luecke-Morava~\cite[Propositions 4.1 and 5.15]{BLM23} for a closely related splitting result.
\begin{proof}
We want to apply the third isomorphism theorem to the sequence of maps of abelian
$\infty$-groups $B\Spin^c\to B\SO\to B\O$ to obtain a short exact sequence
\begin{equation}
\label{OSpinc_SES}
	\shortexact*{B\SO/B\Spin^c}{B\O/B\Spin^c}{B\O/B\SO}.
\end{equation}
It is not immediate how to do this in the $\infty$-categorical setting, but we can do it. Instead of a short exact
sequence, we obtain a cofiber sequence, and in a stable $\infty$-category, the third isomorphism theorem for
cofiber sequences is a consequence of the octahedral axiom.
The $\infty$-category of abelian $\infty$-groups is not stable, as it is equivalent to the $\infty$-category of
connective spectra, but this $\infty$-category embeds in the stable $\infty$-category $\cat{Sp}$ of all spectra,
allowing us to make use of stability in certain settings: specifically, cofiber sequences $A\to B\to C$ of abelian
$\infty$-groups for which the induced map $\pi_0(B)\to\pi_0(C)$ is surjective; these cofiber diagrams map to
cofiber diagrams in $\cat{Sp}$, so we may invoke the octahedral axiom in $\cat{Sp}$. All cofiber sequences of
abelian $\infty$-groups we
discuss in this paper satisfy this $\pi_0$-surjectivity property, so we will not discuss it further. In particular,
we obtain the cofiber sequence~\eqref{OSpinc_SES}. 
Throughout this paper, whenever we write a short exact sequence of abelian $\infty$-groups, we mean a cofiber
sequence.

A similar argument allows one to deduce that fiber and cofiber sequences coincide for abelian $\infty$-groups from
the analogous fact for stable $\infty$-categories, assuming the same $\pi_0$-surjectivity hypothesis.
Since $B\Spin^c$ is the
fiber of $\beta w_2\colon B\SO\to K(\Z, 3)$, which is a map of abelian $\infty$-groups since $\beta w_2$ satisfies
the Whitney sum formula for oriented vector bundles, the cofiber $B\SO/B\Spin^c$ is equivalent, as abelian
$\infty$-groups, to $K(\Z, 3)$.\footnote{\label{HS_foot}This approach to the $K(\Z, 3)$ twist appears in Hebestreit-Sagave~\cite[\S 1]{HS20}.} Here, $\beta\colon H^k(\bl;\Z/2) \rightarrow H^{k+1}(\bl; \Z) $ is the Bockstein.
Likewise, $B\SO$ is the fiber of $w_1\colon
B\O\to K(\Z/2, 1)$, which is a map of abelian $\infty$-groups, so $B\O/B\SO\simeq K(\Z/2, 1)$.

The quotient $B\O\to B\O/B\SO\simeq K(\Z/2, 1)$ admits a section given by the tautological real line bundle
$K(\Z/2, 1)\simeq B\O_1\to B\O$; composing $K(\Z/2, 1)\to B\O$ with the quotient $B\O\to B\O/B\Spin^c$ we obtain a
section of~\eqref{OSpinc_SES}. That section splits~\eqref{OSpinc_SES}, which implies the proposition statement.
\end{proof}
%
\begin{defn}
Given classes $a\in H^1(X;\Z/2)$ and $c\in H^3(X;\Z)$, we call the twist $f_{a,c}\colon X\to B\GL_1(\MTSpin^c)$ that \cref{splitting_spinc} associates to $a$ and $c$ the \term{fake
vector bundle twist} for $a$ and $c$, and likewise for the induced twists of $\ku$ and $\KU$.
\end{defn}
The twist $f_{a,c}$ arises from a vector bundle twist if there is a vector bundle $V\to X$ such that $w_1(V) = a$ and
$\beta(w_2(V)) = c$, but there are choices of $X$, $a$, and $c$ for which no such vector bundle exists, e.g.\ if $c$ is not $2$-torsion.

Now that we have defined these twists, we get to the business of interpreting them.
\begin{defn}
\label{geometric_spinc}
Given $X$, $a$, and $c$ as above, let $\Omega_*^{\Spin^c}(X, a, c)$ denote the groups of bordism classes of manifolds $M$ with a map $f\colon M\to X$ and trivializations of
$w_1(M) - f^*(a)$ and $\beta(w_2(M)) - f^*(c)$. 
\end{defn}
This notion of twisted \spinc bordism, in the special case $a=0$, was first studied by Douglas~\cite[\S 5]{Dou06}, and implicitly appears in Freed-Witten's work~\cite{FW99} on anomaly cancellation.
\begin{lem}[{Hebestreit-Joachim~\cite[Corollary 3.3.8]{HJ20}}]
 \label{twisted_spinc_bordism_comparison}
 There is a natural isomorphism $\pi_*(M^{\MTSpin^c}f_{a,c})\overset\cong\to \Omega_*^{\Spin^c}(X, a, c)$.
\end{lem}
\begin{rem}
\label{HJ_rmk}
Hebestreit-Joachim~\cite{HJ20} use a different framework for twists based on May-Sigurdsson's parametrized homotopy
theory~\cite{MS06}; Ando-Blumberg-Gepner~\cite[Appendix B]{ABG11} prove a comparison theorem that allows us to pass between May-Sigurdsson's framework and Ando-Blumberg-Gepner-Hopkins-Rezk's. Additionally, Hebestreit-Joachim work with twisted spin bordism and $\KO$-theory, but for the complex case the arguments are essentially the same.
\end{rem}
\begin{rem}
\label{HJ_lift}
Though Hebestreit-Joachim~\cite[Corollary 3.3.8]{HJ20} state their results as isomorphisms of bordism groups, their proof actually proves an equivalence of spectra. Focusing on the \spinc case, given $X$, $a$, and $c$ as above, let $\xi_{a,c}\colon B(a,c)\to B\O$ be the fiber of the map
\begin{equation}
    (w_1 - a, \beta(w_2 - c))\colon B\O\times X\to K(\Z/2, 1)\times K(\Z, 3),
\end{equation}
so that $\Omega_*^{\xi_{a,c}}\cong\Omega_*^{\Spin^c}(X,a,c)$. In the twisted \spinc setting, Hebestreit-Joachim's techniques in fact prove that there is a canonical $\MTSpin^c$-module equivalence
\begin{equation}
    M^{\MTSpin^c}f_{a,c} \overset\simeq\longrightarrow \mathit{MT\xi}_{a,c},
\end{equation}
where the $\MTSpin^c$-module structure on $\mathit{MT\xi}_{a,c}$ arises from the canonical $\xi_{a,c}$-structure on the sum of a \spinc vector bundle and a $\xi_{a,c}$-structured vector bundle.

Similar considerations are true for the twisted spin and string structures we study later in this section.
\end{rem}
\begin{lem}[{Hebestreit-Sagave~\cite{HS20}}]
\label{KU_comparison}
With $X$, $a$, and $c$ as above, the homotopy groups of $M^\KU f_{a,c}$ are naturally isomorphic to the twisted $K$-theory groups of~\cite{DK70, Ros89, AS04, ABG10}.
\end{lem}
\begin{exm}
\label{universal_spinc_twists}
\Cref{universal_twist} computes a few example of $\MTSpin^c$-module Thom spectra for us.
\begin{enumerate}
	\item Letting $X = Y = B\O/B\Spin^c$ and $f_1 = \id$, \cref{universal_twist} implies that the Thom spectrum of
	the universal twist $B\O/B\Spin^c\to B\GL_1(\MTSpin^c)$ is $\MTO$. From a bordism point of view, this is the
	fact that since $a$ and $c$ pull back from $K(\Z/2, 1)\times K(\Z, 3)$, they can be arbitrary classes, so the
	required trivializations of $w_1(M) - f^*(a)$ and $\beta(w_2(M)) - f^*(c)$ are uniquely specified by $a =
	w_1(M)$ and $c = \beta(w_2(M))$, so this notion of twisted \spinc structure is no structure at all.
	\item Let $Y$ be as in the previous example and let $f_1\colon X\to Y$ be the map $K(\Z, 3)\simeq
	B\SO/B\Spin^c\to B\O/B\Spin^c$. \Cref{universal_twist} says the Thom spectrum of
	\begin{equation}
		K(\Z, 3)\longrightarrow B\O/B\Spin^c\longrightarrow B\GL_1(\MTSpin^c)
	\end{equation}
	is equivalent to $\MTSO$. We stress that this  twist by $K(\Z,3)$ does not come from a vector bundle because
	all vector bundle twists of $\MTSpin^c$ are torsion and of the form $\beta(w_2(M))$, but the universal twist over $K(\Z,3)$ is not. 
\end{enumerate}
\end{exm}
\begin{lem}
\label{ospinc_not_gp}
    The equivalence of spaces $B\O/B\Spin^c\simeq K(\Z/2, 1)\times K(\Z, 3)$ from~\eqref{ospinc_equiv} is not an equivalence of $\infty$-groups.
\end{lem}
Beardsley-Luecke-Morava~\cite[Corollary 4.9]{BLM23} prove a closely related result.
\begin{proof}
Suppose that this is an equivalence of $\infty$-groups. Then the inclusion $K(\Z/2, 1)\to K(\Z/2, 1)\times K(\Z, 3)\to B\O/B\Spin^c$ is a map of $\infty$-groups, so the composition
\begin{equation}
\label{pullback_to_BZ2}
\varphi\colon K(\Z/2, 1)\longrightarrow K(\Z/2, 1)\times K(\Z, 3)\longrightarrow B\O/B\Spin^c\longrightarrow B\GL_1(\MTSpin^c)
\end{equation}
is a map of $\infty$-groups. By work of Ando-Blumberg-Gepner~\cite[Theorem 1.7]{ABG11}, this implies the Thom spectrum $M\varphi$ is an $E_1$-ring spectrum. We will explicitly identify $M\varphi$ and show this is not the case.

We saw above that the map $K(\Z/2, 1)\to B\O/B\Spin^c$ factors through the map $K(\Z/2, 1)\to B\O$ defined by the tautological line bundle $\sigma\to B\O_1\simeq K(\Z/ 2, 1)$, meaning that the twist~\eqref{pullback_to_BZ2} is the vector bundle twist of $\MTSpin^c$ for the tautological line bundle $\sigma\to B\O_1$. Applying \cref{Thom_change_of_rings} with $R_1 = \Sph$ and $R_2 = \MTSpin^c$, we conclude $M\varphi\simeq \MTSpin^c\wedge (B\O_1)^{\sigma-1}$. Bahri-Gilkey~\cite{BG87a, BG87b} identify this spectrum with $\MTPin^c$, which is known to not be an $E_1$-ring spectrum: for example, a $E_1$-ring structure induces a graded ring structure on homotopy groups, making $\pi_k(\MTPin^c)$ into a $\pi_0(\MTPin^c)$-module for all $k$, but $\pi_0\MTPin^c\cong\Z/2$ and $\pi_2(\MTPin^c)\cong\Z/4$~\cite[Theorem 2]{BG87b}.
\end{proof}

\subsubsection{Twists of $\MTSpin$, $\ko$, and $\KO$}
\label{spin_twists}
The real analogue of \S\ref{spinc_twists} is very similar; we summarize the story here, highlighting the
differences. Once again this perspective is due to Hebestreit-Sagave~\cite{HS20}. Again there are $E_\infty$ ring spectrum maps $\MTSpin\overset{\widehat A}{\to}
\ko\to\KO$~\cite{ABS64, Joa04, AHR10}, allowing us to use \cref{cofiber_twists} to produce a sequence of maps
\begin{equation}
	B\O/B\Spin\longrightarrow B\GL_1(\MTSpin)\overset{\widehat A}{\longrightarrow} B\GL_1(\ko)\longrightarrow
	B\GL_1(\KO).
\end{equation}
Hebestreit-Sagave~\cite{HS20} and Freed-Hopkins~\cite[\S 10]{FH21a} use the $\infty$-group $B\O/B\Spin$ to study twists of spin bordism; Freed-Hopkins call it $\mathbf P$.
\begin{prop}
\label{splitting_spin}
The map $K(\Z/2, 1)\to B\O$ defined by the tautological line bundle induces a homotopy equivalence of spaces
\begin{equation}
\label{OSpin}
	B\O/B\Spin\overset\simeq\longrightarrow K(\Z/2, 1)\times K(\Z/2, 2),
\end{equation}
implying $\MTSpin$, $\ko$, and $\KO$ can be twisted over a space $X$ by classes $a\in H^1(X;\Z/2)$ and $b\in
H^2(X;\Z/2)$.
\end{prop}
The proof is nearly the same as the proof of \cref{splitting_spinc}: fit $B\O/B\Spin$ into a split cofiber sequence with $B\SO/B\Spin\simeq K(\Z/2, 2)$ (because $B\Spin\to B\O$ is the fiber of $w_2\colon B\SO\to
K(\Z/2, 2)$) and $B\O/B\SO\simeq K(\Z/2, 1)$. 
See also Beardsley-Luecke-Morava~\cite[Propositions 4.1 and 5.19]{BLM23}, who prove a closely related splitting result, and Carmeli-Luecke~\cite[Theorem C]{CL24} for an analogous splitting result in $B\GL_1(K(\Z))$.
\begin{defn}
    We call the twist $f_{a,b}\colon X\to B\GL_1(\MTSpin)$ associated to $a$ and $b$ the \term{fake
vector bundle twist} for $a$ and $b$, and likewise for the induced twists of $\ko$ and $\KO$.
\end{defn}
\begin{rem}
\label{pinm_pinp}
The space of homotopy self-equivalences of $K(\Z/2, 1)\times K(\Z/2, 2)$ is not connected: for example, if $a$
denotes the tautological class in $H^1(K(\Z/2, 1);\Z/2)$ and $b$ is the tautological class in $H^2(K(\Z/2,
2);\Z/2)$, the homotopy class of maps $\Phi\colon K(\Z/2, 1)\times K(\Z/2, 2) \to K(\Z/2, 1)\times K(\Z/2, 2)$
defined by the classes $(a, a^2+b)$ is not the identity and is invertible. The choice of identification we made
in~\eqref{OSpin} matters: if one uses a different identification, one obtains a different notion of fake vector
bundle twist and a different formula in \cref{the_twisted_modules} to make \cref{thom_module_calc} true.
\end{rem}
\begin{lem}
\label{ospin_not_gp}
\eqref{OSpin} is not an equivalence of $\infty$-groups.
\end{lem}
This is closely related to a theorem of Beardsley-Luecke-Morava~\cite[Proposition 4.4]{BLM23}.

One can prove \cref{ospin_not_gp} in the same way as \cref{ospinc_not_gp}, by pulling back along the section $K(\Z/2, 1)\to
B\O/B\Spin$ and observing that the Thom spectrum $\MTSpin\wedge (B\O_1)^{\sigma-1}$ is not a ring spectrum in much
the same way:\footnote{This is the first place where the choice of identification~\eqref{OSpin} has explicit
consequences, as promised in \cref{pinm_pinp}: if we compose with the identification of $K(\Z/2, 1)\times K(\Z/2,
2)$ given by the classes$(a, a^2+b)$ described in that remark, we would instead obtain $\MTSpin\wedge
(B\O_1)^{3\sigma-3}$. This is not a ring spectrum either, as it can be identified with $\MTPin^+$~\cite[\S
8]{Sto88}, and $\pi_0(\MTPin^+)\cong\Z/2$ and $\pi_4(\MTPin^+)\cong\Z/16$~\cite{Gia73a}.} using the equivalence
$\MTSpin\wedge (B\O_1)^{\sigma-1}\simeq\MTPin^-$~\cite[\S 7]{Pet68} and the groups $\pi_0(\MTPin^-)\cong\Z/2$ and
$\pi_2(\MTPin^-)\cong\Z/8$~\cite{ABP69, KT90} to show $\MTPin^-$ is not a ring spectrum. There is also another nice
proof, which we give below.
\begin{proof}
If $X$ is a space and $Y$ is an $\infty$-group, the set $[X, Y]$ has a natural group structure. Therefore it
suffices to find a space such that $[X, B\O/B\Spin]$ and $[X, K(\Z/2, 1)\times K(\Z/2, 2)]$ are non-isomorphic
groups.

To calculate the addition in $[\bl, B\O/B\Spin]$, we use the fact that if two maps $f,g\colon X\to \O/B\Spin$ factor through $B\O$, meaning they are represented by rank-zero virtual vector bundles $V_f, V_g\to X$, then $f+g$ is the image of $V_f\oplus V_g$ under $B\O\to B\O/B\Spin$. This implies that for classes in the image of that quotient map, if we use~\eqref{OSpin} to identify two classes $\phi_1,\phi_2\in[X,
B\O/B\Spin]$ with pairs $\phi_i = (a_i\in H^1(X;\Z/2), b_i\in H^2(X;\Z/2))$, then addition follows the Whitney sum
formula:
\begin{equation}
\label{whitney_sum_OSpin}
	(a_1, b_1) \oplus (a_2, b_2) = (a_1 + a_2, b_1 + b_2 + a_1a_2).
\end{equation}
This is different from the componentwise addition on $K(\Z/2, 1)\times K(\Z/2, 2)$: for example, $[B\Z/2, K(\Z/2,
1)\times K(\Z/2, 2)]\cong\Z/2\oplus\Z/2$, but the map $[B\Z/2, B\O]\to [B\Z/2, B\O/B\Spin]$ is surjective, so using~\eqref{whitney_sum_OSpin}, one can show that $[B\Z/2, B\O/B\Spin]\cong\Z/4$.
\end{proof}
\begin{defn}
\label{geometric_spin}
Given $X$, $a$, and $b$ as above, let $\Omega_*^{\Spin}(X, a, b)$ denote the groups of bordism classes of manifolds $M$ with a map $f\colon M\to X$ and trivializations of
$w_1(M) - f^*(a)$ and $w_2(M) - f^*(b)$. 
\end{defn}
B.L.\ Wang~\cite[Definition 8.2]{Wan08} first studied these twists of spin bordism in the case $a = 0$.
\begin{lem}[{Hebestreit-Joachim~\cite[Corollary 3.3.8]{HJ20}}]
 \label{twisted_spin_bordism_comparison}
 There is a natural isomorphism $\pi_*(M^{\MTSpin}f_{a,b})\overset\cong\to \Omega_*^{\Spin}(X, a, b)$.
\end{lem}
\begin{lem}[{Hebestreit-Sagave~\cite{HS20}}]
\label{KO_comparison}
    With $X$, $a$, and $b$ as above, the homotopy groups of $M^\KO f_{a,b}$ are naturally isomorphic to the twisted $\KO$-theory groups of~\cite{DK70, HJ20}.
\end{lem}
\begin{exm}
\label{universal_spin_twists}
\Cref{universal_twist} implies the Thom spectrum of the universal twist of $\MTSpin$ over $B\O/B\Spin$ is $\MTO$, and of the universal twist over $K(\Z/2, 2)\simeq B\SO/B\Spin$ is $\MTSO$. The former equivalence is due to Hebestreit-Joachim~\cite[Observation 3.3.5]{HJ20}, and latter equivalence is due to Beardsley~\cite[\S 3]{Bea17}.
\end{exm}
%
%
\subsubsection{Twists of $\MTString$, $\tmf$, $\Tmf$, and $\TMF$}\label{string_twists}
The final family we consider in this paper is string bordism and topological modular forms. The story has a similar shape: we obtain twists by $B\O/B\String$, and we simplify $B\O/B\String$ to define fake vector bundle twists. However, in \cref{string_not_split} we learn that $B\O/B\String$ is not homotopy equivalent to a product of Eilenberg-Mac Lane spaces. For this reason, the fake vector bundle twist uses a generalized cohomology theory called \term{supercohomology} and denoted $\SH$ (\cref{supercoh_defn}); we finish this subsubsection by studying cohomology classes associated to a degree-$4$ supercohomology class, which we will need in the proof of \cref{thom_module_calc}.

If $V\to X$ is a spin
vector bundle, it has a characteristic class $\lambda(V)\in H^4(X;\Z)$ such that $2\lambda(V) = p_1(V)$; a
\term{string structure} on $V$ is a trivialization of $\lambda$. It is not hard to check that $\lambda$ is additive
in direct sums,
so defines a map of abelian
$\infty$-groups $\lambda\colon B\Spin\to K(\Z, 4)$. The fiber of this map is an $\infty$-group
$B\String$, which is the classifying space for string structures.

Unlike for $K$-theory, there are three different kinds of topological modular forms: a connective spectrum $\tmf$,
a periodic spectrum $\TMF$, and a third spectrum $\Tmf$ which is neither connective nor periodic. All three are
$E_\infty$-ring spectra, and there are ring spectrum maps $\tmf\to\Tmf\to\TMF$. 
  Ando-Hopkins-Rezk~\cite{AHR10} constructed a ring spectrum map $\sigma\colon\MTString\to\tmf$, so
\cref{cofiber_twists} gives us twists of $\tmf$, $\Tmf$, and $\TMF$ from $B\O/B\String$:
\begin{equation}
	B\O/B\String\to B\GL_1(\MTString)\overset\sigma\to B\GL_1(\tmf)\to B\GL_1(\Tmf)\to B\GL_1(\TMF).
\end{equation}

Like in \S\ref{spinc_twists} and \S\ref{spin_twists}, the section $B\O/B\SO\to B\O$ defines a homotopy equivalence of spaces
\begin{equation}
\label{ostring_SES}
	B\O/B\String\overset\simeq\longrightarrow K(\Z/2, 1)\times B\SO/B\String,
\end{equation}
and there is a short exact sequence of abelian $\infty$-groups
\begin{equation}
\label{spin_string_SES}
	\shortexact*[\iota][]{\underbracket{B\Spin/B\String}_{K(\Z, 4)}}{B\SO /
	B\String}{\underbracket{B\SO/B\Spin}_{K(\Z/2, 2)}},
\end{equation}
but now something new happens.
\begin{prop}
\label{string_not_split}
\eqref{spin_string_SES} is not split.
\end{prop}
\begin{proof}
%
A splitting of~\eqref{spin_string_SES} defines a section $s\colon B\SO/B\String\to B\Spin/B\String$, meaning $s\circ\iota = \id$. Therefore the map $\lambda\colon B\Spin\to B\Spin/B\String\overset\simeq\to K(\Z, 4)$ factors through $B\SO$:

\begin{equation}\begin{tikzcd}
	B\Spin & {B\Spin/B\String} & {K(\Z, 4)}. \\
	B\SO & {B\SO/B\String}
	\arrow["\simeq"', from=1-2, to=1-3]
	\arrow[from=1-1, to=1-2]
	\arrow[from=1-1, to=2-1]
	\arrow[from=2-1, to=2-2]
	\arrow["\iota"', shift right, from=1-2, to=2-2]
	\arrow["s"', shift right, from=2-2, to=1-2]
        \arrow["\lambda", curve={height=-18pt}, from=1-1, to=1-3]
\end{tikzcd}\end{equation}
We let $\mu$ denote the extension of $\lambda$ to $B\SO$. Brown~\cite[Theorem 1.5]{Bro82} shows that $H^4(B\SO;\Z)\cong\Z$ with generator $p_1$, so for any class $x\in H^4(B\SO;\Z)$, the pullback of $x$ to $B\Spin$ is some integer multiple of $p_1$. But the pullback of $\mu$ is $\lambda$, which is not an integer multiple of $p_1$, so we have found a contradiction.
\end{proof}
We want an analogue of the fake vector bundle twists from \S\ref{spinc_twists} and \S\ref{spin_twists} for $\MTString$, $\tmf$,
$\Tmf$, and $\TMF$, but since we just saw that $B\SO/B\String$ is not a product of Eilenberg-Mac Lane spaces, we
have to figure out what exactly it is. The answer turns out to be the analogue of an Eilenberg-Mac Lane space for a
relatively simple generalized cohomology theory.

Postnikov theory implies that if $E$ is a spectrum with only two nonzero homotopy groups
$\pi_m(E) = A$ and $\pi_n(E) = B$ (assume $m < n$ without loss of generality), then $E$ is classified by the data
of $m$, $n$, $A$,
$B$, and the \term{$k$-invariant} $k_E\in [\Sigma^m HA, \Sigma^{n+1}HB]$, a stable cohomology operation.
\begin{defn}[{Freed~\cite[\S 1]{Fre08}, Gu-Wen~\cite{GW14}}]
\label{supercoh_defn}
Let $\SH$ be the spectrum with $\pi_{-2}(\SH) = \Z/2$, $\pi_0(\SH) = \Z$, and the $k$-invariant $k_\SH =
\beta\circ\Sq^2\colon H^*(\bl;\Z/2)\to H^{*+3}(\bl;\Z)$. The generalized cohomology theory defined by $\SH$ is
called \term{(restricted) supercohomology}.\footnote{The adjective ``restricted'' is to contrast this theory with
``extended'' supercohomology of Kapustin-Thorngren~\cite{KT17} and Wang-Gu~\cite{WG20}. See~\cite[\S 5.3,
5.4]{GJF19}.}
\end{defn}
Just as the Eilenberg-Mac Lane spectrum $H\Z$ is assembled from Eilenberg-Mac Lane spaces $K(\Z, n)$ and there is a natural isomorphism $H^n(X, \Z)\overset\cong\to [X, K(\Z, n)]$, if one defines $\SK(n)$ to be the abelian $\infty$-group which is the extension
\begin{equation}
\label{sk_defn}
    \shortexact{K(\Z, n)}{\SK(n)}{K(\Z/2, n-2)}{}
\end{equation}
classified by $\beta(\Sq^2(T))\in H^{n+1}(K(\Z/2, n-2); \Z)$, where $T\in H^{n-2}(K(\Z/2, n-2);\Z/2)$ is the tautological class and $\beta$ is the integral Bockstein, then the spaces $\SK(n)$ assemble into a model for the spectrum $\SH$ and there is a natural isomorphism $\SH^n(X)\overset\cong\to [X, \SK(n)]$.

Like Eilenberg-Mac Lane spaces, the spaces $\SK(n)$ are related by loops.
\begin{lem}
\label{SK_loop}
If $n\ge 3$, there is a canonical homotopy class of homotopy equivalences $\Omega\SK(n)\overset\simeq\to\SK(n-1)$ compatible with the identifications $\Omega K(A, n)\overset\simeq\to K(A, n-1)$ and the maps in~\eqref{sk_defn}.
\end{lem}
\begin{proof}
This follows by applying $\Omega$ to the cofiber sequence~\eqref{sk_defn}, then observing that this preserves the $k$-invariant $\beta\circ\Sq^2$.
\end{proof}
\begin{prop}
\label{SO_String_is_supercoh}
There is an equivalence of abelian $\infty$-groups $B\SO/B\String\overset\simeq\to \SK(4)$. Moreover, the space of such equivalences is connected. Therefore there is a natural isomorphism of abelian groups $[X, B\SO/B\String]\cong\SH^4(X)$.
\end{prop}
The point of the last sentence in \cref{SO_String_is_supercoh} is that in our proof, we do not specify an
isomorphism, so a priori there could be ambiguity like in \cref{pinm_pinp}.  But since the space of such
identifications is connected, there is a unique identification in the homotopy category, which suffices for the
calculations we make in this paper.
\begin{proof}[Proof of \cref{SO_String_is_supercoh}]
We are trying to identify the extension~\eqref{spin_string_SES} of abelian $\infty$-groups to relate it to $\SH$.
Because $B\SO/B\String$ is an abelian $\infty$-group, this extension, a priori classified by $H^5(K(\Z/2, 2), \Z)$,
actually is classified by the stabilization $[\Sigma^2 H\Z/2, \Sigma^5 H\Z]$: this extension is equivalent data to a
fiber sequence of connective spectra, so we get to use stable Postnikov theory. Our first step is to understand
$[\Sigma^2 H\Z/2, \Sigma^5 H\Z]$.
\begin{lem}
\label{turn_it_around}
For all $k\in\Z$, $[H\Z/2, \Sigma^k H\Z]\cong [H\Z, \Sigma^{k-1}H\Z/2]$.
\end{lem}
\begin{proof}
This follows by using the universal coefficient theorem to relate both groups to homology groups: the short exact sequences in the universal coefficient theorem simplify to identify the two groups in the lemma statement with $H_{k-1}(H\Z; \Z/2)$, resp.\ $H_{k-1}(H\Z/2;\Z)$ (the latter because the homology of $H\Z/2$ is torsion). Both of these groups are isomorphic to $\pi_{k-1}(H\Z\wedge H\Z/2)$, so the lemma follows.
\end{proof}
\begin{cor}
\label{HZHZ2cor}
$[\Sigma^2 H\Z/2, \Sigma^4 H\Z] = 0$ and $[\Sigma^2 H\Z/2, \Sigma^5 H\Z]\cong\Z/2$.
\end{cor}
\begin{proof}
By \cref{turn_it_around}, we need to compute $[H\Z, \Sigma^i H\Z/2] = H^i(H\Z; \Z/2)$ for $i = 1, 2$.
Let $\cA$ denote the mod $2$ Steenrod algebra; then $H^*(H\Z; \Z/2)\cong\cA\otimes_{\cA(0)}\Z/2$~\cite[\S 9]{Wal60}. This vanishes in
degree $1$ and is isomorphic to $\Z/2$ in degree $2$.
\end{proof}
\Cref{string_not_split} implies~\eqref{spin_string_SES} is classified by a nonzero element of $[\Sigma^2 H\Z/2,
\Sigma^5 H\Z]$. And by definition, $\SK(4)$ is an extension of $K(\Z/2, 2)$ by $K(\Z, 4)$ classified by $\beta\circ\Sq^2$, which is a
nonzero element of $[\Sigma^2 H\Z/2, \Sigma^5 H\Z]$. Since this group is isomorphic to $\Z/2$ by \cref{HZHZ2cor},
these two nonzero elements must coincide, so there is an equivalence of abelian $\infty$-groups $B\SO/B\String\simeq \SK(4)$. There is a homotopy
type of such equivalences, and $\pi_0$ of that homotopy type is a torsor over $[\Sigma^2 H\Z/2, \Sigma^4H\Z]$,
which vanishes by \cref{HZHZ2cor}, so the space of identifications is connected.
\end{proof}
\begin{cor}
\label{splitting_string}
The map $K(\Z/2, 1)\to B\O$ defined by the tautological line bundle induces a homotopy equivalence of spaces
\begin{equation}
\label{ostring_equiv}
	B\O/B\String\overset\simeq\longrightarrow K(\Z/2, 1)\times SK(4),
\end{equation}
implying that $\MTString$, $\tmf$, and $\TMF$ can be twisted over a space $X$ by classes $a\in H^1(X;\Z/2)$ and $d\in
\SH^4(X)$.
\end{cor}
\begin{defn}
 We call the twists associated to $a$ and $d$ in \cref{splitting_string} the
\term{fake vector bundle twists} for $\MTString$, $\tmf$, $\Tmf$, and $\TMF$.
\end{defn}
%
\begin{rem}
\label{super_lambda}
Another consequence of \cref{SO_String_is_supercoh}, applied to the proof strategy of \cref{string_not_split}, is
that, even though $\lambda\in H^4(B\Spin;\Z)$ does not pull back from $B\SO$, its image in $\SH^4(B\Spin)$
\emph{does} pull back from a class $\lambda\in \SH^4(B\SO)$. This is a theorem of Freed~\cite[Proposition
1.9(i)]{Fre08}, with additional proofs given by Jenquin~\cite[Proposition 4.6]{Jen05} and Johnson-Freyd and
Treumann~\cite[\S 1.4]{JFT20}.
\end{rem}
The map $K(\Z, 4) \simeq B\Spin/B\String\to B\SO/B\String$ means degree-$4$ ordinary cohomology classes also
define degree-$4$ twists of string bordism and topological modular forms. Twists of this sort have already been studied, so we compare our twists to the literature.
\begin{defn}
\label{twisted_string_defn}
Given $X$, $a$, and $d$ as in \cref{splitting_string}, let $\Omega_*^\String(X, a, d)$ denote the groups of bordism classes of manifolds $M$ equipped with maps $f\colon M\to X$ and trivializations of $w_1(M) - f^*(a)\in H^1(M;\Z/2)$ and $\lambda(M) - f^*(d)\in\SH^4(M)$.
\end{defn}
A priori we only defined $\lambda$ as a characteristic class of oriented vector bundles; for an unoriented vector bundle $V$, $\lambda(V)$ is be defined to be $\lambda(V\oplus\Det(V))$, as the latter bundle is canonically oriented. \Cref{twisted_string_defn} first appears in work of B.L.\ Wang~\cite[Definition 8.4]{Wan08} in the special case when $a = 0$ and $d$ comes from ordinary cohomology.
\begin{lem}
 \label{twisted_string_bordism_comparison}
 There is a natural isomorphism $\pi_*(M^{\MTString}f_{a,d})\overset\cong\to \Omega_*^{\String}(X, a, d)$.
\end{lem}
This follows from work of Hebestreit-Joachim~\cite{HJ20}, much like \cref{twisted_spinc_bordism_comparison,twisted_spin_bordism_comparison}. Though they do not discuss the $\MTString$ case explicitly, their proof can be adapted to our setting. See~\cite[Remark 2.2.3]{HJ20}.

We can also compare with preexisting twists of $\tmf$.
\begin{lem}
\label{tmf_comparison}
    The fake vector bundle twist defined by $K(\Z, 4)\to \SK(4)\to B\GL_1(\tmf)$ is homotopy equivalent to the twist $K(\Z, 4)\to B\GL_1(\tmf)$ constructed by Ando-Blumberg-Gepner~\cite[Proposition 8.2]{ABG10}.
\end{lem}
\begin{proof}[Proof sketch]
This equivalence is not obvious, because Ando-Blumberg-Gepner construct their twist in a different way: beginning with a map
$\phi\colon \Sigma_+^\infty K(\Z, 3)\to \tmf$ and using the adjunction~\cite[(1.4), (1.7)]{ABGHR14b} between
$\Sigma_+^\infty$ and $\GL_1$. However, their argument builds $\phi$ out of the map $\lambda\colon B\Spin\to
B\Spin/B\String\simeq K(\Z, 4)$, allowing one to pass our construction through their argument and conclude that
our twist, as a class in $[K(\Z, 4), B\GL_1(\tmf)]$, coincides with Ando-Blumberg-Gepner's.
\end{proof}
Though these twists by degree-$4$ cohomology are relatively well-studied, there are not so many examples of
lower-degree twists of string bordism or topological modular forms in the literature. See
Freed-Hopkins-Teleman~\cite[\S 2]{FHT10}, Johnson-Freyd~\cite[\S
2.3]{JF20}, Beardsley-Luecke-Morava~\cite[Example 5.25]{BLM23}, Tachikawa-Yamashita~\cite{TY23ASD, TY21}, Tachikawa-Yonekura~\cite{TY25}, and~\cite[Remark 2.16]{DY24} for some examples.
\begin{exm}
\label{universal_string_twists}
Just as in \cref{universal_spinc_twists,universal_spin_twists}, \cref{universal_twist} calculates some
$\MTString$-module Thom spectra for us: over $B\O/B\String$ we get $\MTO$; over $B\SO/B\String$ we get $\MTSO$, and
over $K(\Z, 4)$ we get $\MTSpin$. The last example is due to Beardsley~\cite[\S 3]{Bea17}.
\end{exm}
%
%
\begin{rem}
Like in \cref{ospinc_not_gp,ospin_not_gp},~\eqref{ostring_SES} is not an equivalence of $\infty$-groups. The same two proofs are available to us: pulling back to $K(\Z/2, 1)$ and showing we do not obtain an $E_1$-ring spectrum, and comparing the group structures on $[\RP^\infty, B\O/B\String]$ and $[\RP^\infty, K(\Z/2, 1)\times B\SO/B\String]$. For the second proof, one observes that $[\RP^\infty, B\O/B\String]\cong\Z/8$ but $[\RP^\infty, K(\Z/2, 1)\times B\SO/B\String]$ has at least four elements of order $4$, then concludes.

For the first proof, we obtain $\MTString\wedge (B\O_1)^{\sigma-1}$ like before; to our knowledge, this notion of bordism has not been studied.\footnote{By analogy with $\SO$ and $\O$ and $\Spin$ and $\Pin^-$, one could call this tring\textsuperscript{$-$} bordism. We hope there is a better name for this spectrum.} However, since this is a vector bundle Thom spectrum, the change-of-rings trick shows that in topological degrees $15$ and below, the $E_2$-page of the Adams spectral sequence computing $\Omega_*^\String((B\O_1)^{\sigma-1})_2^\wedge$ is isomorphic to $\Ext_{\cA(2)}^{s,t}(H^*((B\Z/2)^{\sigma-1};\Z/2), \Z/2)$ (see \S\ref{VB_change_of_rings} for notation and an explanation). Davis-Mahowald~\cite[Table 3.2]{DM78} have computed these Ext groups, and from their computation it directly follows using the Adams spectral sequence that $\pi_0\cong\Z/2$ and $\pi_3\cong\Z/8$, so just like for $\MTPin^c$ and $\MTPin^-$, $\MTString\wedge (B\O_1)^{\sigma-1}$ does not admit an $E_1$-ring spectrum structure.
\end{rem}
In the proof of \cref{thom_module_calc} we will need to understand the mod $2$ cohomology classes naturally associated to a degree-$4$ supercohomology class $d$. The quotient $t\colon \SH\to \Sigma^{-2}H\Z/2$ gives us a degree-$2$ class $t(d)$, sometimes called the \term{Gu-Wen layer} of $d$.

To proceed further, we study the Serre spectral sequence associated to the fibration $K(\Z, 4)\to \SK(4)\to K(\Z/2, 2)$. Let $\overline \delta\in H^4(K(\Z, 4);\Z/2)$ be the mod $2$ reduction of the tautological class; this defines a class in $E_2^{0,4}$ of our Serre spectral sequence, which we also call $\overline\delta$.
\begin{lem}
The class $\overline\delta\in E_2^{0,4}$ survives to the $E_\infty$-page. 
\end{lem}
\begin{proof} The only possible differential that could be nonzero on $\overline\delta$ is the transgressing $d_5$, which pulls back from the transgressing $d_5$ on $\overline\delta$ in the Serre spectral sequence for the universal fibration with fiber $K(\Z, 4)$, namely $K(\Z, 4)\to E(K(\Z, 4))\to B(K(\Z, 4)) \simeq K(\Z, 5)$. In the universal fibration, $d_5(\overline\delta)$ is the mod $2$ tautological class $\epsilon\in H^5(K(\Z, 5);\Z/2)$, so in the fibration with total space $\SK(4)$, $d_5(\overline\delta)$ is the pullback of $\epsilon$ by the classifying map $\beta\circ\Sq^2\colon K(\Z/2, 2)\to K(\Z, 5)$. Thus $\epsilon\mapsto (\beta\Sq^2(B))\bmod 2 = \Sq^1\Sq^2(B)$, where $B\in H^2(K(\Z/2, 2);\Z/2)$ is the tautological class, but $\Sq^1\Sq^2(B) = \Sq^3(B) = 0$, as $B$ has degree $2$. Thus $d_5(\overline\delta)= 0$.
\end{proof}
\begin{rem}
This is an unstable phenomenon: for $n > 4$, a similar argument shows the transgressing differential on the mod $2$ tautological class of $K(\Z, n)$ is nonzero, so no analogue of $\overline\delta$ exists in the cohomology of $\SK(n)$.
\end{rem}
We want to lift $\overline\delta\in E_\infty^{0,4}$ to an element $\delta$ of $H^4(\SK(4);\Z/2)$. If $B$ is the tautological class of $K(\Z/2, 2)$, then there is an ambiguity between $\delta$ and $\delta + B^2$. To resolve this ambiguity, pull back across the map $\lambda\colon B\SO\to \SK(4)$. By comparing the Serre spectral sequences for the fibrations $K(\Z, 4)\to \SK(4)\to K(\Z/2, 2)$ and $B\Spin\to B\SO\to K(\Z/2, 2)$, one learns that $\lambda^*(\delta)$ is either $w_4$ or $w_4 + w_2^2$. Choosing the former allows us to uniquely define $\delta$.
\begin{cor}
\label{delta_SH_defn}
There is a unique class $\delta\in H^4(\SK(4);\Z/2)$ such that $\lambda^*(\delta) = w_4$.
\end{cor}
Phrased differently, associated to every $d\in\SH^4(X)$ is a class $\delta\in H^4(X;\Z/2)$, such that if there is an oriented vector bundle $V\to X$ with $d = \lambda(V)$, then $\delta = w_4(V)$. The same line of reasoning also shows that $\lambda^*(t(d)) = w_2$.

\section{Computing the input to Baker-Lazarev's Adams spectral sequence}\label{section:BLSS}
\label{s:BL}
%
\subsection{Review: the change-of-rings theorem for vector bundle Thom spectra}
\label{VB_change_of_rings}

We begin by reviewing how the story goes for vector bundle Thom spectra, where we can take advantage of a general
change-of-rings theorem. This is a standard technique dating back to work of Anderson-Brown-Peterson~\cite{ABP69}
and Giambalvo~\cite{Gia73b, Gia73a, Gia76}; see Beaudry-Campbell~\cite[\S 4.5]{BC18} for a nice introduction.
\begin{lem}[Change of rings]
\label{CoR_lemma}
Let $\mathcal{B}$ be a graded Hopf algebra and $\mathcal{C} \subset \mathcal{B}$ be a graded Hopf subalgebra. If
$M$ is a graded $\mathcal C$-module and $N$ is a graded $\mathcal{B}$-module, then there is a natural isomorphism 
\begin{equation}\label{eq:changeofrings}
    \Ext^{s,t}_{\mathcal{B}}(\mathcal{B} \otimes_{\mathcal C} M, N) \overset{\cong}{\longrightarrow}\Ext^{s,t}_{\mathcal{C}}(M, N)
\end{equation}
\end{lem}
For the little siblings we consider, we have the following isomorphisms of $\cA$-modules:
\begin{subequations}
\label{eq:commonCOR}
\begin{align}
	H^*(H\Z;\Z/2) &\cong \cA \otimes_{\cA(0)}\Z/2\\
	H^*(\ku;\Z/2) &\cong \cA \otimes_{\cE(1)} \Z/2\\
	H^*(\ko; \Z/2) &\cong \cA \otimes_{\cA(1)} \Z/2\\
	H^*(\tmf; \Z/2) &\cong \cA \otimes_{\cA(2)} \Z/2.
\end{align}
\end{subequations}
Here $\cA(n)$ is the subalgebra of $\cA$ generated by $\set{\Sq^1,\Sq^2,\Sq^4,\dotsc,\Sq^{2^n}}$ and $\cE(1) =
\ang{Q_0, Q_1}$, where $Q_0 = \Sq^1$ and $Q_1 = \Sq^1\Sq^2 + \Sq^2\Sq^1$. The isomorphisms in~\eqref{eq:commonCOR}
were proven by Wall~\cite[\S 9]{Wal60} ($H\Z$), Adams~\cite{adams_1961} ($\ku$), Stong~\cite{Stong63} ($\ko$), and
Hopkins-Mahowald~\cite{HM14} ($\tmf$).

To use \cref{CoR_lemma}, we need to make $\cA(0)$, $\cA(1)$, $\cA(2)$, and $\cE(1)$
into Hopf subalgebras of $\cA$. This is equivalent to specifying how these algebras interplay with the cup product,
which the Cartan formula answers. For the Steenrod squares, this is standard; we also have $Q_i(ab) = aQ_i(b) +
Q_i(a)b$ for $i = 0,1$.

\Cref{CoR_lemma}, paired with \eqref{eq:commonCOR}, greatly simplifies many computations: for any spectrum which splits as $ X= R \wedge Y$ where $R$ is
one of $H\Z$, $\ku$, $\ko$, or $\tmf$, the $E_2$-page of the Adams spectral sequence computing the $2$-completed homotopy groups of $X$ (or the
$R$-homology of $Y$) is identified with Ext groups over $\cA(0)$, $\cE(1)$, $\cA(1)$, or $\cA(2)$, respectively. These algebras are much smaller than the entire
$2$-primary Steenrod algebra, so the Ext groups are easier to calculate; thus one often hears the slogan that
$\ko$-, $\ku$-, and $\tmf$-homology groups are relatively easy to compute with the Adams spectral
sequence,\footnote{The Adams spectral sequence computing $H\Z$-homology is essentially a repackaging of the
Bockstein spectral sequence; see May-Milgram~\cite{MM81}.} and by \eqref{eq:bigandlittle} and \eqref{eq:commonCOR},
those computations also compute \spinc, spin, and string bordism (the latter in dimensions $15$ and below). See
Douglas-Henriques-Hill~\cite{DHH11} for a nice related computation of vector bundle twists of string bordism.
\begin{rem}
\label{WWc_defn}
Another way to phrase this is that, though~\eqref{eq:commonCOR} is about the little siblings only, combining it
with~\eqref{eq:bigandlittle} allows us to write down change-of-rings results for the Adams spectral sequences of
the big siblings. Specifically, there is an $\cA(0)$-module $W_1$, an $\cE(1)$-module $W_2$, and an $\cA(1)$-module
$W_3$ such that
\begin{subequations}
\begin{align}
	H^*(\MTSO;\Z/2) &\cong \cA\otimes_{\cA(0)} W_1\\
	H^*(\MTSpin^c;\Z/2) &\cong\cA\otimes_{\cE(1)} W_2\\
	H^*(\MTSpin; \Z/2) &\cong \cA\otimes_{\cA(1)} W_3,
\end{align}
\end{subequations}
so that the $E_2$-pages of the Adams spectral sequences computing the $2$-completions of $\Omega_*^\SO$,
$\Omega_*^{\Spin^c}$, and $\Omega_*^\Spin$ are the Ext groups of $W_1$, $W_2$, and $W_3$, respectively, over
$\cE(1)$, $\cA(1)$, and $\cA(2)$ respectively.
Explicitly, these modules begin in low degrees with (compare~\eqref{eq:bigandlittle})
\begin{subequations}
\begin{align}
	W_1 &\cong \Z/2 \oplus \Sigma^4\Z/2 \oplus \Sigma^5 \cA(0) \oplus \Sigma^8\Z/2 \oplus \Sigma^8 \Z/2
	\oplus\dotsb\\
	W_2 &\cong \Z/2 \oplus \Sigma^4 \Z/2 \oplus\Sigma^8\Z/2 \oplus \Sigma^8\Z/2\oplus
	\Sigma^{10}\cE(1)\oplus\dotsb\\
	W_3 &\cong \Z/2\oplus \Sigma^8\Z/2 \oplus \Sigma^{10}\cA(1)/\Sq^3 \oplus\dotsb
\end{align}
\end{subequations}
\end{rem}

Often, though, what one wants is twisted. For vector bundle twists in the sense of \cref{VB_Thom}, this is not a
problem: if $f\colon X\to B\GL_1(R)$ is a vector bundle twist specified by a rank-$r$ virtual vector bundle $V\to
X$, or strictly speaking by the rank-$0$ virtual vector bundle $V - r\coloneqq V - \underline\R^r$, then $f$
factors through $B\GL_1(\Sph)$, so \cref{Thom_change_of_rings} provides a natural homotopy equivalence\footnote{For
a different, less abstract proof of this splitting, see~\cite[\S 10]{FH21a} or~\cite[\S 10.4]{DDHM23}.}
 \begin{equation}
     Mf\overset{\simeq}{\longrightarrow} R\wedge X^{V - r}.
 \end{equation}
Thus, for the ring spectra $R$ we discussed above, one can also use the change-of-rings isomorphism to simplify the computation of twisted $R$-homology for vector bundle twists: for $\ko$, the $E_2$-page is
\begin{equation}
    E^{s,t}_2 = \Ext^{s,t}_{\cA(1)}(H^*(X^{V-r}; \Z/2), \Z/2) \Longrightarrow ko_{t-s}(X)^{\wedge}_2,
\end{equation}
and the other choices of $R$ are analogous. The $\cA$-action (and hence also the $\cA(n)$ and $\cE(1)$-actions) on $H^*(X^{V-r};\Z/2)$ is easy to compute: the Thom isomorphism tells us the cohomology as a vector space, and the Stiefel-Whitney classes of $V$ twist the Steenrod squares as described in~\cite[Remark 3.3.5]{BC18}.

This is a powerful generalization: many bordism spectra of interest arise as twists in this way, including pin\textsuperscript{$\pm$} bordism and all of the bordism spectra studied in~\cite{BG97, Cam17, WW19, WWZ20, Deb21, FH21a}.

\subsection{Baker-Lazarev's $R$-module Adams spectral sequence}\label{sub:BLASS}
For $R$ an $E_\infty$-ring spectrum,\footnote{Baker-Lazarev work with commutative algebras in
Elmendorf-Kriz-Mandell-May's $\Sph$-modules; as we discussed in Footnote~\ref{footnote:spectra}, we may
equivalently work with $E_\infty$-ring spectra.} Baker-Lazarev~\cite{baker2001adams} develop an $R$-module spectrum
generalization of the Adams spectral sequence which reduces to the usual Adams spectral sequence when $R = \Sph$.

\begin{defn}
For $R$-modules $H$ and $M$, the \term{$R$-module $H$-homology} of $M$ is
\begin{subequations}
\begin{equation}
	H^R_*(M) \coloneqq \pi_*(H \wedge_R M),
\end{equation}
and the \term{$R$-module $H$-cohomology} of $M$ is
\begin{equation}
	H_R^*(M) \coloneqq \pi_{-*}\mathrm{Map}_R(M, H).
\end{equation}
\end{subequations}
\end{defn}
For the purposes of this paper, $R$ will be one of the little siblings. For each such $R$, there is a canonical
isomorphism $\pi_0(R)\overset\cong\to\Z$, which lifts to identify the Postnikov quotient $\tau_{\le
0}R\overset\simeq\to H\Z$; as $\tau_{\le 0}R$ is an $E_\infty$ $R$-algebra spectrum via the quotient map $R\to\tau_{\le 0}R$ (see~\cite{Kri93, Bas99}),
this data provides a canonical $E_\infty$ $R$-algebra structure on $H\Z$. Composing with the mod $n$ reduction map $H\Z\to
H\Z/n$, we also obtain canonical $E_\infty$ $R$-algebra structures on $H\Z/n$ for all $n$. This data makes both $H_R^*H$ and $H^R_*H$ into Hopf algebras, analogously to how the Steenrod algebra $H^*(H\Z/p;\Z/p)$ and its dual are Hopf algebras (see~\cite{Mil58}). For $n = 2$ we have the
following isomorphisms of ``Hopf algebras of $R$-module cohomology operations:''
\begin{thm}
\label{HRH_calc}
Let $R$ be one of the little siblings and $H = H\Z/2$ with the $R$-algebra structure defined above. Then there are
Hopf algebra isomorphisms
\begin{subequations}\label{eq:HRH}
\begin{alignat}{2}
    R &= H\Z\,, \quad &H_R^{*}H &\cong \cA(0)\\
     R &= \ko\,, \quad &H_R^{*}H &\cong \cA(1)\\
      R &= \ku\,, \quad &H_R^{*}H &\cong \cE(1)\\
       R &= \tmf\,, \quad &H_R^{*}H &\cong \cA(2),
\end{alignat}
\end{subequations}
and dualizing gives the corresponding Hopf algebras of homology operations, e.g.\ $H^{H\Z}_*H\cong \cA(0)_*$.
\end{thm}
This theorem was proven in pieces: the part for $H\Z$ is standard;
for $\ko$ and $\ku$ this is due to Baker~\cite[Theorem 5.1]{baker2020homotopy}; and for $\tmf$ it is due to
Henriques~\cite{DFHH14}.

In the setting of \cref{HRH_calc}, for any $R$-module spectrum $M$, $H_R^*(M)$ is naturally an $H_R^*H$-module and $H^R_*(H)$ is naturally an $H^R_*H$-comodule, analogously to the mod $2$ cohomology and homology of a spectrum with respect to the Steenrod algebra and its dual.

At this point, we detour briefly to compare $H^R_*$, for $R$ one of the little siblings, with the \term{$H_*(R;\Z/2)$-module indecomposables} functor~\cite[\S 5]{Sto92}.
\begin{defn}[{Stolz~\cite[\S 5]{Sto92}}]
\label{indecomp_defn}
Let $\cA_*$ denote the $2$-primary dual Steenrod algebra, $\cB_*$ be a sub-Hopf algebra of $\cA_*$, and $R$ be an $E_\infty$-ring spectrum such that there is an isomorphism of both $\cA_*$-comodules and $\Z/2$-algebras
\begin{equation}
\label{cotensor_R}
    H_*(R;\Z/2) \overset\cong\longrightarrow \cA_* \mathbin{\square}_{\cB_*} \Z/2,
\end{equation}
where $\mathbin{\square}_{\cB_*}$ denotes the cotensor product of $\cB_*$-comodules. If $N$ is a bounded-below, finite-type $R$-module spectrum, the \term{$H_*(R;\Z/2)$-module indecomposables} of $N$ is the $\cB_*$-comodule
\begin{equation}
    \overline{H_*(N)}\coloneqq H_*(N;\Z/2)\otimes_{H_*(R;\Z/2)} \Z/2.
\end{equation}
\end{defn}
Stolz~\cite[Proposition 5.4]{Sto92} showed that if $N$ is as in \cref{indecomp_defn}, there is a natural isomorphism
\begin{equation}
    H_*(N;\Z/2) \overset\cong\longrightarrow \cA_*\mathbin{\square}_{\cB_*} \overline{H_*(N)}.
\end{equation}
See Stolz~\cite[\S 5]{Sto92} for further discussion with $R = \ko$, Führing~\cite[\S 5]{Fuh22} for $R = H\Z$, and Granath~\cite[\S 2.9]{Gra23} for $R = \ku$.

The little siblings $H\Z$, $\ku$, $\ko$, and $\tmf$ all satisfy~\eqref{cotensor_R} with $\cB_*$ equal to $\cA(0)_*$, $\cE(1)_*$, $\cA(1)_*$, and $\cA(2)_*$ respectively; this follows formally by dualizing~\eqref{eq:commonCOR}.
\begin{prop}
\label{stolz_BL}
Let $R$ be one of $H\Z$, $\ku$, $\ko$, or $\tmf$, so that $\cB_*\cong H^R_*H$ by \cref{HRH_calc}. The functors $H^R_*$ and $\overline{H^*(\bl)}$, from $R$-module spectra to $H^R_*H$-comodules, are naturally isomorphic.
\end{prop}
\begin{proof}
In this proof, all cohomology has $\Z/2$ coefficients.
Consider the Künneth spectral sequence (see \cite[Theorem IV.4.1]{EKMM97} and~\cite{Til16})\footnote{This version of the Künneth spectral sequence is associated to the smash product $(H\wedge H)\wedge_{H\wedge R} (H\wedge N)\simeq H\wedge_R N$. See Lawson~\cite[Proposition 2.7.5]{Law18} or Senger~\cite[\S 3]{Sen24}.}
\begin{equation}
\label{first_kn}
    E^2_{*,*} = \Tor^{H_*(R)}_{*,*}(\Z/2, H_*(N)) \Longrightarrow \pi_*(H\wedge_R N) = H^R_*(N).
\end{equation}
To prove the proposition, it suffices to show that these $\Tor$ groups vanish in positive homological degrees: then the spectral sequence collapses for degree reasons to imply an isomorphism
\begin{equation}
    \Z/2\otimes_{H_*(R)} H_*(N) = \Tor_{0,*}^{H_*(R)}(\Z/2, H_*(N)) \overset\cong\longrightarrow H^R_*(N).
\end{equation}
Proving the claimed higher Tor vanishing is not so hard: the natural isomorphism $H_*(N)\cong H_*(R)\otimes\overline{H_*(N)}$~\cite[\S 5]{Sto92} simplifies the $E^2$-page of~\eqref{first_kn}:
\begin{equation}
    E^2_{*,*} \cong \Tor_{*,*}^{H_*(R)}(\Z/2, H_*(R)\otimes \overline{H_*(N)}) \cong\Tor_{*,*}^{\Z/2}(\Z/2, \overline{H_*(N)}),
\end{equation}
and $\Tor$ over a field vanishes in positive homological degrees.
\end{proof}
%
%
%
\begin{rem}
Despite the equivalence in \cref{stolz_BL}, the two homology theories $H^R_*$ and $\overline{H_*(\bl)}$ have different strengths. The definition of $H^R_*$ makes it easier to use for the applications we have in mind, and $\overline{H_*(\bl)}$ is more generally applicable to ``homology $R$-modules'' (see Stolz~\cite[\S 2]{Sto94}). As we do not need this generality, we stick with $H^R_*$ and $H_R^*$.
\end{rem}


We now present the spectral sequence; let $M$ and $N$ be $R$-modules, and let $H$ be a commutative $R$-ring spectrum.
\begin{thm}[{Baker-Lazarev~\cite{baker2001adams}}]
\label{baker_laz}
Let $M$ and $N$ be $R$-modules and $H$ be an $E_\infty$ $R$-algebra, and suppose that $H^R_*H$ is a flat
$\pi_*(H)$-module. Then there is a spectral sequence of Adams type, natural in $M$, $N$, $H$, and $R$, with
$E_2$-page
\begin{equation}
    E^{s,t}_2 = \Ext^{s,t}_{H_R^* H}(H^*_R M,H^*_R N),
\end{equation}
and if $N$ is connective and $M$ is a cellular $R$-module spectrum with finitely many cells in each
degree,\footnote{This condition on $M$ is the analogue in $\cat{Mod}_R$ of the notion of a CW spectrum with
finitely many cells in each degree. If $M$ is the $R$-module Thom spectrum associated to a map $f\colon X\to
B\GL_1(R)$, which is the only case we consider in this paper, then this condition on $M$ is met if $X$ is a CW
complex with finitely many cells in each dimension.} then this spectral sequence converges to the homotopy groups
of the ($R$-module) $H$-nilpotent completion of $N_*^RM$.
\end{thm}
Without the flatness assumption, one in general only has a description of the $E_1$-page, and it is more
complicated,\footnote{In applications, this may be less bad than it seems: for example, McNamara-Reece~\cite[\S
6.2]{MR22} interpret the $E_1$-page of the classical Adams spectral sequence in the context of quantum gravity.} though see also recent work of Burklund-Pstrągowski~\cite{BP25}.
For example, this issue occurs when $R = \Sph$ and $H = \ku$, $\ko$, or $\tmf$; see~\cite{Mah81, Dav87, LM87,
BOSS19, BBBCX20, BBBCX21}. However, if $p$ is a prime number, $\pi_*(H\Z/p)\cong\Z/p$ is a field, so the flatness
assumption is satisfied for all $R$; as this is the only case we consider in this paper, we say no more about the
flatness assumption in \cref{baker_laz}.

The notion of the \term{$H$-nilpotent completion} of a spectrum is due to Bousfield~\cite[\S 5]{Bou79}. When $H =
H\Z/p$, $p$ prime, this is the usual $p$-completion~\cite[Example 1.16]{Rav84}.\footnote{Ravenel assumes $R = \Sph$, but his result is true in the generality we work in.} Thus if the homotopy groups of $N
\wedge_R M$ are finitely generated abelian groups, this as usual detects free and $p^k$-torsion summands, but not
torsion for other primes.

When $R = \Sph$, \cref{baker_laz} reduces to the classical $H$-based Adams spectral sequence, with its standard
convergence results. We will apply \cref{baker_laz} when $R$ is one of the little siblings, $H = H\Z/p$ for $p$
prime, and $N = R$: there is a canonical homotopy equivalence $R\wedge_R M\simeq M$, so in this setting
Baker-Lazarev's spectral sequence takes as input $\Ext_{H_R^*H}(H_R^*(M), \Z/2)$, and converges to the
$p$-completed homotopy groups of $M$.


For Thom spectra $H_R^*$ is easy.
\begin{lem}[$R$-module Thom isomorphisms]
 \label{Rmod_Thom_iso}
   For any $E_\infty$-ring spectrum $R$ such that $H\coloneqq H\Z/2$ is an $R$-algebra and any map $f\colon X\to B\GL_1(R)$, there are isomorphisms
\begin{subequations}
\begin{gather}
\label{homological_Thom}
    H_*(X;\Z/2)\overset\cong\longrightarrow H^R_*(Mf)\\
    H^*(X;\Z/2)\overset\cong\longrightarrow H_R^*(Mf).
\end{gather}
\end{subequations}
\end{lem}
This means that $H_R^*(Mf)$ is a free $H^*(X;\Z/2)$-module on a class $U\in H_R^0(Mf)$, which is the Thom class in
this setting.
\begin{proof}
Apply \cref{Thom_change_of_rings} with $R_1 = R$ and $R_2 = H\Z/2$ to learn that $Mf\wedge_R H\Z/2$, the object
whose homotopy groups are $H^R_*(Mf)$, is the Thom spectrum of a twist $f'\colon X\to B\GL_1(H\Z/2)$. By
\cref{discrete_example}, $B\GL_1(H\Z/2)$ is contractible, so $f'$ is null-homotopic, so by \cref{trivial_twists},
$Mf\wedge_R H\Z/2\simeq X_+\wedge H\Z/2$. Take homotopy groups to obtain~\eqref{homological_Thom}.

For cohomology,\footnote{We thank an anonymous referee for a suggestion to simplify this part of the proof.} we have a chain of equivalences of spectra
\begin{equation}\label{hom_2_coh}
    \begin{aligned}
        \Map_R(Mf, H\Z/2) &\simeq \Map_{H\Z/2}(Mf\wedge_R H\Z/2, H\Z/2)\\
        &\simeq \Map_{H\Z/2}((X_+)\wedge H\Z/2, H\Z/2)\\
        &\simeq \Map_\Sph(X_+, H\Z/2),
    \end{aligned}
\end{equation}
and the claim follows by taking homotopy groups. The first and third equivalences in~\eqref{hom_2_coh} are instances of the natural isomorphism $\Map_A(B, C)\simeq \Map_E(B\wedge_A E, C)$ for an $E_\infty$-ring spectrum $A$, an $E_\infty$ $A$-algebra spectrum $E$, and $A$-modules $B$ and $C$, and the middle equivalence in~\eqref{hom_2_coh} is the homology Thom isomorphism from the first part of the proof.
%
\end{proof}

For most of our applications we will take $H = H \Z/2$.
\begin{exm}[$\tmf$ at the prime $3$]
\label{intro_Atmf}
We will also work with an interesting odd-primary example, where $H = H\Z/3$ and $R = \tmf$. Let $\cA_3 \coloneqq
H^*H$, which is the mod $3$ Steenrod algebra, and let $\cA^{\tmf}\coloneqq H^*_\tmf H$; Henriques and Hill, using
the work of Behrens \cite{behrens2006modular} and unpublished work of Hopkins-Mahowald, showed that
\begin{equation}\label{eq:Atmf_relations}
    \cA^{\tmf}\cong \Z/3\ang{\beta, \cP^1}/(\beta^2, \beta(\cP^1)^2\beta - (\beta\cP^1)^2 - (\cP^1\beta)^2, (\cP^1)^3).
\end{equation}
Curiously, Rezk showed that $H^*(\tmf;\Z/3)$ is not isomorphic to $\cA_3\otimes_{\cA^{\tmf}}\Z/3$: see~\cite[\S 2]{Cul21}.

The map $\phi\colon H_\tmf^*H\to H^*H$ sends $\beta$ to the Bockstein of $0\to\Z/3\to\Z/9\to\Z/3\to 0$ and $\cP^1$ to the first Steenrod power. However, unlike in the previous examples we studied, $\phi$ is not injective! The relation $\beta(\cP^1)^2 + \cP^1\beta\cP^1 + (\cP^1)^2\beta = 0$ is present in $\cA_3$ but not in $\cA^{\tmf}$ (see, e.g.,~\cite[Corollary 13.7]{BR21}).

Baker-Lazarev's \cref{baker_laz} implies that for any $\tmf$-module spectrum $M$, $H_\tmf^*(M)$ carries a natural $\cA^{\tmf}$-module action, and there is an Adams spectral sequence
\begin{equation}
    E_2^{s,t} = \Ext_{\cA^{\tmf}}^{s,t}(H_\tmf^*(M), \Z/3) \Longrightarrow \pi_{t-s}(M)_3^\wedge.
\end{equation}
In general, we will let $H_\tmf^*(M)$ refer to the mod $2$ $\tmf$-module cohomology and denote the mod $3$
$\tmf$-module cohomology by $H_\tmf^*(M; \Z/3)$. Because $(\Z/3)^\times$ is nontrivial, $B\GL_1(H\Z/3)$ is not
contractible, so the proof of \cref{Rmod_Thom_iso} does not directly generalize to this setting; however, as
$B\GL_1(H\Z/3)\cong B(\Z/3)^\times$ (see \cref{discrete_example}), for any twist $f\colon X\to B\GL_1(\tmf)$
factoring through a simply connected space, the induced twist of $H\Z/3$ is trivial and the argument goes through
to show $H_{\tmf}^*(M^\tmf f;\Z/3)\cong H^*(X;\Z/3)$. As $\SK(4)$ is simply connected, this includes the fake
vector bundle twists of $\tmf$ whose components in $H^1(\bl;\Z/2)$ vanish.

Like for the mod $2$ subalgebras of the Steenrod algebra that we discussed, we will want to know how $\cA^{\tmf}$ acts on products. The map $\cA^{\tmf}\to\cA_3$ is a map of Hopf algebras~\cite[\S 13.1]{BR21}, allowing us to use the Cartan formula and multiplicativity of the Bockstein in $\cA_3$ to conclude that in $\cA^{\tmf}$,
\begin{subequations}
    \begin{align}
        \cP^1(ab) &= \cP^1(a)b + a\cP^1(b)\,,\\
        \beta(ab) &= \beta(a)b + (-1)^{\abs a}a\beta(b).
    \end{align}
\end{subequations}
\end{exm}
When $R$ is one of the little siblings, \cref{HRH_calc} implies that for any $R$-module spectrum $M$,
Baker-Lazarev's spectral sequence calculates $\pi_*(M)_2^\wedge$ as the Ext of \emph{something} over an algebra much
smaller than $\cA$ -- one of $\cA(0)$, $\cE(1)$, $\cA(1)$, or $\cA(2)$. Thus the change-of-rings approach to
computing $\pi_*(R\wedge Y)_2^\wedge$ that we described in \S\ref{VB_change_of_rings} generalizes to other
$R$-modules $M$, in particular when $M$ is an $R$-module Thom spectrum -- we just have to figure out $H_R^*(M)$.
This will be the main result of the next section.



\subsection{Proof of the main theorem}\label{sub:mainthm}
At this point, we know from the previous section that even for non-vector-bundle Thom spectra $M^Rf$ over $R =
H\Z$, $\ku$, $\ko$ and $\tmf$, we can work over $\cA(0)$, $\cE(1)$, $\cA(1)$, and $\cA(2)$ to compute the $E_2$-page of
Baker-Lazarev's Adams spectral sequence, implying that a change-of-rings formula for these Thom spectra exists.
Our next step is to
determine the $\cA(0)$-, $\cE(1)$-, $\cA(1)$-, and $\cA(2)$-modules $H_R^*(M^Rf)$. We describe the actions of the generators of
$\cA(0)$, $\cE(1)$, $\cA(1)$, and $\cA(2)$ below in \cref{the_twisted_modules}; however, it is not yet clear that they
satisfy the Adem relations, so we describe these modules over freer algebras, then later in the proof of
\cref{thom_module_calc} we show they are compatible with the Adem relations, hence are in fact $H_R^*H$-modules.
\begin{defn}
\label{the_twisted_modules}
Let $X$ be a space.
\begin{enumerate}
	\item Given $a\in H^1(X;\Z/2)$, let $M_{H\Z}(a, X)$ be the $\Z/2[s_1]$-module which is a free
	$H^*(X;\Z/2)$-module on a single generator $U$, and with $s_1$-action
	\begin{equation}
		s_1(Ux) \coloneqq U(ax + \Sq^1(x)).
	\end{equation}
	\item Given $a\in H^1(X;\Z/2)$ and $c\in H^3(X;\Z)$, let $M_\ku(a, c, X)$ be the $\Z/2\ang{q_0, q_1}$-module
	which is a free $H^*(X;\Z/2)$-module on a single generator $U$, and with $q_0$- and $q_1$-actions given by
	\begin{equation}\label{eq:kuActions}
	\begin{aligned}
		q_0(Ux) &\coloneqq U(ax + Q_0(x))\\
		q_1(Ux) &\coloneqq U((c\bmod 2+a^3)x + Q_1(x)).
	\end{aligned}
	\end{equation}
	\item Given $a\in H^1(X;\Z/2)$ and $b\in H^2(X;\Z/2)$, let $M_\ko(a, b, X)$ be the $\Z/2\ang{s_1, s_2}$-module
	which is a free $H^*(X;\Z/2)$-module on a single generator $U$, and with $s_1$- and $s_2$-actions
	\begin{equation}
	\label{ko_s1s2}
	\begin{aligned}
		s_1(Ux) &\coloneqq U(ax + \Sq^1(x))\\
		s_2(Ux) &\coloneqq U(bx + a\Sq^1(x) + \Sq^2(x)).
	\end{aligned}
	\end{equation}
	\item Given $a\in H^1(X;\Z/2)$, and $d\in SH^4(X)$, let $M_\tmf(a,  d, X)$ be the
	$\Z/2\ang{s_1, s_2, s_4}$-module
	which is a free $H^*(X;\Z/2)$-module on a single generator $U$, with $s_1$- and $s_2$-actions given
	by~\eqref{ko_s1s2} with $b = t(d)$, and $s_4$-action given by
	\begin{equation}
		s_4(Ux) = U(\delta x + t(d) a+\Sq^1(t(d))) \Sq^1(x) + t(d)\Sq^2(x) + a\Sq^3(x) + \Sq^4(x)).
	\end{equation}
	\item Given $d\in \SH^4(X)$, let $M_\tmf'(d, X)$ be the $\Z/3\ang{\beta, p_1}/(\beta^2)$-module which is a
	free $H^*(X;\Z/2)$-module on a single generator $U$ and $\beta$- and $p_1$-actions specified by
	\begin{equation}
	\begin{aligned}
		\beta(Ux) &\coloneqq U\beta(x)\\
		p_1(Ux) &\coloneqq U((d\bmod 3)x + \cP^1(x)).
	\end{aligned}
	\end{equation}
\end{enumerate}
\end{defn}
The mod $3$ reduction of the supercohomology class $d$ is defined as usual as the image of $d$ after passing to the
mod $3$ Moore spectrum $\Sph/3$:
\begin{equation}
	[X, \Sigma^4\SH] \longrightarrow [X, \Sigma^4\SH\wedge \Sph/3]\overset\cong\longrightarrow [X, \Sigma^4 H\Z/3],
\end{equation}
because $H\Z/2\wedge \Sph/3\simeq 0$. Thus $d\bmod 3$ is well-defined as a class in $H^4(X;\Z/3)$.
\begin{lem}
Keep the notation from \cref{the_twisted_modules}.
\begin{enumerate}
	\item The action of $s_1$ on $M_{H\Z}(a, X)$ squares to $0$, so the $\Z/2[s_1]$-module structure on $M_{H\Z}(a,
	X)$ refines to an $\cA(0)$-module structure with $\Sq^1(x)\coloneqq s_1(x)$.
	\item The actions of $q_0$ and $q_1$ on $M_\ku(a, c, X)$ commute and both square to $0$, so the $\Z/2\ang{q_0,
	q_1}$-module structure on $M_\ku(a, c, X)$ refines to an $\cE(1)$-module structure, where for $i = 0,1$,
	$Q_i(x)\coloneqq q_i(x)$.
	\item The actions of $s_1$ and $s_2$ on $M_\ko(a, b, X)$, and of $s_1$, $s_2$, and $s_4$ on $M_\tmf(a, b, c,
	X)$, satisfy the Adem relations with $s_i$ in place of $\Sq^i$, hence refine to an $\cA(1)$-module structure on
	$M_\ko(a, b, X)$ and an $\cA(2)$-module structure on $M_\tmf(a, c, d, X)$.
	\item The actions of $\beta$ and $p^1$ on $M_\tmf'(c, X)$ satisfy the relations in \eqref{eq:Atmf_relations}, hence refine the $\Z/3\ang{\beta, p^1}/(\beta^2)$-module structure on $M_\tmf'(c, X)$ to
	an $\cA^\tmf$-module structure, where the Bockstein acts as $\beta$ and $\cP^1$ acts as $p^1$.
\end{enumerate}
\end{lem}
Rather than prove this directly, we will obtain it as a corollary of \cref{thom_module_calc}. This theorem
says that the modules defined in \cref{the_twisted_modules} are $H_R^*$ of the Thom spectra for the corresponding
twists.
\begin{thm}
Let $X$ be a topological space.
\label{thom_module_calc}
\begin{enumerate}
	\item Given $a\in H^1(X;\Z/2)$, let $f_a\colon X\to B\GL_1(H\Z)$ be the corresponding fake vector bundle twist.
	Then there is an isomorphism of $\cA(0)$-modules
	\begin{equation}
		H_{H\Z}^*(M^{H\Z} f_{a})\overset\cong\longrightarrow M_{H\Z}(a, X).
	\end{equation}
	\item Given $a\in H^1(X;\Z/2)$ and $c\in H^3(X;\Z)$, let $f_{a,c}\colon X\to B\GL_1(\ku)$ be the corresponding
	fake vector bundle twist. Then there is an isomorphism of $\cE(1)$-modules
	\begin{equation}
		H_\ku^*(M^\ku f_{a,c})\overset\cong\longrightarrow M_{\ku}(a, c, X).
	\end{equation}
	\item Given $a\in H^1(X;\Z/2)$ and $b\in H^2(X;\Z/2)$, let $f_{a,b}\colon X\to B\GL_1(\ko)$ be the
	corresponding fake vector bundle twist. Then there is an isomorphism of $\cA(1)$-modules
	\begin{equation}
		H_\ko^*(M^\ko f_{a,b})\overset\cong\longrightarrow M_{\ko}(a, b, X).
	\end{equation}
	\item \label{mainthm:tmf} Given $a\in H^1(X;\Z/2)$, and $d\in SH^4(X)$, let $f_{a,d}\colon X\to
	B\GL_1(\tmf)$ be the corresponding fake vector bundle twist. Then there is an isomorphism of $\cA(2)$-modules
	\begin{equation}
		H_\tmf^*(M^\tmf f_{a,d})\overset\cong\longrightarrow M_{\tmf}(a, d, X),
	\end{equation}
	and an isomorphism of $\cA^\tmf$-modules
	\begin{equation}
		H_\tmf^*(M^\tmf f_{0, d}; \Z/3)\overset\cong\longrightarrow M_{\tmf}'(d, X).
	\end{equation}
\end{enumerate}
\end{thm}
In the last isomorphism, we turn off degree-$1$ twists so that we have a Thom isomorphism for mod $3$ cohomology.
\begin{cor}
\label{little_sibling_Adams_corollary}
Keep the notation from \cref{thom_module_calc}.
\begin{description}
	\item[Twisted $\Z$-homology] The $E_2$-page of Baker-Lazarev's Adams spectral sequence computing $\pi_*(M^{H\Z}
	f_{a})_2^\wedge$ is isomorphic as $\Ext_{\cA(0)}^{*,*}(\Z/2, \Z/2)$-modules to $\Ext_{\cA(0)}^{s,t}(M_{H\Z}(a, X), \Z/2)$.
	\item[Twisted $\ku$-homology] The $E_2$-page of Baker-Lazarev's Adams spectral sequence computing $\pi_*(M^\ku
	f_{a,c})_2^\wedge$ is isomorphic as $\Ext_{\cE(1)}^{*,*}(\Z/2, \Z/2)$-modules to $\Ext_{\cE(1)}^{s,t}(M_\ku(a, c, X),
	\Z/2)$.
	\item[Twisted $\ko$-homology] The $E_2$-page of Baker-Lazarev's Adams spectral sequence computing $\pi_*(M^\ko
	f_{a,b})_2^\wedge$ is isomorphic as $\Ext_{\cA(1)}^{*,*}(\Z/2, \Z/2)$-modules to $\Ext_{\cA(1)}^{s,t}(M_\ko(a, b, X),
	\Z/2)$.
	\item[Twisted $\tmf$-homology]\hfill
	\begin{enumerate}
		\item The $E_2$-page of Baker-Lazarev's Adams spectral sequence computing $\pi_*(M^\tmf f_{a, d})_2^\wedge$
		is isomorphic as $\Ext_{\cA(2)}^{*,*}(\Z/2, \Z/2)$-modules to $\Ext_{\cA(2)}^{s,t}(M_\tmf(a ,d, X), \Z/2)$.
		\item The $E_2$-page of Baker-Lazarev's Adams spectral sequence computing $\pi_*(M^\tmf f_{0, d})_3^\wedge$
		is isomorphic as $\Ext_{\cA^\tmf}^{*,*}(\Z/3, \Z/3)$-modules to $\Ext_{\cA^\tmf}^{s,t}(M_\tmf'(d, X), \Z/3)$.
	\end{enumerate}
\end{description}
\end{cor}
\begin{rem}
In \S\ref{so_twists} we saw that the twists of $H\Z$ discussed above are all vector bundle twists, so that the
$H\Z$ part of \cref{little_sibling_Adams_corollary} follows from the standard change-of-rings argument; the same is
true for the twists of $\MTSO$ appearing below in \cref{big_sibling_Adams_corollary}. In both cases, the other
calculations are new.
\end{rem}
\begin{rem}
The analogue of \cref{little_sibling_Adams_corollary} is true for a few standard variants of the Adams spectral
sequence. For example, one could switch the order of $H_R^*(Mf)$ and $\Z/2$ in $\Ext_{H_R^*H}$ and obtain the
$E_2$-page of Baker-Lazarev's Adams spectral sequence computing twisted $R$-cohomology for twists over a finite type space. One could also work out a version of \cref{little_sibling_Adams_corollary} in terms of $R$-module $H$-homology with its $H^R_*H$-comodule
structure.
\end{rem}
Now, given a big sibling and little sibling pair $M\to R$, we lift to $M$. While it would be nice to completely describe the $M$-module Baker-Lazarev Adams spectral sequences for $M = \MTSO$, $\MTSpin^c$, $\MTSpin$, and $\MTString$, this ranges between very complicated and intractable. This is because these Adams spectral sequences would in principle determine the ring structures on $M_*$ for these spectra $M$, which are not presently known for $\MTSpin^c$, $\MTSpin$, and $\MTString$ and which is intricate for $\MTSO$.\footnote{See Abdallah-Salch~\cite{AS24} for recent progress in the \spinc case.} Thus we provide two different lifts of \cref{little_sibling_Adams_corollary}:
\begin{enumerate}
    \item In \cref{big_sibling_Adams_estimate}, we use the connectivity of the orientations from the big to the little siblings to \emph{partially} calculate the Baker-Lazarev Adams spectral sequence for each of the big siblings.
    \item In \cref{big_sibling_Adams_corollary}, we use the splittings of $M$ that we reviewed in~\eqref{eq:bigandlittle} to noncanonically describe $M$-module Thom spectra as sums of $R$-module Thom spectra, and therefore obtain an $R$-module Baker-Lazarev Adams spectral sequence that computes spectrum-level information about $M$-module Thom spectra. This does not work for $\MTString$, which has not been split at $2$.\footnote{While there is substantial evidence suggesting that, $2$-locally, $\MTString$ splits as a wedge sum of $\tmf$, $\Sigma^{16}\tmf_0(3)$, and other pieces, for example in~\cite{Pen83, MG95, MH02, MR09, Lau04, Lau16, LO16, LO18, LS19, Dev19, Abs21, Dev20, Tok24}, it is not a foregone conclusion that a splitting exists. For example, Kochman~\cite[Part 1, Theorem 5.4]{Koc93} proved that the symplectic bordism spectrum $\mathit{MTSp}$ is indecomposable at the prime $2$.}
\end{enumerate}
Before we compare the Baker-Lazarev Adams spectral sequences for the big and little siblings, we need a few facts in homological algebra.
\begin{lem}
\label{first_step_resolution}
Let $k$ be a field and $A_1$ and $A_2$ be $\Z$-graded $k$-algebras concentrated in nonnegative degrees. Suppose that we have the following data for some positive integer $n$:
\begin{enumerate}
    \item A $k$-algebra homomorphism $\phi\colon A_1\to A_2$ which is a $k$-vector space isomorphism in degrees $\le n$.
    \item $A_i$-modules $M_i$ concentrated in nonnegative degrees, and an $A_1$-module homomorphism $\psi\colon M_1\to M_2$ which is an isomorphism in degrees $\le n$.
\end{enumerate}
Then for $i = 1,2$, there are free $A_i$-modules $F_i$ and surjective $A_i$-module homomorphisms $\chi_i\colon F_i\to M_i$, together with an $A_1$-module map $\theta\colon F_1\to F_2$ which is an isomorphism in degrees $\le n$, and such that the following diagram commutes:
\begin{equation}
\label{step_one}
\begin{tikzcd}
	{F_1} & {M_1} \\
	{F_2} & {M_2}
	\arrow["{\chi_1}", from=1-1, to=1-2]
	\arrow["\theta"', from=1-1, to=2-1]
	\arrow["\psi"', from=1-2, to=2-2]
	\arrow["{\chi_2}", from=2-1, to=2-2]
\end{tikzcd}\end{equation}
\end{lem}
Here the $A_1$-module structures on $M_2$ and $F_2$ are the ones induced across $\phi$.
\begin{proof}
Present $M_2$ as an $A_2$-module, and let $\chi_2\colon F_2\to M_2$ be the quotient map sending the free $A_2$-module on the generating set $S$ of $M_2$ to their images in $M_2$. Let $F_1$ be the free $A_1$-module on $S$, so that $\phi$ induces the $A_1$-module map $\theta\colon F_1\to F_2$. Since $F_1$ and $F_2$ are concentrated in nonnegative degrees and $\phi$ is an isomorphism in degrees $n$ and below, the same is true of $\theta$.

Now to build $\chi_1$. Let $s\in S$, which we regard as a generator of $F_1$. If $\deg(s) > n$, let $\chi_1(s) = 0$. If $\deg(s)\le n$, $\chi_2(\theta(s))\in M_2$ has a unique preimage under $\psi$, because $\psi$ is an isomorphism in degrees $n$ and below; define $\chi_1(s) \coloneqq\psi^{-1}(\chi_2(\theta(s))$. At this point, we have now verified the entire lemma statement except for surjectivity of $\chi_1$; if $\chi_1$ as constructed has cokernel, add free $A_1$-module summands to $F_1$ so that $\chi_1$ is surjective. Then define $\theta$ on these summands by the requirement that~\eqref{step_one} commutes: since the $A_1$-module structure on $M_2$ is induced from its $A_2$-module via $\phi$, the $A_1$-module map $\psi\circ\chi_1$ factors through $F_1\to F_1\otimes_{A_1}A_2$, which is a free $A_2$-module, so we may extend $\theta$.
\end{proof}
\begin{cor}
\label{compare_resolution}
Suppose for $i = 1,2$, $A_i$, $M_i$, $\phi$, $\psi$, and $n$ are as in \cref{first_step_resolution}. Then there are free resolutions $P_\bullet^{(i)}\to M_i$ of $A_i$-modules and a map $\theta_\bullet\colon P_\bullet^{(1)}\to P_\bullet^{(2)}$ of chain complexes of $A_1$-modules such that for all homological degrees $s$, $\theta_s$ is an isomorphism in grading-degree $n$ and below.
\end{cor}
\begin{proof}
Use \cref{first_step_resolution} to build $d_0^{(i)}\colon P_1^{(i)}\to M_i$ for each $i$ and the map $\theta_1\colon P_1^{(1)}\to P_1^{(2)}$. Then there is an induced $A_1$-module map $\psi'\colon \ker(d_0^{(1)})\to\ker(d_0^{(2)})$ which is an isomorphism in degrees $\le n$, so we can build the next step in the free resolution by applying \cref{first_step_resolution} to $\ker(d_0^{(i)})$ and $\psi'$, and so on.
\end{proof}
Recall the notion of a \term{minimal resolution} of a module over an augmented algebra from, e.g.,~\cite[\S 4.4]{BC18}.
\begin{lem}
\label{augmented}
Keeping the notation and assumptions from \cref{compare_resolution}, now assume in addition that $A_1$ and $A_2$ are augmented algebras, such that each augmentation $A_i\to k$ is an isomorphism when restricted to degree-$0$ elements. Assume also that $\phi$ is a homomorphism of augmented algebras. Then the resolutions $P_{\bullet}^{(i)}\to M_i$ may be chosen to be minimal resolutions such that $\theta_s$ is an isomorphism in grading-degrees $n+s$ and below.
\end{lem}
\begin{proof}
Building the minimal resolutions is exactly as in the proof of \cref{compare_resolution}, thanks to the observation that we can choose the surjections in \cref{first_step_resolution} to satisfy the minimality property.

The assumption that the augmentations $A_i\to k$ are isomorphisms in degree $0$ implies that $P_s^{(i)}$ is concentrated in degrees $s$ and above. Therefore we may shift the grading on $P_s^{(i)}$ down by $s$ before applying \cref{first_step_resolution} in the inductive step of \cref{compare_resolution}, then shift it back up, to obtain an isomorphism in degrees $\le n+s$, as required.
\end{proof}
\begin{lem}
\label{ext_tor_appx}
Suppose $H = H\Z/p$ for a prime $p$ and we have connective, $E_\infty$-ring spectra $R_1$ and $R_2$ with $E_\infty$-ring maps $f\colon R_1\to R_2$ and $g\colon R_2\to H$ such that $f$ is $n$-connected for some $n\ge 1$. Suppose we have connective $R_i$-module spectra $N_i$ and an $n$-connected $R_1$-module map $\varphi\colon N_1\to N_2$. Then the induced maps
\begin{subequations}
\begin{gather}
    \label{tor_appx}
    H^{R_1}_*(N_1) \longrightarrow H^{R_2}_*(N_2)\\
    H_{R_1}^*(N_1) \longrightarrow H_{R_2}^*(N_2)
    \label{ext_appx}
\end{gather}
\end{subequations}
are isomorphisms in degrees $*\le n$.
\end{lem}
\begin{proof}
First~\eqref{tor_appx}. For $i = 1,2$, consider the Künneth spectral sequences
\begin{equation}
    E^2_{*,*} = \Tor_{*,*}^{\pi_*(R_i)}(\Z/p, \pi_*(N_i)) \Longrightarrow H^{R_i}_*(N_i).
\end{equation}
The Künneth spectral sequence is natural in the data of the map $R_i\to H$ and $N_i$, so $f$ induces a map of spectral sequences, i.e.\ a map on each page which commutes with differentials, and which converges to the map $H^{R_1}_*(N_1)\to H^{R_2}_*(N_2)$. \Cref{compare_resolution} implies that the induced map on $E^2$-pages is an isomorphism in grading degree $n$ and below, and therefore also in total degree $n$ and below. This immediately implies the result for~\eqref{tor_appx} (there may be differentials from total degree $n+1$ to total degree $n$, but $n$-connectivity implies surjectivity in degree $n+1$, so those differentials are carried from the first spectral sequence to the second, therefore also implying the isomorphism in degree $n$).

For~\eqref{ext_appx}, use the equivalence $\Map_{R_i}(N_i, H)\overset\simeq\to \Map_H(H\wedge_{R_i} N_i, H)$ to reduce to~\eqref{tor_appx}, similarly to the proof of \cref{Rmod_Thom_iso}.
\end{proof}
\begin{cor}
\label{ext_approx_2}
With notation as in \cref{ext_tor_appx}, assume also that $H_{R_i}^0H\cong\Z/p$. Then the induced map
\begin{equation}
\label{RiMi}
    \Ext_{H_{R_1}^*H}^{s,t}(H_{R_1}^*(N_1), \Z/p) \longrightarrow \Ext_{H_{R_2}^*H}^{s,t}(H_{R_2}^*(N_2), \Z/p)
\end{equation}
is an isomorphism in topological degree $t-s\le n$.
\end{cor}
\begin{proof}
The condition on $H_{R_i}^0H$ implies that $H_{R_i}^*H$ is canonically augmented by the algebra map quotienting by all positive-degree elements. Therefore we may use \cref{augmented} with $A_i = H_{R_i}^*H$, $M_i =\pi_*(N_i)$, $\phi$ the map on $H^*_{(\bl)}H$, and $\psi = \pi_*(\varphi)$.
\end{proof}
Now we provide our first lift of \cref{little_sibling_Adams_corollary} to the big siblings: a computation of the Baker-Lazarev Adams spectral sequence, but only in a range.
\begin{thm}
\label{big_sibling_Adams_estimate}
Keep the notation from \cref{thom_module_calc}.
\begin{description}
	\item[Twisted oriented bordism] In topological degrees $t-s\le 3$, the $E_2$-page of Baker-Lazarev's Adams spectral sequence computing $(\Omega_*^{\SO}(X, a))_2^\wedge$ is isomorphic as $\Ext_{\cA(0)}^{*,*}(\Z/2, \Z/2)$-modules to
	$\Ext_{\cA(0)}^{s,t}(M_{H\Z}(a, X), \Z/2)$.
	\item[Twisted \spinc bordism] In topological degrees $t-s\le 3$, the $E_2$-page of Baker-Lazarev's Adams spectral sequence computing $(\Omega_*^{\Spin^c}(X, a, c))_2^\wedge$ is isomorphic as $\Ext_{\cE(1)}^{*,*}(\Z/2, \Z/2)$-modules to
	$\Ext_{\cE(1)}^{s,t}(M_\ku(a, c, X), \Z/2)$.
	\item[Twisted spin bordism] In topological degrees $t-s\le 7$, the $E_2$-page of Baker-Lazarev's Adams spectral sequence computing $(\Omega_*^{\Spin}(X, a, b))_2^\wedge$ is isomorphic as $\Ext_{\cA(1)}^{*,*}(\Z/2, \Z/2)$-modules to
	$\Ext_{\cA(1)}^{s,t}(M_\ko(a, b, X), \Z/2)$. 
	$\Ext_{\cA(1)}^{s,t}(M_\ko(a, b, X), \Z/2)$.
	\item[Twisted string bordism] In topological degrees $t-s \le 15$, the $E_2$-page of Baker-Lazarev's Adams spectral sequence computing $(\Omega_*^\String(X, a, d))_2^\wedge$, resp.\ $(\Omega_*^\String(X, 0, d))_3^\wedge$, are isomorphic
	to $\Ext_{\cA(2)}^{s,t}(M_\tmf(a, d, X), \Z/2)$, resp.\ $\Ext_{\cA^{\tmf}}^{s,t}(M_\tmf'(d, X), \Z/3)$, as modules
	over $\Ext_{\cA(2)}^{*,*}(\Z/2, \Z/2)$, resp.\ $\Ext_{\cA^{\tmf}}^{*,*}(\Z/3, \Z/3)$.
\end{description}
\end{thm}
\begin{proof}
Each of these is a consequence of \cref{ext_approx_2}, where $R_1$ is the big sibling, $R_2$ is the little sibling, $N_i$ is the $R_i$-module Thom spectrum for the fake vector bundle twist in question, $f$ is the orientation $R_1\to R_2$ introduced in \S\ref{sub:faketwist}, and $\varphi$ is the induced map of Thom spectra. Here $n$ is $3$ for oriented and \spinc bordism, $n = 7$ for spin bordism, and $n = 15$ for string bordism. The only hypothesis we have yet to confirm is that $\varphi$ is $n$-connected, which we do now. By \cref{Thom_change_of_rings}, $N_2\simeq R_2\wedge_{R_1} N_1$ and $\varphi\simeq f\wedge \id\colon N_1\simeq \textcolor{BrickRed}{R_1}\wedge_{R_1}N_1\to \textcolor{MidnightBlue}{R_2}\wedge_{R_1}N_1\simeq N_2$. 
Thus we get an induced map between the following two Künneth spectral sequences:
\begin{subequations}
    \begin{align}
        E^2_{*,*} &= \Tor_{*,*}^{\pi_*(R_1)}(\pi_*(\textcolor{BrickRed}{R_1}), \pi_*(N_1)) \Longrightarrow \pi_*(\textcolor{BrickRed}{R_1}\wedge_{R_1}N_1) = \pi_*(\textcolor{BrickRed}{N_1})\\
        E^2_{*,*} &= \Tor_{*,*}^{\pi_*(R_1)}(\pi_*(\textcolor{MidnightBlue}{R_2}), \pi_*(N_1)) \Longrightarrow \pi_*(\textcolor{MidnightBlue}{R_2}\wedge_{R_1}N_1) = \pi_*(\textcolor{MidnightBlue}{N_2}),
    \end{align}
\end{subequations}
which we apply \cref{compare_resolution} to, similarly to the proof of \cref{ext_tor_appx}. Thus for each of the four cases in the theorem statement, we have verified the hypotheses of \cref{ext_approx_2}; the conclusion of that corollary finishes the proof of this theorem.
\end{proof}
Because \cref{big_sibling_Adams_estimate} only calculates the Baker-Lazarev Adams spectral sequence in a range of degrees, we also provide a version in all degrees for $\MTSO$, $\MTSpin^c$, and $\MTSpin$, which heuristically records the fact that the Wall, resp.\ Anderson-Brown-Peterson splittings of these Thom spectra fiber over $B\O/BH$. Thus these splittings are compatible with fake vector bundle twists.

Recall the modules $W_1$, $W_2$, and $W_3$ from \cref{WWc_defn}.
\begin{cor}
\label{big_sibling_Adams_corollary}
Keep the notation from \cref{thom_module_calc}.
\begin{description}
	\item[Twisted oriented bordism] There is a strongly convergent spectral sequence of Adams type with signature
    \begin{equation}
        E_2^{s,t} = \Ext_{\cA(0)}^{s,t}(M_{H\Z}(a, X)\otimes W_1, \Z/2) \Longrightarrow \Omega_*^\SO(X, a)_2^\wedge.
    \end{equation}
	\item[Twisted \spinc bordism]
    There is a strongly convergent spectral sequence of Adams type with signature
    \begin{equation}
        E_2^{s,t} = \Ext_{\cE(1)}^{s,t}(M_{\ku}(a, c, X)\otimes W_2, \Z/2) \Longrightarrow \Omega_*^{\Spin^c}(X, a, c)_2^\wedge.
    \end{equation}
	\item[Twisted spin bordism]
    There is a strongly convergent spectral sequence of Adams type with signature
    \begin{equation}
        E_2^{s,t} = \Ext_{\cA(1)}^{s,t}(M_{\ko}(a, b, X)\otimes W_3, \Z/2) \Longrightarrow \Omega_*^{\Spin}(X, a, b)_2^\wedge.
    \end{equation}
\end{description}
\end{cor}
All tensor products are taken over $\Z/2$ and given an $\cA(0)$-, $\cE(1)$-, or $\cA(1)$-module structure using the Hopf algebra structure on $\cA(0)$, $\cE(1)$, and $\cA(1)$, respectively.
\begin{proof}
Throughout this proof, implicitly $2$-localize. We give the proof for twisted spin bordism; the remaining cases are analogous. The input is a theorem of Hebestreit-Joachim~\cite{HJ20} that the Anderson-Brown-Peterson decomposition of $\MTSpin$ as a sum of $\ko$-modules upgrades to a splitting of local systems of spectra over $B\O/B\Spin$. Therefore, given a fake vector bundle twist $f_{a,b}\colon X\to B\O/B\Spin$, there is an equivalence of spectra
\begin{equation}
\label{twisted_ABP}
    M^{\MTSpin}f_{a,b} \simeq \textcolor{BrickRed}{\bigvee_i \Sigma^{\ell_i} M^\ko f_{a,b}} \vee \textcolor{Green}{\bigvee_j \Sigma^{m_j} M^\ko f_{a,b}\wedge_\ko \ko\ang 2}\vee \textcolor{MidnightBlue}{\bigvee_k \Sigma^{n_k} M^\ko f_{a,b}\wedge_\ko H\Z/2},
\end{equation}
where the indices $i$, $j$, $k$, $n_i$, $n_j$, and $n_k$ represent the indices and shifts in the original Anderson-Brown-Peterson decomposition~\cite{ABP67} and $\ko\ang 2$ is the $1$-connected cover of $\ko$.

The right-hand side of~\eqref{twisted_ABP} is manifestly a $\ko$-module; use this equivalence to define a $\ko$-module structure on $M^\MTSpin f_{a,b}$. Then the spectral sequence in the corollary statement is the Baker-Lazarev $\ko$-module Adams spectral sequence for $M^{\MTSpin}f_{a,b}$; $W_3$ appears because it is the direct sum of $H_\ko^*$ of the summands in the Anderson-Brown-Peterson decomposition.

Hebestreit-Joachim's proof goes through in exactly the same way for $\MTSpin^c$ and $\ku$~\cite{HJ20}. For $\MTSO$ and $H\Z$, we use the fact that the map $B\O/B\SO\to B\GL_1(\MTSO)$ factors through $B\GL_1(\mathbb S)$, so every equivalence of spectra fibers over it.
\end{proof}

\begin{proof}[Proof of \cref{thom_module_calc}]\label{proof:thom_module_calc}
All five parts of the theorem have similar proofs, so we walk through
the full proof in two cases --- $R = \ku$, whose proof carries through for $H\Z$, $\ko$, and $\tmf$ at $p = 3$ with
minor changes; and $R = \tmf$ at $p = 2$, where the presence of supercohomology means the proof is slightly
different.

Now we specialize to $R= \ku$ and a fake vector bundle twist $f_{a,c}\colon X\to B\GL_1(\ku)$. To begin, use
\cref{Rmod_Thom_iso} to learn that $H_\ku^*(M^\ku f_{a,c})\cong H^*(X;\Z/2)$ as $\Z/2$-vector spaces.
(In the more familiar case where the twist is given by a vector bundle, this is the Thom isomorphism.) Next, the
Thom diagonal (\cref{def:ThomDiag}) and the Cartan formula provide a formula for $Q_i(Ux)$, $i = 0,1$, in terms of
$Q_i(U)$ and $Q_i(x)$. In particular, this formula implies that if we can show $Q_0(U) = Ua$ and $Q_1(U) =
U(a^3+c)$, then the $\cE(1)$-module action defined on $M_\ku(X, a, c)$ in \cref{the_twisted_modules} is identified
with $H_\ku^*(Mf_{a,c})$.
%
By the naturality of cohomology operations, it suffices to compute $Q_0(U)$ and $Q_1(U)$ for the
the universal twist over $B\O/B\Spin^c$. \Cref{universal_twist} then allows us to infer what the cohomology
operations on the Thom class have to be in order to recover the correct $\cA$-module structure on the Thom spectrum
after applying the universal twist.

Let $f\colon B\O/B\Spin^c\to B\GL_1(\MTSpin^c)$ be the universal fake vector bundle twist,
$M^{\MTSpin^c}f$ be its associated Thom spectrum, and $M^{\ku}f$ be the $\ku$-module Thom spectrum obtained by
composing $f$ with the map $B\GL_1(\MTSpin^c)\to B\GL_1(\ku)$ induced by the Atiyah-Bott-Shapiro map. The
Atiyah-Bott-Shapiro map is $3$-connected, so the map $M^{\MTSpin^c}f\to M^{\ku}f$ is also $3$-connected. Thus, for
example, $\pi_0(M^{\ku}f)\cong\pi_0(M^{\MTSpin^c} f)$; by \cref{universal_spinc_twists} $M^{\MTSpin^c}f\simeq\MTO$,
so $\pi_0(M^{\ku} f)\cong\Omega_0^\O\cong\Z/2$. This and similar ideas will determine $Q_0(U)$ and $Q_1(U)$ for us:
%
in particular we will find $\Sq^1(U) = Ua$ and $Q_1(U) = U(c \bmod 2+a^3)$ because this is the unique
 choice that is compatible with the known homotopy groups of the Thom spectra of the universal twists from \S\ref{spinc_twists}: $\MTO$ over $K(\Z/2, 1)\times K(\Z, 3)$, $\MTSO$ over $K(\Z, 3)$, and $\MTPin^c$ over $K(\Z/2, 1)$.
 
We first consider $Q_0$: $Q_0(U)$ is either $0$ or $Ua$. 
For either of the two options for $Q_0(U)$, one can explicitly
write the $\cE(1)$-module structure on $H_\ku^*(M^\ku f)$ in low degrees. Then, using Baker-Lazarev's Adams
spectral sequence, one finds that if $Q_0(U) = 0$, $\pi_0(M^\ku f)_2^\wedge\cong \pi_0(M^{\MTSpin^c}f)$ has at
least $4$ elements, but since $Mf \simeq\MTO$, we know this group is $\Omega_0^\O\cong\Z/2$. Thus $Q_0(U) = Ua$.

There are three options for $Q_1(U)$: $0$, $U c\bmod 2$, and $U(c\bmod 2+a^3)$. In order to verify the $Q_1$
action, we pull back to $K(\Z, 3)$ and $K(\Z/2, 1)$ separately, and then argue in a similar way.
\begin{itemize}
	\item For $f\colon K(\Z, 3)\to B\GL_1(\MTSpin^c)$, $M^{\MTSpin^c} f\simeq\MTSO$, which is incompatible with
	$Q_1(U) = 0$; the argument is similar to that for $Q_0$.
	\item For $f\colon K(\Z/2, 1)\to B\GL_1(\MTSpin^c)$, $M^{\MTSpin^c} f\simeq \MTPin^c$.
	In $H_\ku^*(M^\ku f)$, $Q_1(U)\ne 0$, which one can show by pulling back further along
	\begin{equation}
		M^\ku f\wedge H\Z/2\longrightarrow M^\ku\wedge_{\ku} H\Z/2.
	\end{equation}
\end{itemize}
Thus $Q_1(U) = U(c\bmod 2+a^3)$.
%
 Using the fact that $\cE(1) = \langle Q_0, Q_1\rangle$ and applying the Cartan formula recovers the actions in \eqref{eq:kuActions}. 

Because the fake vector bundle twist for $\tmf$ uses supercohomology, its part of the proof is different enough
that we go into the details. The reduction to the computation of $\Sq^1(U)$, $\Sq^2(U)$, and $\Sq^4(U)$ in the case
of the universal twist proceeds in the same way as for $\ku$. In \S\ref{string_twists} we computed
$H^*(B\SO/B\String;\Z/2)$ in low degrees; this and the Künneth formula imply that in the mod $2$ cohomology of
$K(\Z/2, 1)\times B\SO/B\String$, $H^1$ is spanned by $a$, $H^2$ is spanned by $\set{a^2, t(d)}$, and $H^4$ is
spanned by $\set{a^4, a^2 t(d), a\Sq^1t(d), \delta, t(d)^2}$. Therefore there are
$\lambda_1,\dotsc,\lambda_8\in\Z/2$ such that
\begin{subequations}
\begin{align}
	\Sq^1(U) &= U\lambda_1 a\\
    \label{eq:Sq2tmf}
	\Sq^2(U) &= U(\lambda_2 a^2 + \lambda_3 t(d))\\
	\Sq^4(U) &= U(\lambda_4 a^4 + \lambda_5 a^2t(d) + \lambda_6 a\Sq^1t(d) + \lambda_7\delta + \lambda_8 t(d)^2).\label{eq:Sq4tmf}
\end{align}
\end{subequations}
We finish the proof by indicating how to find $\lambda_1$ through $\lambda_8$. 
To find $\lambda_7$, consider the twist pulled back to $f\colon K(\Z, 4)\simeq B\Spin/B\String\to B\O/B\String$.
Like in the proof for twists of $\ku$, the action of $\Sq^4$ on the Thom class can be detected on either
$M^{\MTString}f$ or $M^{\tmf}f$; as we discussed in \cref{universal_string_twists}, $M^{\MTString}f\simeq\MTSpin$,
so $\pi_3(M^{\MTString}f)\cong\Omega_3^\Spin = 0$, and since the map $M^{\MTString}f\to M^{\tmf}f$ is
sufficiently connected, $\pi_3(M^{\tmf}f) = 0$ as well. In $H_{\tmf}^*(M^{\tmf}f)$, the only options for $\Sq^4(U)$
are $0$ or $U$ times the tautological class. One can run the Baker-Lazarev Adams spectral sequence for these two
options and see that only the latter choice is compatible with $\pi_3(M^{\tmf}f) = 0$.\footnote{To do so, it will
be helpful to know $\Ext_{\cA(2)}(C\nu, \Z/2)$, where $C\nu$ is the $\cA(2)$-module with two $\Z/2$ summands in
degrees $0$ and $4$, joined by a $\Sq^4$. These Ext groups have been computed by Bruner-Rognes~\cite[Corollary
4.16, Figure 4.3]{BR21}.} Thus $\lambda_7 = 1$.

For the other coefficients, we pull back to vector bundle twists for various vector bundles $V\to X$, where we know
$\Sq^k(U) = Uw_k(V)$, $a\mapsto w_1(V)$, $t(d)\mapsto w_2(V)$, and $\delta\mapsto w_4(V)$. Choosing vector bundles with auspicious values of $w_1$, $w_2$, and $w_4$ quickly determines the remaining coefficients.

\begin{itemize}
	\item Pulling back the twist to $K(\Z/2, 1)\simeq B\O_1$ gives the Thom spectrum  $\tmf\wedge
	(B\O_1)^{\sigma-1}$, where $\sigma\to B\O_1$ is the tautological line bundle. As $w_1(\sigma)\ne 0$ but $w_2(\sigma) = 0$ and $w_4(\sigma) = 0$, we can plug these
	Stiefel-Whitney classes into ~\eqref{eq:Sq2tmf} (with $w_1(V)$ in place of $a$, $w_2(V)$ in place of $t(d)$,
	and $w_4(V)$ in place of $\delta$ as usual) to conclude
	$\lambda_1 = 1$, $\lambda_2 = 0$, and $\lambda_4 = 0$.
%
	\item Let $V \coloneqq \mathcal{O}(1)\oplus\mathcal{O}(2)\to\mathbb{CP}^2$. If $\alpha\in H^2(\mathbb{CP}^2;\Z/2)\cong\Z/2$ is the unique nonzero element, then $w_1(V) = 0$, $w_2(V) = \alpha$, and $w_4(V) = 0$. Plugging this into~\eqref{eq:Sq2tmf}, we find $\Sq^2(U) = U\alpha = U\lambda_3\alpha$, so $\lambda_3 = 1$. And plugging $w_1(V)$, $w_2(V)$, and $w_4(V)$ into~\eqref{eq:Sq4tmf}, we obtain $\Sq^4(U) = 0 = U\lambda_8\alpha^2$, so $\lambda_8 = 0$.
        \item Let $x$, resp.\ $y$ be the nonzero classes in $H^1(\RP^2\times\RP^2;\Z/2)$ pulled back from the first, resp.\ second copy of $\RP^2$, and let $\sigma_x,\sigma_y\to\RP^2\times\RP^2$ be the real line bundles satisfying $w_1(\sigma_x) = x$ and $w_1(\sigma_y) = y$. Now let $V \coloneqq \sigma_x\oplus \sigma_y^{\oplus 3}$; then $w_1(V) = x+y$, $w_2(V) = xy+y^2$, and $w_4(V) = 0$. Plugging into~\eqref{eq:Sq4tmf}, we have $\Sq^4(U) = 0 = U\lambda_5x^2y^2$, so $\lambda_5 = 0$.
        \item Repeat the preceding example, but with $\RP^1\times\RP^3$ in place of $\RP^2\times\RP^2$; this time, $w_1(V) = x+y$, $w_2(V) = xy+y^2$, and $w_4(V) = xy^3$. Plugging into~\eqref{eq:Sq4tmf}, we have $\Sq^4(U) = Uxy^3 = U(1+\lambda_6)xy^3$, so $\lambda_6 = 0$.
        \qedhere
\end{itemize}
\end{proof}

\section{Applications}\label{section:applications}

In this section, we give examples in which we use \cref{little_sibling_Adams_corollary,big_sibling_Adams_corollary}
to make computations of twisted (co)homology groups.
\subsection{U-duality and related twists of spin bordism}\label{sub:uduality}
\label{s:u_duality}
Let $G$ be a topological group and
\begin{subequations}
\begin{equation}
	\shortexact*{\set{\pm 1}}{\wG}{G}{}
\end{equation}
be a central extension classified by $\beta\in H^2(BG;\set{\pm 1})$. Then the central extension
\begin{equation}
\label{spin_G_xtn}
	\shortexact*[][p]{\set{\pm 1}}{\Spin\times_{\set{\pm 1}}\wG}{\SO\times G}{}
\end{equation}
\end{subequations}
is classified by $w_2 + \beta\in H^2(B(\SO\times G);\Z/2)$. One can
prove this is the extension by pulling back along $\SO \rightarrow \SO \times G$ and $G \rightarrow \SO \times G$ and observing that both
pulled-back extensions are non-split. 
 Therefore given an oriented vector bundle
$E\to X$ and a principal $G$-bundle $P\to X$, i.e.\ the data of an $\SO\times G$ structure on $E$, a lift of this
data to a $\Spin\times_{\set{\pm 1}}\wG$-structure is a trivialization of $w_2(E) + f_P^*(\beta)$, where
$f_P\colon X\to BG$ is the classifying map of $P\to X$. That is, if $\xi$ denotes the composition
\begin{equation}
	\xi\colon B(\Spin\times_{\set{\pm 1}}\wG)\overset{Bp}{\longrightarrow} B\SO\times BG\to B\SO\to B\O,
\end{equation}
then a $\xi$-structure on $E$ is equivalent to a $(BG, \beta)$-twisted spin structure, meaning that by \cref{twisted_spin_bordism_comparison} the Thom spectrum $\mathit{MT\xi}$ is canonically equivalent to the $\MTSpin$-module Thom spectrum
$Mf_{0,\beta}$ associated to the fake vector bundle twist $f_{0,\beta} \colon BG\to B\GL_1(\MTSpin)$ (see \cref{HJ_lift} for the spectrum-level statement). $\mathit{MT\xi}$ may or may not split as $\MTSpin\wedge X$ for a spectrum $X$: a sufficient condition is the existence of a
vector bundle $V\to BG$ such that $w_2(V) = \beta$, as we discussed in \S\ref{VB_change_of_rings}. But as we will see soon,
there are choices of $(G, \beta)$, even when $G$ is a compact, connected Lie group, for which no such $V$ exists.
For these $G$ and $\beta$, \cref{thom_module_calc} significantly simplifies the calculation of $\xi$-bordism.

As an example, consider $G = \SU_8/\set{\pm 1}$ and $\beta$ the nonzero element of $H^2(BG;\Z/2)
\cong\Hom(\pi_1(G), \Z/2) \cong\Z/2$, corresponding to the central extension
\begin{equation}
	\shortexact*{\set{\pm 1}}{\SU_8}{\SU_8/\set{\pm 1}}.
\end{equation}
In~\cite{DY22}, we studied $\Omega_*^{\Spin\times_{\set{\pm 1}}\SU_8}$ as part of an argument that the
$E_{7(7)}(\R)$ U-duality symmetry of four-dimensional $\cN = 8$ supergravity is anomaly-free. Speyer~\cite{Speyer}
shows that all representations of $G$ are spin, so $\beta\ne w_2(V)$ for any vector bundle $V\to BG$ induced from a
representation of $G$, and this can be upgraded to show $Mf_{0,\beta}\not\simeq\MTSpin\wedge X$ for
any spectrum $X$ (see~\cite[Footnote 6]{DY22}). This precludes the standard shearing/change-of-rings argument for
computing $\Spin\times_{\set{\pm 1}}\SU_8$ bordism, and indeed in~\cite[\S 4.3]{DY22} we had to give a more
complicated workaround. However, thanks to \cref{thom_module_calc}, we can now argue over $\cA(1)$. We need as input
the low-degree cohomology of $B(\SU_8/\set{\pm 1})$.
\begin{prop}[{\cite[Theorem 4.4]{DY22}}]
\label{SU8_Z2_coh}
$H^*(B(\SU_8/\set{\pm 1});\Z/2)\cong\Z/2[\beta, b, c, d, e, \dots]/(\dots)$ with $\abs \beta = 2$, $\abs b = 3$,
$\abs c = 4$, $\abs d = 5$, and $\abs e = 6$; there are no other generators below degree $7$ and no relations below
degree $7$. The Steenrod squares are
\begin{equation}\label{eq:SteenrodSU8/Z2}
\begin{aligned}
	\Sq(\beta) &= \beta + b + \beta^2\\
	\Sq(b) &= b+d + b^2\\
	\Sq(c) &= c+e + \Sq^3(c) + c^2\\
	\Sq(d) &= d + b^2 + \Sq^3(d) + \Sq^4(d) + d^2.
\end{aligned}
\end{equation}
\end{prop}
By \cref{big_sibling_Adams_estimate}, to understand the $\MTSpin$-module Baker-Lazarev Adams spectral sequence for $M^{\MTSpin}f_{0,\beta}$ in the degrees we care about (i.e.\ $5$ and below), it is equivalent to consider the $\ko$-module analogue for $M^\ko f_{0,\beta}$,
\Cref{thom_module_calc} tells us how $\cA(1)$ acts on $H_\ko^*(M^\ko f_{0,\beta})$: $\Sq^1(U) = 0$ and $\Sq^2(U) =
U\beta$; to make more computations, use the Cartan formula and the Steenrod squares in \cref{SU8_Z2_coh}.
Then using the information from~\eqref{eq:SteenrodSU8/Z2} yields
\begin{subequations}
\label{seagull_maker}
\begin{gather}
\begin{aligned}
    \Sq^1(U\beta) &= U\Sq^1(\beta) + \Sq^1(U)\beta = Ub\\
	\Sq^2(U\beta) &= U\Sq^2(\beta) + \Sq^1(U)\Sq^1(\beta) + \Sq^2(U)\beta = U(2\beta^2) = 0\\
\end{aligned}\\
\begin{aligned}
	\Sq^1(Ub) &= U\Sq^1(b) + \Sq^1(U)b = 0\\
	\Sq^2(Ub) &= U\Sq^2(b) + \Sq^1(U)\Sq^1(b) + \Sq^2(U)b = U(d+b\beta)
\end{aligned}\\
\begin{aligned}
	\Sq^1(U(d+b\beta)) &= U\Sq^1(d+b\beta) + \Sq^1(U)(d+b\beta) = U(2b^2) = 0\\
	\Sq^2(U(d+b\beta)) &= U\Sq^2(d+b\beta) + \Sq^1(U)\Sq^1(d+b\beta) + \Sq^2(U)(d+b\beta)= 0.
\end{aligned}
\end{gather}
\end{subequations}
See the lower left (red) piece of \cref{Uduality_sseq}, left, for a picture of this data.
This calculation implies the vector space generated by $\set{U, U\beta, Ub, U(d+b\beta)}$ is an $\cA(1)$-submodule
of $H_\ko^*(M^\ko f_{0,\beta})$; specifically, it is isomorphic to the ``seagull'' $\cA(1)$-module $M_0\coloneqq
\cA(1)\otimes_{\cA(0)}\Z/2$.\footnote{Adamyk~\cite{adamyk2021classifying} introduced the name ``seagull'' for $M_0$.} This is an $\cA(1)$-module whose $\cA(1)$-action does not compatibly extend to an
$\cA$-action. Continuing to compute $\Sq^1$- and $\Sq^2$-actions as in~\eqref{seagull_maker}, we learn that there
is an isomorphism of $\cA(1)$-modules 
\begin{equation}
\label{spin-SU8_Hko}
	H_{\ko}^*(M^\ko f_{0,\beta})\cong \textcolor{BrickRed}{M_0} \oplus
		\textcolor{Green}{\Sigma^4 M_0} \oplus
		\textcolor{MidnightBlue}{\Sigma^4 M_1} \oplus
		\textcolor{Fuchsia}{\cA(1)} \oplus P,
\end{equation}
where $P$ is concentrated in degrees $6$ and above (so we can and will ignore it), and
$\textcolor{MidnightBlue}{M_1}$ is an $\cA(1)$-module which is isomorphic to either $M_0$ or $C\eta\coloneqq
\cA(1)\otimes_{\cE(1)}\Z/2$. We draw the decomposition~\eqref{spin-SU8_Hko} in \cref{Uduality_sseq}, left.
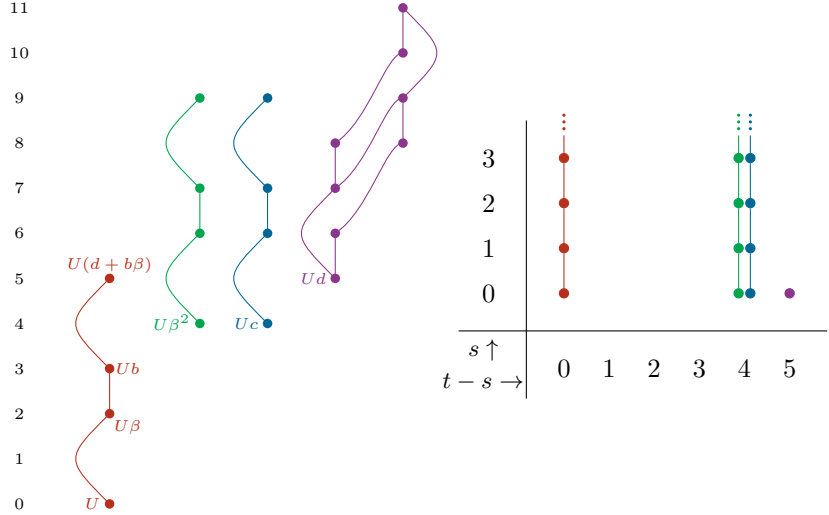
\begin{figure}[h!]
\begin{subfigure}[c]{0.35\textwidth}
\begin{tikzpicture}[scale=0.6, every node/.style = {font = \tiny}]
        \foreach \y in {0, ..., 11} {
                \node at (-2, \y) {$\y$};
        }
        \begin{scope}[BrickRed]
                \Mzero{0}{0}{$U$};
                \draw (0.4,1.7) node {$U\beta$};
                \draw (0.4,3) node {$Ub$};
                \draw (0,5.3) node {$U(d+b\beta)$};
        \end{scope}
        \begin{scope}[Green]
                \Mzero{2}{4}{$U\beta^2$};
        \end{scope}
        \begin{scope}[MidnightBlue]
                \Mzero{3.5}{4}{$Uc$};
        \end{scope}
        \begin{scope}[Fuchsia]
                \Aone{5}{5}{$Ud$};
        \end{scope}
\end{tikzpicture}
\end{subfigure}
\quad
\begin{subfigure}[c]{0.35\textwidth}
\includegraphics{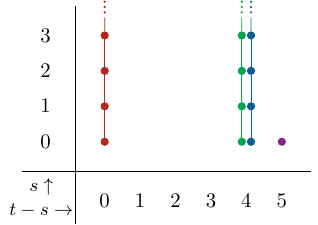}
\end{subfigure}
\caption{Left: the $\cA(1)$-module structure on $H_{\ko}^*(M^\ko f_{0,\beta})$ in low degrees, where
$\beta\in H^2(B(\SU_8/\set{\pm 1});\Z/2)$ is the generator. The pictured submodule contains all elements in degrees
$5$ and below. We have not determined $\Sq^3(Uc)$ -- it may be $0$, in which case the blue summand would vanish in
degrees $7$ and above. In either case, the pictured $\cA(1)$-module cannot arise as the restriction of an
$\cA$-action to $\cA(1)$, indicating that the fake vector bundle twist $f_{0,\beta}$ of $\ko$ cannot arise from a
vector bundle.  Right: the $E_2$-page of the corresponding $\ko$-module Adams spectral sequence, which as discussed
in \S\ref{s:u_duality} also computes the $2$-completion of $\Omega_*^{\Spin\times_{\set{\pm 1}}\SU_8}$ in degrees $7$ and below.}
\label{Uduality_sseq}
\end{figure}

The change-of-rings isomorphism (\cref{CoR_lemma}) and Koszul duality~\cite[Remark 4.5.4]{BC18} allow us to compute
$\Ext_{\cA(1)}(M_0)\cong\Z/2[h_0]$ and $\Ext_{\cA(1)}(C\eta)\cong\Z/2[h_0, v_1]$ with $h_0$ in bidegree $(t-s, s)
= (0, 1)$ and $v_1$ in bidegree $(t-s, s) = (2, 1)$~\cite[Examples 4.5.5 and 4.5.6]{BC18}. Therefore we can draw
the $E_2$-page of the Adams spectral sequence computing the twisted $\ko$-homology associated to the fake vector
bundle twist $f_{0,\beta}\colon B(\SU_8/\set{\pm 1})\to B\GL_1(\ko)$ in \cref{Uduality_sseq}, right. By \cref{big_sibling_Adams_estimate}, this also computes $\pi_*(M^{\MTSpin}f_{0,\beta})_2^\wedge$ in low degrees, and by
\cref{twisted_spin_bordism_comparison}, this is isomorphic to the $2$-completion of the corresponding twisted spin bordism groups, which we saw above are $\Omega_*^{\Spin\times_{\set{\pm 1}}\SU_8}$. This spectral sequence collapses on the $E_2$-page in
degrees $5$ and below, using $h_0$-linearity of differentials, so we have made the following computation.
\begin{thm}[{\cite[Theorem 4.26]{DY22}}]
\label{u_duality_bordism}
\begin{equation}
\begin{aligned}
	\Omega_0^{\Spin\times_{\set{\pm 1}}\SU_8} &\cong \Z\\
	\Omega_1^{\Spin\times_{\set{\pm 1}}\SU_8} &\cong 0\\
	\Omega_2^{\Spin\times_{\set{\pm 1}}\SU_8} &\cong 0\\
	\Omega_3^{\Spin\times_{\set{\pm 1}}\SU_8} &\cong 0\\
	\Omega_4^{\Spin\times_{\set{\pm 1}}\SU_8} &\cong \Z^2\\
	\Omega_5^{\Spin\times_{\set{\pm 1}}\SU_8} &\cong \Z/2.
\end{aligned}
\end{equation}
\end{thm}
There are a few other choices of compact Lie groups $G$ and classes $\beta \in H^2(BG;\Z/2)$ such that $\beta$ is
not equal to $w_2$ of any representation, including
\begin{itemize}
	\item $\SU_{4n}/\set{\pm 1}$ for $n > 1$, where $\beta$ corresponds to the double cover
	$\SU_{4n}\to\SU_{4n}/\set{\pm 1}$~\cite{Speyer},
	\item $\PSO_{8n}$, where $\beta$ corresponds to the double cover $\SO_{8n}\to\PSO_{8n}$~\cite{TJF},
	\item $\mathrm{PSp}_n$ and the double cover $\mathrm{Sp}_n\to \mathrm{PSp}_n$ for $n > 1$, and
	\item $E_7/\set{\pm 1}$ and the double cover $E_7\to E_7/\set{\pm 1}$.
\end{itemize}
For the last two items, the proof is analogous to~\cite[Footnote 6]{DY22} for $\SU_8/\set{\pm 1}$: compute the low-degree mod $2$ cohomology of $BG$ and use this to show that if $\beta$ is $w_2$ of a representation $V$, the $\cA$-action on the cohomology of the corresponding Thom spectrum violates the Adem relations.

For all of these choices of $G$ and $\beta$, one can define (at a physics level of rigor) unitary quantum field theories with fermions
and a background $\wG$ symmetry, such that $-1\in\wG$ acts by $-1$ on fermions and by $1$ on bosons. Then, as
described in \cite{WWW19,Seiberg:2016rsg}, these theories can be defined on manifolds with differential
$\Spin_n\times_{\set{\pm 1}}\wG$ structures, so by work of Freed-Hopkins~\cite{FH21a}, the anomaly field theories of
these QFTs are classified using the bordism groups $\Omega_*^{\Spin\times_{\set{\pm 1}}\wG}$, and computations such
as \cref{u_duality_bordism} are greatly simplified using \cref{thom_module_calc}.

Kuroda~\cite{Kur25} makes some of these computations, using similar methods to the ones we used here to determine the $\Spin\times_{\set{\pm 1}}\Sp_4$, $\Spin\times_{\set{\pm 1}}\SU_8$, and $\Spin\times_{\set{\pm 1}}\Spin_{16}$ bordism groups in degrees $7$ and below.
\begin{rem}
\label{stabdiff}
Though we focused on invertible field theories in this section, there are other applications of twisted spin
bordism groups. For example, Kreck's modified surgery~\cite{Kre99} uses twisted spin bordism to classify closed,
smooth $4$-manifolds whose universal covers are spin up to stable diffeomorphism: given such a manifold $M$, one
shows that $w_1(M)$ and $w_2(M)$ pull back from $B\pi_1(M)$, then considers twisted spin bordism for the fake
vector bundle twist over $B\pi_1(M)$ given by $w_1(M)$ and $w_2(M)$.
Often one computes these bordism groups with Teichner's \term{James spectral sequence}~\cite[\S II]{Tei93}, a
version of the Atiyah-Hirzebruch spectral sequence for spin bordism that can handle non-vector-bundle twists.
However, extension questions in this spectral sequence can be difficult, and it is helpful to have the Adams
spectral sequence to resolve them (see~\cite{Ped17} for an example for a vector bundle twist). Therefore
\cref{big_sibling_Adams_corollary} could be a useful tool for studying stable diffeomorphism classes of
$4$-manifolds, since not all of the relevant twists come from vector bundles.
\end{rem}
\subsection{Twists of string bordism}\label{sub:stringbord}
A story very similar to that of~\S\ref{s:u_duality} takes place one level up in the Whitehead tower for $B\O$. Many
supergravity theories require spacetime manifolds $M$ to satisfy a \term{Green-Schwarz condition} specified by a
Lie group $G$ and a class $c\in H^4(BG;\Z)$, which Sati-Schreiber-Stasheff~\cite{SSS12} characterize as data of a
spin structure on $M$, a principal $G$-bundle $P\to M$ and a trivialization of $\lambda(M) - c(M)$, i.e.\ the data
of a $(BG, c)$-twisted string structure on $M$ (see also~\cite{Sat10, Sat11, SS19}). In many example theories of
interest, this twist does not come from a vector bundle, including
the $E_8\times E_8$ heterotic string and the CHL string~\cite[Lemma 2.2]{Deb23}.
The corresponding twisted string bordism groups are used to study anomalies and defects for these theories;
anomalies were touched on in \S\ref{s:u_duality}, and the use of bordism groups to learn about defects is through
the McNamara-Vafa cobordism conjecture~\cite{MV19}.

\Cref{big_sibling_Adams_estimate} allows us to use the Baker-Lazarev Adams spectral sequence at $p = 2$ and $p = 3$ to calculate these
twisted string bordism groups in dimensions $15$ and below, which suffices for applications to superstring theory.
(Calculations at primes greater than $3$ are easier and can be taken care of with other methods.) We will show an
example computation, relevant for the $E_8\times E_8$ heterotic string at $p = 3$; for applications of \cref{big_sibling_Adams_estimate} to twisted string bordism at $p = 2$, see~\cite[\S 2.2, \S 2.4.1]{Deb23} and~\cite{BDDM}, and for more $p = 3$ calculations, see~\cite{BDDM}.

Because $E_8$ is a connected, simply connected, simple Lie group, there is an isomorphism $c\colon
H^4(BE_8;\Z)\overset\cong\to\Z$ uniquely specified by making the Chern-Weil class of the Killing form positive; let
$c$ be the preimage of $1$ under this isomorphism. Bott-Samelson~\cite[Theorems IV, V(e)]{BS58} showed that, interpreted as a map
$BE_8\to K(\Z, 4)$, $c$ is $15$-connected.

For $i = 1,2$, let $c_i\in H^4(BE_8\times BE_8;\Z)$ be the copy of $c$ coming from the $i^{\mathrm{th}}$ copy of
$E_8$. Let $\Z/2$ act on $E_8\times E_8$ by switching the two factors; then in the Serre spectral sequence for the
fibration of classifying spaces induced by the short exact sequence
\begin{equation}
\label{E82_Serre}
	\shortexact*{E_8\times E_8}{(E_8\times E_8)\rtimes\Z/2}{\Z/2},
\end{equation}
the class $c_1 + c_2\in E_2^{0,4} = H^4(BE_8\times BE_8;\Z)$ survives to the $E_\infty$-page and lifts uniquely to
define a class $c_1 + c_2\in H^4(B((E_8\times E_8)\rtimes\Z/2);\Z)$. The Green-Schwarz condition for the $E_8 \times
E_8$ heterotic string asks for an $(E_8\times E_8)\rtimes\Z/2$-bundle $P\to M$ and a trivialization of $\lambda(M) -
(c_1+c_2)(P)$, so we want to compute $\Omega_*^\String(B((E_8\times E_8)\rtimes\Z/2), c_1+c_2)$.
\Cref{big_sibling_Adams_estimate} allows us to use the change-of-rings theorem to simplify the Adams spectral sequence at $p = 2,3$ for this
computation in degrees $15$ and below; we will give the $3$-primary computation here and point the interested reader to~\cite[\S 2.2]{Deb23}
for the longer $2$-primary computation.
\begin{thm}[{\cite[Theorem 2.65]{Deb23}}]
\label{heterotic_at_3}
The $(B((E_8\times E_8)\rtimes\Z/2), c_1+c_2)$-twisted string bordism groups lack $3$-primary torsion in degrees
$11$ and below.
\end{thm}
Just like for $\Spin\times_{\set{\pm 1}}\SU_8$ bordism and~\cite{DY22} in \S\ref{s:u_duality}, the computation
in~\cite{Deb23} does not take advantage of the change-of-rings theorem, works over the entire Steenrod algebra, and
is significantly harder than our proof here.
\begin{proof}
Recall the notation $\cA^{\tmf}$, $\beta$, and $\cP^1$ from \cref{intro_Atmf}.
By \cref{twisted_string_bordism_comparison} (see also \cref{HJ_lift}), the Thom spectrum for $(B((E_8\times E_8)\rtimes\Z/2), c_1+c_2)$-twisted string bordism
is identified with the $\MTString$-module Thom spectrum $M^{\MTString}f_{0,c_1+c_2}$, where $f_{0,c_1+c_2}$ is the
fake vector bundle twist defined by the image of the class $c_1+c_2\in H^4(B((E_8\times E_8)\rtimes\Z/2);\Z)$ in
supercohomology.
Let
$M^{\tmf}f_{0,c_1+c_2}$ be the $\tmf$-module Thom spectrum induced by the
Ando-Hopkins-Rezk map $\sigma\colon \MTString\to\tmf$. As a consequence of \cref{big_sibling_Adams_estimate}, in topological degrees $15$ and below, the $\MTString$-module Baker-Lazarev Adams spectral sequence for $M^{\MTString}f_{0,c_1+c_2}$ coincides with the $\tmf$-module Baker-Lazarev Adams spectral sequence for $M^\tmf f_{0,c_1+c_2}$.
\Cref{thom_module_calc} describes the $\cA^{\tmf}$-module structure on $H_\tmf^*(M^\tmf f_{0,c_1+c_2};\Z/3)$, and
hence the input to the $\tmf$-module Baker-Lazarev Adams spectral sequence, in terms of the $\cA_3$-module structure on $H^*(B(E_8\times
E_8)\rtimes\Z/2;\Z/3)$.
\begin{lem}
Let $x\coloneqq (c_1 + c_2)\bmod 3$ and $y\coloneqq c_1c_2\bmod 3$. Then $H^*(B(E_8\times
E_8)\rtimes\Z/2;\Z/3)\cong \Z/3[x, \cP^1(x), \beta\cP^1(x), y, \dots]/(\dots)$; there are no other generators below
degree $12$, nor any relations below degree $12$.
\end{lem}
The actions of $\cP^1$ and $\beta$ are as specified via the names of the generators.
\begin{proof}
Because $H^*(B\Z/2;\Z/3)$ vanishes in positive degrees, the Serre spectral sequence for~\eqref{E82_Serre} collapses
at $E_2$ to yield an isomorphism to the ring of invariants
\begin{equation}
	H^*(B(E_8\times E_8)\rtimes\Z/2;\Z/3) \overset\cong\longrightarrow (H^*(BE_8\times BE_8;\Z/3))^{\Z/2}.
\end{equation}
The lemma thus follows once we know $H^*(BE_8;\Z/3)\cong\Z/3[c\bmod 3, \cP^1(c\bmod 3), \beta\cP^1(c\bmod 3),
\dotsc]/(\dotsc)$, where we have given all generators and relations in degrees $11$ and below. Because $c\colon BE_8\to K(\Z, 4)$ is $15$-connected~\cite[Theorems IV, V(e)]{BS58}, we may replace $BE_8$ with $K(\Z,
4)$, and the mod $3$ cohomology of $K(\Z, 4)$ was computed by Cartan~\cite{Car54} and Serre~\cite{Ser52}; see
Hill~\cite[Corollary 2.9]{Hil09} for an explicit description.
\end{proof}
To compute $H_\tmf^*(M^\tmf f_{0,c_1+c_2})$, we also need to know $\cP^1(U)$, and \cref{thom_module_calc} tells us
 $\cP^1(U)=  Ux$.
 Then as usual we compute on all
classes in degrees $11$ and below using the Cartan formula.
\begin{cor}
Let $N_1\coloneqq \cA^{\tmf}/(\beta, (\cP^1)^2, \beta\cP^1\beta)$ and $N_2\coloneqq \cA^\tmf/(\beta, \beta\cP^1,
\cP^1\beta(\cP^1)^2)$. Then there is a map of $\cA^{\tmf}$-modules
\begin{equation}
\label{tmf_decomp}
	H_\tmf^*(M^\tmf f_{0,c_1+c_2})\longrightarrow \textcolor{BrickRed}{N_2} \oplus
		\textcolor{Green}{\Sigma^8 N_1} \oplus \textcolor{MidnightBlue}{\Sigma^8 N_1}
\end{equation}
which is an isomorphism in degrees $11$ and below.
\end{cor}
We draw the decomposition~\eqref{tmf_decomp} in \cref{twisted_tmf_Adams}, left. The next step is to compute the Ext
groups of $N_1$ and $N_2$ over $\cA^{\tmf}$. To do so, we will repeatedly use the fact that a short exact sequence
of $\cA^{\tmf}$-modules induces a long exact sequence in Ext; see~\cite[\S 4.6]{BC18} for more information on this
technique, including how to depict the long exact sequence in an Adams chart along with some examples. Let $C\nu$
denote the $\cA^{\tmf}$-module consisting of two $\Z/3$ summands in degrees $0$ and $4$ linked by a nontrivial
$\cP^1$-action. Then there are short exact sequences
\begin{subequations}
\begin{gather}
	\label{Cnu_SES}
	\shortexact{\textcolor{RubineRed}{\Sigma^4\Z/3}}{C\nu}{\textcolor{Periwinkle}{\Z/3}},\\
	\label{qn_SES}
	\shortexact{\textcolor{Cerulean}{\Sigma^5\Z/3}}{N_1}{\textcolor{PineGreen}{C\nu}},\\
	\label{2qn_SES}
	\shortexact{\textcolor{RawSienna}{\Sigma^4 N_1}}{N_2}{\textcolor{RedOrange}{\Z/3}}.
\end{gather}
\end{subequations}
We will address~\eqref{Cnu_SES} in \cref{cnu_fig}, \eqref{qn_SES} in \cref{qn_fig}, and~\eqref{2qn_SES} in
\cref{2qn_fig}. 
As input to our computations, we need $\Ext_{\cA^{\tmf}}(\Z/3)\coloneqq \Ext_{\cA^{\tmf}}^{*,*}(\Z/3, \Z/3)$; this acts on $\Ext_{\cA^\tmf}(V)\coloneqq \Ext_{\cA^{\tmf}}^{*,*}(V, \Z/3)$ for any
$\cA^{\tmf}$-module $V$ by the Yoneda product (see~\cite[\S 4.2]{BC18}). The boundary maps in the long exact
sequences of Ext groups induced by short exact sequences of $\cA^{\tmf}$-modules are linear for this
$\Ext_{\cA^{\tmf}}(\Z/3)$-action, which we will use in \cref{nu_bdry,qn_bdry,2qn_bdry}. Throughout this subsection, if we do not specify the base, $\Ext$ means $\Ext_{\cA^\tmf}$.
\begin{thm}[{Henriques-Hill~\cite{Hil07, DFHH14}}]
\label{Ext_Z2_tmf}
$\Ext_{\cA^{\tmf}}(\Z/3)$ is generated by the classes
	$h_0\in\Ext^{1,1}$,
	$\alpha\in\Ext^{1,4}$,
	$c_4\in\Ext^{2,10}$,
	$\beta\in\Ext^{2,12}$,
	$c_6\in\Ext^{3,15}$, and
	$\Delta\in \Ext^{3,27}$, modulo the relations $\alpha^2 = 0$, $h_0\alpha = 0$, $h_0\beta = 0$, $\alpha c_4 =
	0$, $\beta c_4 = 0$, $\alpha c_6 = 0$, $\beta c_6 = 0$, and $c_4^3 - c_6^2 = h_0^3\Delta$.
\end{thm}
\begin{rem}
Our notation differs from that of some authors who study $\Ext_{\cA^{\tmf}}(\Z/3)$. Compared with Hill~\cite[\S
2]{Hil07}, our names for generators agree except that what we call $h_0$ Hill calls $v_0$. Comparing with
Bruner-Rognes~\cite[Chapter 13]{BR21}: our $h_0$ is their $a_0$, our $\alpha$ is
their $h_0$, and our $\beta$ is their $b_0$, and other names of generators agree.
\end{rem}
The action of $h_0$ on the $E_\infty$-page of this Adams spectral sequence lifts to multiplying by $3$ on the
twisted $\tmf$-homology groups that the spectral sequence converges to.


In the long exact sequence in Ext corresponding to~\eqref{Cnu_SES}, let $x\in\Ext^{0,0}$ be either generator of
$\Ext_{\cA^{\tmf}}(\textcolor{Periwinkle}{\Z/3})$ and $y\in\Ext^{0,4}$ be either generator of
$\Ext_{\cA^{\tmf}}(\textcolor{RubineRed}{\Sigma^4\Z/3})$, both as modules over $\Ext_{\cA^{\tmf}}(\Z/3)$. In both
cases, there are exactly two generators and they differ by a sign.
\begin{lem}
\label{nu_bdry}
In the long exact sequence in $\Ext$ associated to~\eqref{Cnu_SES}, $\partial(y) = \pm \alpha x$, $\partial(\beta y) =
\pm \alpha\beta x$, and the boundary map vanishes on all other elements in degrees $14$ and below (except for
$-c$ where $c$ was a class already listed).
\end{lem}
We draw this in \cref{cnu_fig}, bottom left.
\begin{proof}
Apart from on $\pm y$ and $\pm \beta y$, the boundary map vanishes for degree reasons; since $\partial$ commutes
with the action of $\Ext_{\cA^{\tmf}}(\Z/3)$, once we show $\partial(y) = \pm \alpha x$, $\partial(\beta y) = \pm
\alpha\beta x$ follows. Since $\Ext^{1,4}(\textcolor{Periwinkle}{\Z/3})\cong\Z/3$, if we show $\partial(y)\ne 0$
the only options for $\partial y$ are $\pm \alpha x$.

Since $y$ and $-y$ are the only nonzero elements in $\Ext^{4,0}$ of both $\textcolor{Periwinkle}{\Z/3}$ and
$\textcolor{RubineRed}{\Sigma^4\Z/3}$, $\partial(y) = 0$ if and only if $\Ext_{\cA^{\tmf}}^{0,4}(C\nu) = 0$. And
this Ext group is $\Hom_{\cA^{\tmf}}(C\nu, \Sigma^4 \Z/3) = 0$.
\end{proof}
\begin{rem}
\label{hidden_alpha}
In $\Ext_{\cA^{\tmf}}(C\nu)$, $\alpha(\alpha y) = \beta x$,\footnote{This does not contradict the relation
$\alpha^2 = 0$ from \cref{Ext_Z2_tmf}: since $y$ was killed in the long exact sequence computing $\Ext(C\nu)$,
the class $\alpha y\in \Ext(C\nu)$ is not $\alpha$ times anything, so $\alpha(\alpha y)$ need not vanish.} but this
is not detected by the long exact sequence in Ext. This action is denoted with a dashed gray line in
\cref{cnu_fig}, bottom right. We do not need this hidden $\alpha$-action, so we will not prove it;\footnote{We do
use this $\alpha$-action in the proof of \cref{2qn_bdry}, but only to determine Ext groups that will be in too high
of a degree to matter in the final computation, so that part of the proof can be left out.} one way to check it is
to compute $\Ext_{\cA_3}(C\nu)$ using the software developed by Bruner~\cite{Bru18} or by Beauvais-Feisthauer, Chatham, and Chua~\cite{CC21},
obtain the hidden $\alpha$-action in $\Ext_{\cA_3}(C\nu)$, and chase it across the map of Ext groups induced by
$\cA^{\tmf}\to\cA_3$.
\end{rem}
Thus we obtain $\Ext_{\cA^\tmf}(C\nu)$ in \cref{cnu_fig}, bottom right.


\begin{figure}[h!]
\begin{subfigure}[c]{0.4\textwidth}
\begin{tikzpicture}[scale=0.5]
        \PoneR(0, 0);
        \begin{scope}[RubineRed]
                \foreach \x in {-5, 0} {
                        \tikzpt{\x}{2}{}{};
                }
                \draw[->, thick] (-4.5, 2) -- (-0.5, 2);
        \end{scope}
        \begin{scope}[Periwinkle]
                \foreach \x in {0, 5} {
                        \tikzpt{\x}{0}{}{};
                }
                \draw[->, thick] (0.5, 0) -- (4.5, 0);
        \end{scope}
        \node[below=2pt] at (-5, 0) {$\Sigma^4\Z/3$};
        \node[below=2pt] at (0, 0) {$C\nu$};
        \node[below=2pt] at (5, 0) {$\Z/3$};
\end{tikzpicture}
\end{subfigure}

\begin{subfigure}[c]{0.48\textwidth}
\includegraphics{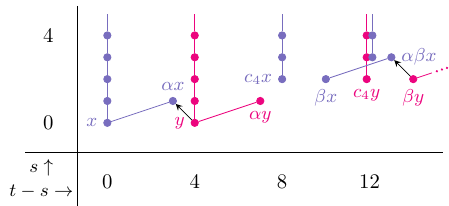}
\end{subfigure}
\begin{subfigure}[c]{0.48\textwidth}
\includegraphics{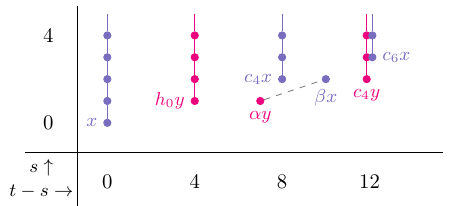}
\end{subfigure}
\caption{Top: the short exact sequence~\eqref{Cnu_SES} of $\cA^{\tmf}$-modules. Lower left: the induced long exact
sequence in Ext; we compute the pictured boundary maps in \cref{nu_bdry}. Lower right: $\Ext_{\cA^{\tmf}}(C\nu)$ as
computed by the long exact sequence. The dashed line is a nonzero $\alpha$-action not visible to this computation;
see \cref{hidden_alpha}.}
\label{cnu_fig}
\end{figure}
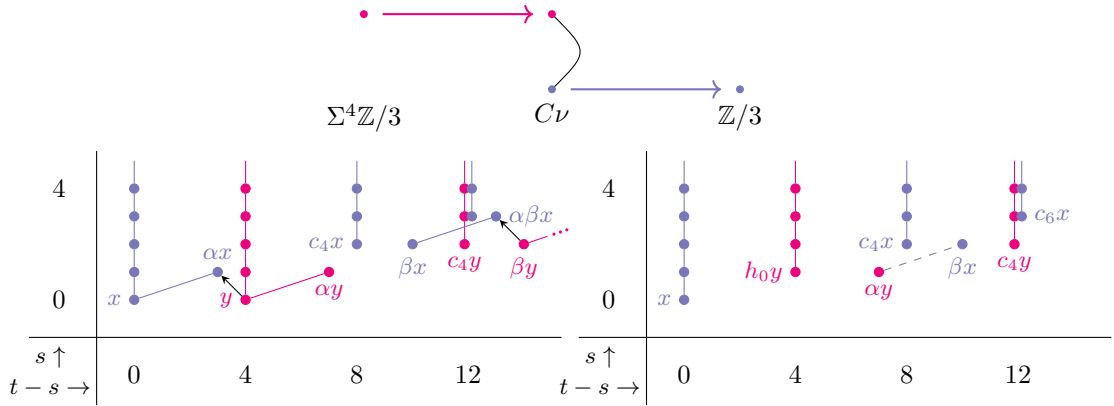

Now we turn to~\eqref{qn_SES} and its long exact sequence in Ext, depicted in \cref{qn_fig}. We keep the notation
for elements of $\Ext(\textcolor{PineGreen}{C\nu})$ from above, so elements are specified by products of classes in
$\Ext(\Z/3)$ with $x$ or $y$. In the long exact sequence induced by~\eqref{qn_SES}, let $z\in\Ext^{0,5}$ be a
generator of $\Ext(\textcolor{Cerulean}{\Sigma^5\Z/3})$ as a module over $\Ext(\Z/3)$ (again, there is exactly one
other generator, which is $-z$).
\begin{lem}
\label{qn_bdry}
In the long exact sequence in Ext associated to~\eqref{qn_SES}, $\partial(h_0^iz) = \pm h_0^iy$, $\partial(h_0c_4z) =
\pm h_0^i c_4y$, and the boundary map vanishes on all other elements in degrees $14$ and below (except for $-c$
where $c$ was a class already listed).
\end{lem}
We draw this in \cref{qn_fig}, bottom left.
\begin{proof}
The proof is essentially the same as for \cref{nu_bdry}: all boundary maps other than the ones in the theorem
statement vanish for degree reasons; then, $\Ext(\Z/3)$-linearity of boundary maps reduces the theorem statement to
the computation of $\partial(z)$, which must be $\pm h_0y$ because $\Ext_{\cA^{\tmf}}^{0,5}(N_1) =
\Hom_{\cA^{\tmf}}(N_1, \Sigma^5\Z/3) = 0$.
\end{proof}
\begin{rem}
\label{hidden_h0}
Like in \cref{hidden_alpha}, the long exact sequence does not fully specify the $\Ext(\Z/3)$-action on $\Ext(N_1)$.
One can show that $h_0\cdot\alpha z = \pm c_4x$, but this is missed by our long exact sequence calculation. We do not need this relation in our proof of \cref{heterotic_at_3}, so we do not
prove it; one way to see $h_0\cdot\alpha z = \pm c_4x$ would be to deduce it from the analogous $h_0$-action in
$\Ext(N_2)$ via the long exact sequence in Ext induced from~\eqref{2qn_SES}. To see the corresponding $h_0$-action
in $\Ext(N_2)$, let $N_3$ be a nonsplit $\cA^{\tmf}$-module extension of $C\nu$ by $\Sigma^8\Z/3$; this
characterizes $N_3$ up to isomorphism. Then there is a short exact sequence $\Sigma^9\Z/3\to N_2\to N_3$, and the $h_0$-action we want to detect is visible to the corresponding long exact sequence in Ext.
\end{rem}
Thus we have $\Ext(N_1)$ in \cref{qn_fig}, bottom right.

\begin{figure}[h!]
\begin{subfigure}[c]{0.4\textwidth}
\begin{tikzpicture}[scale=0.5]
        \threebock(0, 2);
        \begin{scope}[Cerulean]
                \foreach \x in {-5, 0} {
                        \tikzpt{\x}{2.5}{}{};
                }
                \draw[->, thick] (-4.5, 2.5) -- (-0.5, 2.5);
        \end{scope}
        \begin{scope}[PineGreen]
                \foreach \x in {0, 5} {
                        \tikzpt{\x}{0}{}{};
                        \tikzpt{\x}{2}{}{};
                        \PoneL(\x, 0);
                }
                \draw[->, thick] (0.5, 0) -- (4.5, 0);
        \end{scope}
        \node[below=2pt] at (-5, 0) {$\Sigma^5\Z/3$};
        \node[below=2pt] at (0, 0) {$N_1$};
        \node[below=2pt] at (5, 0) {$C\nu$};
\end{tikzpicture}
\end{subfigure}

\begin{subfigure}[c]{0.48\textwidth}
\includegraphics{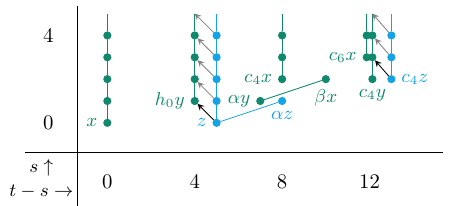}
\end{subfigure}
\begin{subfigure}[c]{0.48\textwidth}
\includegraphics{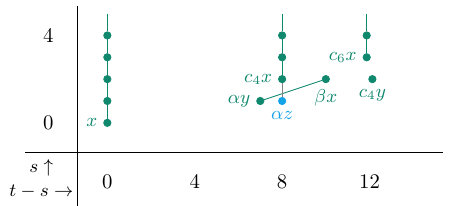}
\end{subfigure}

\caption{Top: the short exact sequence~\eqref{qn_SES} of $\cA^{\tmf}$-modules. Lower left: the induced long exact
sequence in Ext. We compute the pictured boundary maps in \cref{qn_bdry}. Lower right: $\Ext_{\cA^{\tmf}}(N_1)$ as
computed by the long exact sequence. The gray line joining $\alpha z$ and $c_4x$ indicates a
nonzero $h_0$-action not visible to this computation; see \cref{hidden_h0}.}
\label{qn_fig}
\end{figure}

The last long exact sequence we have to run is the one induced by~\eqref{2qn_SES}. We keep the notation for
elements of $\Ext(\textcolor{RawSienna}{N_1})$ from above --- classes in $\Ext(\Z/3)$ times $x$, $y$, or $z$. We let
$w$ denote a generator of $\Ext(\textcolor{RedOrange}{\Z/3})$ as an $\Ext(\Z/3)$-module; like before, the two
generators are $w$ and $-w$.
\begin{lem}
\label{2qn_bdry}
In the long exact sequence in Ext associated to~\eqref{2qn_SES}, the boundary map takes the values $\partial(x) =
\pm \alpha w$, $\partial(\alpha y) = \pm \beta w$, and $\partial(\beta x) = \pm \alpha\beta w$, and vanishes on all
other classes in degrees $14$ and below (except for $-c$ where $c$ was a class already listed).
\end{lem}
We draw this in \cref{2qn_fig}, bottom left.
\begin{proof}
As in \cref{nu_bdry,qn_bdry}, apart from $\partial(\pm x)$, $\partial(\pm \alpha y)$, and $\partial(\pm \beta x)$,
the boundary map vanishes for degree reasons, and we infer $\partial(x) = \pm \alpha w$ because this is the only
way for $\Ext^{0,4}(N_2) = \Hom(N_2, \Sigma^4 \Z/3)$ to vanish. And since $\alpha(\alpha y) = \beta x$, as we
discussed in \cref{hidden_alpha}, it remains only to prove $\partial(\alpha y) = \pm\beta w$; then $\partial(\beta
x) = \alpha\beta w$ follows from $\Ext(\Z/3)$-linearity; and since
$\Ext_{\cA^{\tmf}}^{2,12}(\textcolor{RedOrange}{\Z/3})$ is one-dimensional, to show $\partial(\alpha y) = \pm \beta
w$ it suffices to show $\partial(\alpha y)$ is nonzero.

To compute $\partial(\alpha y)$, we use the characterization of
$\Ext_{\cA^\tmf}^{1, t}(M, N)$ as a set of equivalence classes of $\cA^{\tmf}$-module extensions $0\to \Sigma^t
N\to L\to M\to 0$. We will represent $\alpha y$ as an explicit extension of $\Sigma^4 N_1$ by $\Sigma^{12}\Z/3$ and
then show this extension cannot be the pullback of an extension of $N_2$ by $\Sigma^{12}\Z/3$, which implies
$\partial(\alpha y)\ne 0$ by exactness. Up to isomorphism, there is only one non-split extension of $\Sigma^4N_1$
by $\Sigma^{12}\Z/3$, with $\alpha y$ and $-\alpha y$ distinguished by a sign in the extension maps; we draw this
extension in \cref{alpha_y_explicit}, left. In \cref{alpha_y_explicit}, right, we illustrate what goes wrong if we
try to obtain this extension as the pullback of an extension of $N_2$: the relation $(\cP^1)^3 = 0$ in $\cA^{\tmf}$
is violated. Thus $\partial(\alpha y)\ne 0$.
\end{proof}

\begin{figure}[h!]
\begin{subfigure}[c]{0.45\textwidth}
\begin{tikzpicture}[scale=0.5]
	\PoneL(0, 4);
	\begin{scope}[Fuchsia!50!white]
		\foreach \x in {-5, 0} {
			\tikzpt{\x}{6}{}{};
		}
		\draw[->, thick] (-4.5, 6) -- (-0.5, 6);
	\end{scope}
	\begin{scope}[Fuchsia!75!black]
		\foreach \x in {0, 5} {
			\tikzpt{\x}{2}{}{};
			\tikzpt{\x}{4}{}{};
			\tikzpt{\x}{4.5}{}{};
			\PoneL(\x, 2);
			\threebock(\x, 4);
		}
		\draw[->, thick] (0.5, 2) -- (4.5, 2);
	\end{scope}
	\node[below=2pt] at (-5, 0) {$\Sigma^{12}\Z/3$};
	\node[below=2pt] at (5, 0) {$\Sigma^4 N_1$};
\end{tikzpicture}
\end{subfigure}
\begin{subfigure}[c]{0.45\textwidth}
\begin{tikzpicture}[scale=0.5]
	\PoneL(0, 4);
	\begin{scope}[Fuchsia!50!white]
		\foreach \x in {-5, 0} {
			\tikzpt{\x}{6}{}{};
		}
		\draw[->, thick] (-4.5, 6) -- (-0.5, 6);
	\end{scope}
	\begin{scope}[Fuchsia!75!black]
		\foreach \x in {0, 5} {
			\tikzpt{\x}{0}{}{};
			\tikzpt{\x}{2}{}{};
			\tikzpt{\x}{4}{}{};
			\tikzpt{\x}{4.5}{}{};
			\PoneL(\x, 2);
			\threebock(\x, 4);
		}
		\draw[->, thick] (0.5, 2) -- (4.5, 2);
		\begin{scope}[dashed]
			\PoneL(0, 0);
		\end{scope}
		\PoneL(5, 0);
	\end{scope}
	\node[below=2pt] at (-5, 0) {$\Sigma^{12}\Z/3$};
	\node[below=2pt] at (5, 0) {$N_2$};
\end{tikzpicture}
\end{subfigure}

\caption{Left: an extension of $\cA(1)$-modules representing the class $\alpha
y\in\Ext_{\cA^{\tmf}}^{1,12}(\Sigma^4 N_1)$. Right: if we try to form an analogous extension of $N_2$, we are
obstructed by the fact that $(\cP^1)^3 = 0$ in $\cA^{\tmf}$. This is part of the proof of \cref{2qn_bdry}.}
\label{alpha_y_explicit}
\end{figure}
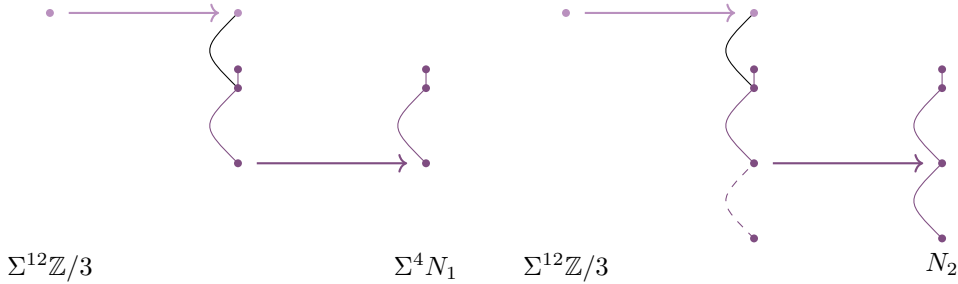

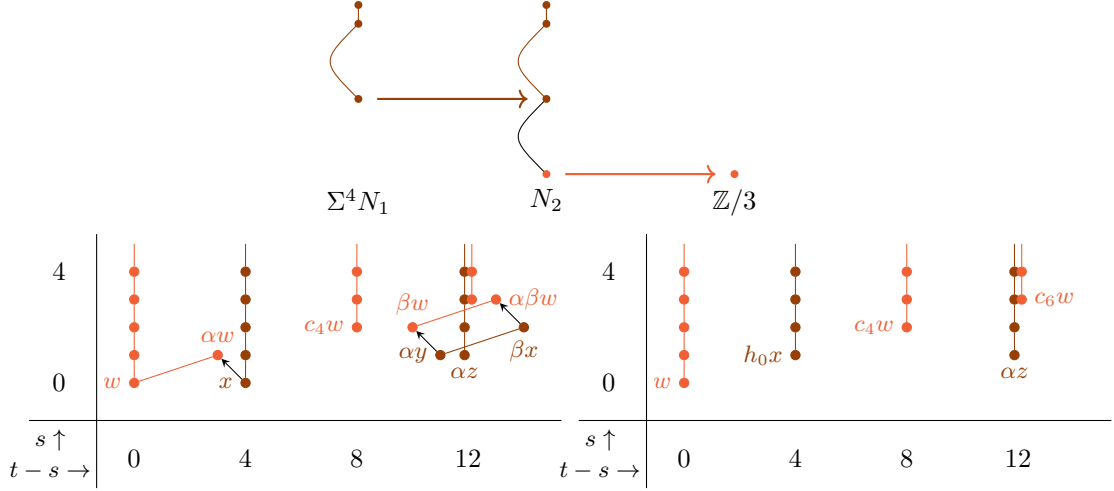
\begin{figure}[h!]
\begin{subfigure}[c]{0.4\textwidth}
\begin{tikzpicture}[scale=0.5]
        \PoneL(0, 0);
        \begin{scope}[RawSienna]
                \foreach \x in {-5, 0} {
                        \tikzpt{\x}{2}{}{};
                        \tikzpt{\x}{4}{}{};
                        \tikzpt{\x}{4.5}{}{};
                        \PoneL(\x, 2);
                        \threebock(\x, 4);
                }
                \draw[->, thick] (-4.5, 2) -- (-0.5, 2);
        \end{scope}
        \begin{scope}[RedOrange]
                \foreach \x in {0, 5} {
                        \tikzpt{\x}{0}{}{};
                }
                \draw[->, thick] (0.5, 0) -- (4.5, 0);
        \end{scope}
        \node[below=2pt] at (-5, 0) {$\Sigma^4 N_1$};
        \node[below=2pt] at (0, 0) {$N_2$};
        \node[below=2pt] at (5, 0) {$\Z/3$};
\end{tikzpicture}
\end{subfigure}

\begin{subfigure}[c]{0.48\textwidth}
\includegraphics{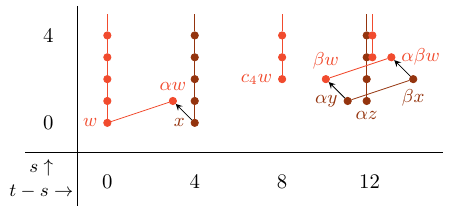}
\end{subfigure}
\begin{subfigure}[c]{0.48\textwidth}
\includegraphics{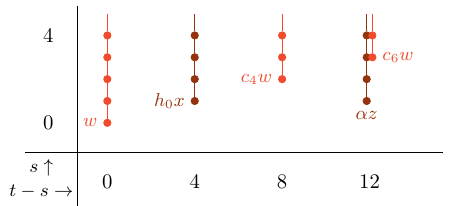}
\end{subfigure}
\caption{Top: the short exact sequence~\eqref{2qn_SES} of $\cA^{\tmf}$-modules. Lower left: the induced long exact
sequence in Ext. We compute the boundary maps in \cref{2qn_bdry}. Lower right: $\Ext_{\cA^{\tmf}}(N_2)$ as computed by the long exact sequence.}
\label{2qn_fig}
\end{figure}

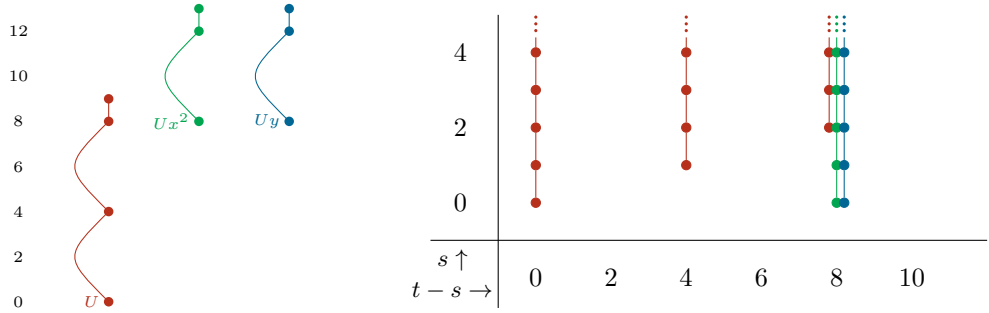
\begin{figure}[h!]
\begin{subfigure}[c]{0.3\textwidth}
        \begin{tikzpicture}[scale=0.6, every node/.style = {font=\tiny}]
                \foreach \y in {0, 2, ..., 12} {
                        \node at (-2, \y/2) {$\y$};
                }
                \begin{scope}[BrickRed]
                        \tikzpt{0}{0}{$U$}{};
                        \tikzpt{0}{2}{}{};
                        \tikzpt{0}{4}{}{};
                        \tikzpt{0}{4.5}{}{};
                        \PoneL(0, 0);
                        \PoneL(0, 2);
                        \threebock(0, 4);
                \end{scope}
                \begin{scope}[Green]
                        \tikzpt{2}{4}{$Ux^2$}{};
                        \tikzpt{2}{6}{}{};
                        \tikzpt{2}{6.5}{}{};
                        \PoneL(2, 4);
                        \threebock(2, 6);
                \end{scope}
                \begin{scope}[MidnightBlue]
                        \tikzpt{4}{4}{$Uy$}{};
                        \tikzpt{4}{6}{}{};
                        \tikzpt{4}{6.5}{}{};
                        \PoneL(4, 4);
                        \threebock(4, 6);
                \end{scope}
        \end{tikzpicture}
\end{subfigure}
\qquad
\begin{subfigure}[h!]{0.6\textwidth}
\includegraphics{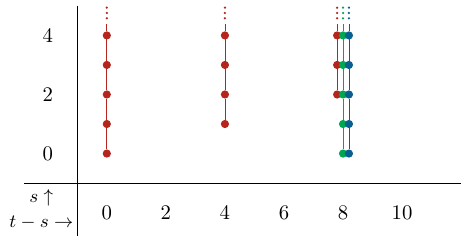}
\end{subfigure}
\caption{Left: the $\cA^{\tmf}$-module structure on $H_{\tmf}^*(M)$ in low degrees; the pictured submodule contains
all elements in degrees $11$ and below. Right: the $E_2$-page of the Adams spectral sequence computing
$\pi_*(M)_3^\wedge$, which as we discuss in the proof of \cref{heterotic_at_3} is isomorphic to the $3$-completion
of the twisted string bordism groups relevant for $E_8\times E_8$ heterotic string theory.}
\label{twisted_tmf_Adams}
\end{figure}

Now that we know the Ext groups of all $\cA^{\tmf}$-modules appearing in~\eqref{tmf_decomp}, we can draw the
$E_2$-page of the Adams spectral sequence computing $\pi_*(M^\tmf f_{0,c_1+c_2})_3^\wedge$ in
\cref{twisted_tmf_Adams}, right (and hence, as noted above, the corresponding twisted string bordism groups in degrees $15$ and below). For degree reasons, this spectral sequence collapses at $E_2$ in degrees $t-s\le
11$; since $h_0$-actions lift to multiplication by $3$, there is no $3$-torsion in this range, and we conclude.
\end{proof}
\begin{rem}
Other examples of twisted string structures appear in the math and physics literature; see
Dierigl-Oehlmann-Schimmanek~\cite[\S 3.4]{DOS23} for another $3$-primary example.
\end{rem}
\begin{rem}
Just as in \cref{stabdiff}, Kreck's modified surgery gives a classification of some closed, smooth $8$-manifolds up
to stable diffeomorphism in terms of twisted string bordism. There is work applying this in examples corresponding
to vector bundle twists~\cite{FK96, Fan99, FW10, WW12, CN20}; it would be interesting to apply the $\tmf$-module Adams
spectral sequence to classes of manifolds where the twist is not given by a vector bundle.
\end{rem}

\subsection{$H\Z/2$ as a $\ku$-module Thom spectrum}\label{sub:kumod}
Devalapurkar uses methods from chromatic homotopy theory to prove the following result. We will reprove it using
the tools in this paper.
\begin{thm}[{Devalapurkar~\cite[Remark 2.3.16]{Dev}}]
\label{sanath_thm}
There is a map of $E_1$-spaces $f\colon \U_2\to B\GL_1(\ku)$ and a $2$-local equivalence of $E_1$-ring spectra
$Mf\simeq H\Z/2$.
\end{thm}
We will prove \cref{sanath_thm} in a sequence of steps. First, we establish an additive equivalence:
\begin{prop}
\label{additive_sanath_thm}
There is a map $f\colon \U_2\to B\GL_1(\ku)$ and a $2$-local equivalence $Mf\simeq H\Z/2$.
\end{prop}
\begin{thm}[{Borel~\cite[Théorèmes 8.2 et 8.3]{Bor54}}]
\label{borel_unitary}
Let $A$ be $\Z$ or $\Z/2$.
\begin{enumerate}
    \item $H^*(B\U_n; A)\cong A[c_1,\dotsc,c_n]$, with $\abs{c_i} = 2i$.
    \item\label{borel2} $H^*(\U_n; A)\cong \Lambda_A(b_1,\dotsc,b_n)$ with $\abs{b_i} = 2i-1$.
    \item The same is true with $\SU_n$ in place of $\U_n$, except that we leave out $c_1$.
\end{enumerate}
The inclusion maps $\U_{n-1}\hookrightarrow\U_n$ and $B\U_{n-1}\to B\U_n$ send $b_i\mapsto b_i$, resp.\ $c_i\mapsto c_i$ (and likewise with $\SU$ in place of $\U$). Moreover, these statements are true for $n = \infty$ (so for $B\U$, $\U$, $B\SU$, and $\SU$).
\end{thm}
Here $\Lambda_A(\dots)$ denotes an exterior $A$-algebra on the specified generators. Below, we will use $b_i$ and $c_i$ to denote the $\Z$-cohomology classes and $\overline b_i$ and $\overline c_i$ to denote the $\Z/2$-cohomology classes.


\begin{proof}[Proof of \cref{additive_sanath_thm}]
Let $f\colon \U_2\to B\GL_1(\ku)$ be the fake vector bundle twist given by $(\overline b_1, b_3)$ (see \S\ref{spinc_twists} for the definition of this class of twists). Borel's theorems that we cited in \cref{borel_unitary} can be used to show that $\cA$ acts trivially on $H^*(\U_2;\Z/2)$.\footnote{In fact, Miller~\cite{Mil80} showed that the triviality of $H^*(\U_2;\Z/2)$ as an $\cA$-algebra lifts to a wedge sum decomposition of $\Sigma^\infty\U_2$ itself; see also~\cite{Jam59, Cra87}.}
%
\Cref{thom_module_calc} shows that $H_\ku^*(Mf)$ is isomorphic to $H^*(\U_2;\Z/2)$ as $\Z/2$-vector spaces, and
that the $\cE(1)$-action is twisted by $Q_0(U) = U\overline b_1$ and $Q_1(U) = U\overline b_3$. This and the Cartan
rule imply $H_\ku^*(Mf)\cong\cE(1)$ as $\cE(1)$-modules, so $\Ext_{\cE(1)}(H_\ku^*(Mf), \Z/2)$ consists of a single
$\Z/2$ in bidegree $(0, 0)$ and vanishes elsewhere. Thus the $\ku$-module Adams spectral sequence immediately
collapses, and we learn $\pi_0(Mf)_2^\wedge\cong\Z/2$ and all other homotopy groups vanish. This property
characterizes $H\Z/2$ up to $2$-local equivalence (e.g.\ it implies $H^0(Mf;\Z/2)\cong\Z/2$, giving a map $Mf\to
H\Z/2$ which is an isomorphism on $2$-completed homotopy groups, allowing us to conclude by Whitehead).
\end{proof}
The rest of the proof is:
\begin{prop}
\label{is_loop_map}
There is a map $F\colon B\U_2\to B(B\O/B\Spin^c)$ such that $\Omega F\simeq f$.
\end{prop}
Before we prove \cref{is_loop_map} we must identify the space $B(B\O/B\Spin^c)$. Recall the space $\SK(4)$ from~\eqref{sk_defn}, which represents $\SH^4$ (degree-$4$ supercohomology, defined in \cref{supercoh_defn}).
\begin{prop}
\label{is_SK}
There is a homotopy equivalence $B(B\O/B\Spin^c)\overset\simeq\to \SK(4)$ of spaces.
\end{prop}
\begin{proof}
Since $B\O/B\Spin^c$ is an abelian $\infty$-group, $B(B\O/B\Spin^c)\simeq \Sigma (B\O/B\Spin^c)$. Thus, to obtain the homotopy groups of $B(B\O/B\Spin^c)$, we shift up the homotopy groups of $B\O/B\Spin^c$ that we obtained from \cref{splitting_spinc}. Thus $B(B\O/B\Spin^c)$ has two
nonzero homotopy groups, $\pi_2\cong\Z/2$ and $\pi_4\cong\Z$. By definition, $\SK(4)$ also has $\pi_2\cong\Z/2$ and $\pi_4\cong\Z$, so to establish that $B(B\O/B\Spin^c)\simeq \SK(4)$, it suffices to show their $k$-invariants are
equal. In the text around~\eqref{sk_defn}, we chose the $k$-invariant of $\SK(4)$ to be $\beta\circ\Sq^2$, where $\beta\colon
H^*(\bl;\Z/2)\to H^{*+1}(\bl;\Z)$ is the Bockstein, and by~\cite[Corollary 4.9]{BLM23}, the $k$-invariant of $B\O/B\Spin^c$ is also $\beta\circ\Sq^2$.\footnote{Beardsley-Luecke-Morava phrase their results in terms of the \term{Picard spectrum} $\mathrm{Pic}(\KU)$; the relation to $B\O/B\Spin^c$ appears in (\textit{ibid.}, \S 5.2) for twisted spin and string structures, and the story for twisted \spinc structures is analogous.} 
\end{proof}
By applying the loop space functor and \cref{SK_loop}, we also get:
\begin{cor}
\label{SH3}
There is a homotopy equivalence of spaces $B\O/B\Spin^c\simeq \SK(3)$.
\end{cor}
\begin{rem}
\label{abelian_loop}
Using the equivalence of $\infty$-categories between infinite loop spaces and connective spectra, one can prove that on the sub-$\infty$-category of \emph{connected} abelian $\infty$-groups, the functor $\Sigma\Omega$ is naturally isomorphic to the identity. Thus if $f\colon X\to Y$ is a map of connected abelian $\infty$-groups, $\Omega f\colon \Omega X\to\Omega Y$ is the unique homotopy class of maps whose suspension is $f$.
\end{rem}
\begin{lem}
\label{b1_loop}
Regard $\overline b_1\in H^1(\U_2;\Z/2)$ as a map $\overline b_1\colon \U_2\to K(\Z/2, 1)$, and likewise for $\overline c_1\colon B\U_2\to K(\Z/2, 2)$. Then $\Omega \overline c_1 \simeq\overline b_1$.
\end{lem}
\begin{proof}
By \cref{borel_unitary}, the pullback maps $H^1(\U_2;\Z/2)\to H^1(\U_1;\Z/2)$ and $H^2(B\U_2;\Z/2)\to H^2(B\U_1;\Z/2)$ are isomorphisms, so it suffices to prove this result with $\U_2$ replaced with $\U_1$. The map $\overline c_1\colon B\U_1\to K(\Z/2, 2)$ is a map of abelian $\infty$-groups (heuristically, the characteristic class $\overline c_1$ is additive in tensor products of line bundles), which implies the lemma by \cref{abelian_loop}.
\end{proof}
\begin{lem}
\label{b3_loop}
Regard $b_3\in H^3(\U_2;\Z/2)$ as a map $b_3\colon \U_2\to K(\Z, 3)$, and likewise for $c_2\colon B\U_2\to K(\Z, 4)$. Then there is some $\lambda\in\Z$ such that $\Omega (\pm c_2+\lambda c_1^2) \simeq b_3$.
\end{lem}
\begin{proof}
Let $i\colon \U_2\hookrightarrow \U$ and $j\colon\SU\hookrightarrow\U$ be the usual inclusions. Then we have a commutative diagram
\begin{equation}
\label{three_loop_diagram}
\begin{tikzcd}
	{H^4(B\U_2;\Z)} & {H^4(B\U;\Z)} & {H^4(B\SU;\Z)} \\
	{H^3(\U_2;\Z)} & {H^3(\U;\Z)} & {H^3(\SU;\Z),}
	\arrow["\Omega"', from=1-1, to=2-1]
	\arrow["{(Bi)^*}"', from=1-2, to=1-1]
	\arrow["{(Bj)^*}", from=1-2, to=1-3]
	\arrow["\Omega"', from=1-2, to=2-2]
	\arrow["\Omega"', from=1-3, to=2-3]
	\arrow["{i^*}"', from=2-2, to=2-1]
	\arrow["{j^*}", from=2-2, to=2-3]
\end{tikzcd}
\end{equation}
and by \cref{borel_unitary}, $i^*$ and $(Bi)^*$ are isomorphisms and $j^*$ and $(Bj)^*$ are surjective. Specifically, we learn that if $x\in H^4(B\U;\Z)$ is such that $(Bj)^*(x) = c_2$, then $x = c_2 + \lambda c_1^2$ for some $\lambda\in\Z$. Passing through the isomorphisms $i^*$ and $(Bi)^*$, we have the analogous fact for $B\U_2$ in place of $B\U$.

Since $B\SU$ has the direct sum abelian $\infty$-group structure, the Whitney sum formula shows that $c_2\colon B\SU\to K(\Z, 4)$ is a morphism of connected $\infty$-groups. Alternatively, one may identify this map with the cofiber of the forgetful map $B\U\ang 6\to B\SU$, which is a map of abelian $\infty$-groups (see also the text around~\eqref{OSpinc_SES}).

By \cref{abelian_loop}, $c_2\colon B\SU\to K(\Z, 4)$ loops to a generator of $H^3(\SU;\Z)$, which must be $\pm b_3$. Chase this fact across~\eqref{three_loop_diagram} to finish the proof.
\end{proof}

\begin{proof}[Proof of \cref{is_loop_map}]
We claim that here is a commutative diagram of long exact sequences
\begin{equation}
\label{loop_SH_LES}
\begin{tikzcd}
	\dotsb & {H^4(B\U_2;\Z)} & {\SH^4(B\U_2)} & {H^2(B\U_2;\Z/2)} & {H^5(B\U_2;\Z)} & \dotsb \\
	\dotsb & {H^3(\U_2;\Z)} & {\SH^3(\U_2)} & {H^1(\U_2;\Z/2)} & {H^4(\U_2;\Z)} & \dotsb
	\arrow["{\beta\circ\Sq^2}", from=1-1, to=1-2]
	\arrow[from=1-2, to=1-3]
	\arrow["\Omega"', from=1-2, to=2-2]
	\arrow["t", from=1-3, to=1-4]
	\arrow["\Omega"', from=1-3, to=2-3]
	\arrow["{\beta\circ\Sq^2}", from=1-4, to=1-5]
	\arrow["\Omega"', from=1-4, to=2-4]
	\arrow[from=1-5, to=1-6]
	\arrow["\Omega"', from=1-5, to=2-5]
	\arrow["{\beta\circ\Sq^2}", from=2-1, to=2-2]
	\arrow[from=2-2, to=2-3]
	\arrow["t", from=2-3, to=2-4]
	\arrow["{\beta\circ\Sq^2}", from=2-4, to=2-5]
	\arrow[from=2-5, to=2-6]
\end{tikzcd}\end{equation}
where the vertical arrows are the loop space functor. Specifically, the interpretation of the loop space functor as a map $\Omega\colon H^n(X; A)\to H^{n-1}(\Omega X; A)$ is just as in \cref{b1_loop,b3_loop}; the interpretation on supercohomology is completely analogous, using that $\Omega\SK(n)\simeq\SK(n-1)$ (\cref{SK_loop}). 

The commutative diagram in~\eqref{loop_SH_LES} exists essentially because the loop space functor, applied to a cofiber sequence of connected infinite loop spaces, returns a cofiber sequence of infinite loop spaces.

We claim that all four maps labeled $\beta\circ\Sq^2$ in~\eqref{loop_SH_LES} vanish. For all of them except $\beta\circ\Sq^2\colon H^2(B\U_2;\Z/2)\to H^5(B\U_2;\Z)$, this follows because $\Sq^2$ vanishes on classes in degrees less than $2$. To see that the remaining $\beta\circ\Sq^2$ vanishes, check on the generator $\overline c_1\coloneqq c_1\bmod 2$ (\cref{borel_unitary}): $\Sq^2(\overline c_1) = \overline c_1^2$ for degree reasons, but $\overline c_1^2 = c_1^2\bmod 2$, so $\beta(\overline c_2^2) = 0$. Thus~\eqref{loop_SH_LES} simplifies to a map of short exact sequences:
\begin{equation}\begin{tikzcd}
	0 & {H^4(B\U_2;\Z)} & {\SH^4(B\U_2)} & {H^2(B\U_2;\Z/2)} & 0 \\
	0 & {H^3(\U_2;\Z)} & {\SH^3(\U_2)} & {H^1(\U_2;\Z/2)} & 0
	\arrow[from=1-1, to=1-2]
	\arrow[from=1-2, to=1-3]
	\arrow["\Omega_1"', from=1-2, to=2-2]
	\arrow["t", from=1-3, to=1-4]
	\arrow["\Omega_2"', from=1-3, to=2-3]
	\arrow[from=1-4, to=1-5]
	\arrow["\Omega_3"', from=1-4, to=2-4]
	\arrow[from=2-1, to=2-2]
	\arrow[from=2-2, to=2-3]
	\arrow["t", from=2-3, to=2-4]
	\arrow[from=2-4, to=2-5]
\end{tikzcd}\end{equation}
Here we give the loop space functor maps different names $\Omega_i$ to distinguish them. By \cref{b1_loop}, $\Omega_3$ is an isomorphism, and by \cref{b3_loop}, $\Omega_1$ is surjective (since the generator of $H^3(\U_2;\Z)\cong\Z$ is in the image of $\Omega_1$). Therefore by the four lemma, $\Omega_2$ is surjective. Reinterpreting this fact as in \cref{is_SK,SH3}, we have that the map
\begin{equation}
    \Omega\colon [B\U_2, B(B\O/B\Spin^c)] \longrightarrow [\U_2, B\O/B\Spin^c]
\end{equation}
is surjective, which suffices to prove the proposition.
%
%
%
\end{proof}
\begin{proof}[Proof of \cref{sanath_thm}]
By \cref{cofiber_twists}, the map $T\colon B\O/B\Spin^c\to B\GL_1(\ku)$ is a map of abelian $\infty$-groups, and by the
recognition principle, since $f\simeq\Omega F$ by \cref{is_loop_map}, $f$ is a map of $E_1$-spaces. Thus
$T\circ f$ is also $E_1$, which by~\cite[Theorem 1.7]{ABG11} implies that its Thom spectrum is an
$E_1$-$\ku$-algebra. We identified this Thom spectrum as $H\Z/2$, which has a unique $E_1$-ring structure, in \cref{additive_sanath_thm}.
\end{proof}
\newcommand{\etalchar}[1]{$^{#1}$}

\end{document}